\documentclass{preprint}
\usepackage[utf8]{inputenc}
\usepackage{amsmath}
\usepackage{graphicx}
\usepackage{amssymb}
\usepackage{mathtools}
\usepackage{enumerate}
\usepackage[dvipsnames]{xcolor}

\usepackage{graphicx}
\usepackage{placeins}
\usepackage{amssymb}
\usepackage[utf8]{inputenc}
\usepackage{esint}

\usepackage{pgfplots}
\usepackage{pgfplotstable}
\usepackage{tikz}
\usepackage{tkz-euclide}
\usepgfplotslibrary{external,statistics,groupplots}
\graphicspath{{figures/}}

\journal{\phantom{.}}

\newtheorem{theorem}{Theorem}
\newtheorem{lemma}{Lemma}
\newtheorem{remark}{Remark}
\newtheorem{assumption}{Assumption}

\usepackage{custom-commands-private}

\title{Reduced Lagrange multiplier approach for non-matching coupling of mixed-dimensional domains}

\author[sissa]{Luca Heltai}

\author[polimi]{Paolo Zunino}

\address[sissa]{International School for Advanced Studies, Via Bonomea 265, 34136, Trieste, TS, Italy}

\address[polimi]{MOX, Department of Mathematics, Politecnico di Milano, Piazza Leonardo da Vinci 32, 20133 Milano, MI, Italy}

\begin{document}

\begin{frontmatter}

\begin{abstract}
Many physical problems involving heterogeneous spatial scales, such as the flow through fractured porous media, the study of fiber-reinforced materials or the modeling of the small circulation in living tissues -- just to mention a few examples -- can be described as coupled partial differential equations defined in domains of heterogeneous dimensions that are embedded into each other. This formulation is a consequence of geometric model reduction techniques that transform the original problems defined in complex three-dimensional domains into more tractable ones. The definition and the approximation of coupling operators suitable for this class of problems is still a challenge. The main objective of this work is to develop a general mathematical framework for the analysis and the approximation of partial differential equations coupled by non-matching constraints across different dimensions. Considering the non standard formulation of the coupling conditions, we focus on their enforcement using Lagrange multipliers. In this context we address in abstract and general terms the well posedness, the stability, the robustness of the problem with respect to the smallest characteristic length of the embedded domain. We also address the  the numerical approximation of the problem and we discuss the \emph{inf-sup} stability of the proposed numerical scheme for some representative configuration of the embedded domain. 
The main message of this work is twofold: from the standpoint of the theory of mixed-dimensional problems, we provide general and abstract mathematical tools to formulate coupled problems across dimensions. From the practical standpoint of the numerical approximation, we show the interplay between the mesh characteristic size, the dimension of the Lagrange multiplier space and the size of the inclusion in representative configurations interesting for applications. The latter analysis is complemented with illustrative numerical examples.
\end{abstract}

\begin{keyword}
mixed-dimensional problems \sep non-matching coupling \sep Lagrange multipliers \sep model reduction \sep numerical approximation
\end{keyword}

\end{frontmatter}

\section{Introduction}

The definition, analysis and approximation of boundary value problems governed by partial differential equations with incomplete constrains at the boundary is a relevant topic in computational fluid dynamics\cite{HRT96}. The so called \emph{do nothing} conditions were introduced to accommodate incomplete boundary data. These techniques were applied to \emph{geometric multiscale problems}, where domains with different dimensionality are coupled together by means of suitable interface conditions\cite{Formaggia2001561}. Since problems of higher dimensionality (PDEs in two or three dimensions) are supplied with interface data of low dimensionality (one dimensional or lumped parameter models), the incomplete information must be supplemented by some modeling assumptions.
These issues become more challenging if we consider the coupling of PDEs defined in domains of  heterogeneous dimensions that are embedded into each other. These models represent many real physical problems, such as the flow through fractured porous media, the study of fiber-reinforced materials or the modeling of the small circulation (microcirculation) in biological tissues (a representative geometrical configuration of such problems is shown in Figure \ref{fig:domain}, left panel). In these examples, physical models defined in inclusions characterized by a small dimension can be approximated using models of lower dimensionality, giving rise to coupled problems across multiple spatial scales. As a result, the information that is transferred at the interface can not match, because of the dimensionality gap. We describe these cases as non-matching mixed-dimensional coupled problems.

The main objective of this work is to develop a general mathematical framework
for the analysis and the approximation of PDEs coupled by non-matching
constraints across different dimensions. Considering the non standard
formulation of the coupling conditions, we focus on their enforcement using
Lagrange multipliers (LM). The enforcement and approximation of
boundary/interface conditions using LM is a central topic in the development of
the finite element method\cite{Babuska1973179,Bramble,Pitkaranta,Boffi2021} among many
others. More recently, the LM method has been applied to couple PDEs across
interfaces\cite{BURMAN20102680,Melenk} just to mention a few examples of a broad
field in the literature.

The novelty of this work with respect to such literature consists in the use of
the LM method to couple equations defined on domains with heterogeneous dimensions.
For this reason, an essential aspect of the work is to shed light on the
interaction between the LM formulation and the restriction/extension operators
that govern the transition of PDEs across spatial dimensions. We work in the
abstract framework of saddle point problems in Hilbert spaces and their
approximation through Mixed Finite Elements\cite{Boffi2013}. An abstract
framework based on exterior calculus has recently appeared for the formulation
and the approximation of mixed-dimensional (coupled)
problems\cite{Boon2021757,Boon2022}. We will explore the intersection of this
work with ours in future works.

We consider here three main aspects of the problem. First, we analyse under what conditions the stability of the LM formulation is preserved after the application of the dimensional restriction operator. In other words, we will study how PDEs coupled by LM behave when a mixed-dimensional formulation is adopted. Second, we focus on the gap between physically relevant quantities at the interface. At the level of the continuous problem formulation, we introduce an approximation parameter, $N$, that affects the richness of the LM space that matches heterogeneous dimensions (being $N=1$ the fully nonmatching scenario and $N\rightarrow\infty$ the perfectly matched case) and we study how it influences the satisfaction of the boundary constraints, or equivalently the magnitude of the gap between interface unknowns. Third, we analyze the difference between the full-scale formulation and the mixed-dimensional formulation, putting into evidence how the so called \emph{dimensionality reduction error} varies with the characteristic spatial dimension of the problem and with the approximation parameter, $N$. 

With this work, we aim at shedding light on the common mathematical framework that embraces many recent works involving the applications mentioned above. LM formulations for Dirichlet-Neumann type interface conditions for these problems were recently proposed\cite{AlzettaHeltai-2020-a,HeltaiCaiazzo-2019-a,HeltaiCaiazzoMueller-2021-a}. In these works, a three dimensional bulk problem for mechanical deformations is coupled to a one dimensional model for the mechanical behavior of fibers and vessels, respectively. Also, the complementary problem made of thin and slender mechanical structures immersed into a fluid or solid continuum is particularly relevant\cite{Steinbrecher20201377,Hagmeyer2022} and could be addressed with the proposed tools. {\color{black}A preliminary mathematical study of these problems was recently developed\cite{Kuchta2021558,Kuchta2019375,Kuchta2016B962,Laurino2019,boulakia:hal-03501521}.}

After considering the classical Dirichlet-Neumann interface conditions, we also address Robin type transmission conditions. This variant of the problem formulation is particularly significant for multiscale mass transport and fluid mechanics problems\cite{DAngelo2008,Cattaneo20141347,Kppl2018,Koch2020,Possenti20213356}. Here, we show how a prototype of these applications can be formulated and analyzed in the framework of perturbed saddle point problems.

We organize the work as follows. In section \ref{sec:lagrange} we introduce the problem formulation for Dirichlet transmission conditions and we recall some fundamental assumptions and results. The \emph{reduced} Lagrange multiplier formulation of the problem is presented and analyzed in section \ref{sec:reduced}, where we state the well posedness of the problem. The extension to Robin type transmission conditions is also addressed here. We illustrate the behavior of the \emph{dimensionality reduction error} in Section~\ref{sec:dimensionality-reduction-error}, and in 
Section \ref{sec:isomorphism} we discuss the properties of the restriction and extension operators when the small inclusion can be described as a mapping of a reference domain by means of an isomorphism. Section \ref{sec:fourier} discusses the particular but very important case of inclusions isomorphic to a cylinder, where we enforce the constraints at the boundary of the inclusion by means of the projection onto a Fourier space truncated to the $N$-th frequency. The larger $N$, the better is the satisfaction of the constraint at the boundary. We introduce the numerical approximation of the problem and its properties in section \ref{sec:approximation}, and finally in section \ref{sec:numerics} we discuss some numerical experiments in support of this theory.

\section{Model problem and Lagrange multiplier formulation}\label{sec:lagrange}

Consider the following model problem:
\begin{subequations}
    \label{eq:model-problem}
    \begin{align}
        \label{eq:model-problem-a}
         -\Delta u &= f && \text{ in } \Omega\setminus V\\
        \label{eq:model-problem-b}
         u &= 0 && \text{ on } \partial \Omega \\
        \label{eq:model-problem-c}
        u &= g && \text{ on } \Gamma := \partial V,
    \end{align}
\end{subequations}
where $V$, represented in Figure \ref{fig:domain}, is a set of (possibly disconnected) immersed \textit{inclusions} (such as sheets, vasculature networks, or macroparticles, see for example Figure \ref{fig:domain-bis}, left panel), denoted as $V := \cup_i V_i$, and whose boundary is denoted by $\Gamma := \cup_i \Gamma_i \equiv \cup_i \partial V_i = \partial V$.
\begin{figure}
    \centering
    \includegraphics[width=\textwidth]{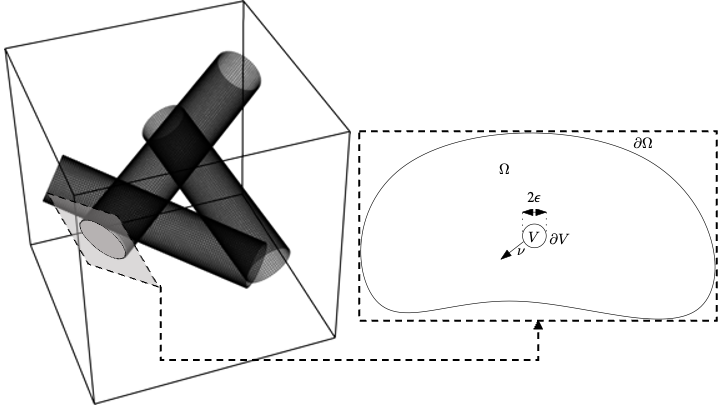}
    \caption{(left panel) Geometry of randomly placed cylindrical fibers in a three-dimensional continuum. 
    The cylinders have radius $\epsilon=0.2$ and height $\eta=0.5$ are placed randomly in a non-overlapping way and at a finite distance from the boundary of the domain.
    (right panel) Example of the two dimensional section of a single one fiber, corresponding to the domain $V$ of radius $\epsilon$, embedded in a macroscopic domain $\Omega$.}
    \label{fig:domain}
\end{figure}

To develop our approach, we consider a weak formulation of \eqref{eq:model-problem-a} that uses Lagrange multipliers to impose the given boundary condition on $\Gamma$. In particular, we start by testing \eqref{eq:model-problem-a} with smooth test functions $v$ \emph{defined on the entire domain $\Omega$}, and that are zero on $\partial\Omega$
\begin{equation*}
    (\nabla u, \nabla v)_{\Omega \setminus V} + \duality{\nabla u \cdot n, v}_{\Gamma} = (f,v)_{\Omega \setminus V},
\end{equation*}
where $n$ denotes the outward normal vector to $V$, and we extend the problem to the entire domain $\Omega$ by adding to it the weak form of $-\Delta u = \tilde f$ in $V$, and impose continuity of the function $u$ at the interface $\Gamma$:
\begin{equation*}
    (\nabla u, \nabla v)_{\Omega} + \duality{(\nabla u^+-\nabla u^-) \cdot n, v}_{\Gamma} = (\tilde f,v)_{\Omega}.
\end{equation*}

Here $\tilde f \in L^2(\Omega)$ is an arbitrary extension of $f$ in the entire $\Omega$.
With a little abuse of notation, from now on we will not distinguish  between $\tilde f$ and $f$.
We denote respectively $w^+,\,w^-$ the outer and inner values of a function $w$ with respect to $\Gamma = \partial V$ along the outer normal direction $n$, and  we denote by $[w]=w^+-w^-$ the jump of $w$ across $\Gamma$. Then equations \eqref{eq:model-problem-a}-\eqref{eq:model-problem-b} are equivalent to
\begin{equation}\label{eq:model-problem-ab}
    (\nabla u, \nabla v)_{\Omega} + \duality{[\nabla u] \cdot n, v}_{\Gamma} = (f,v)_{\Omega},\ \forall v\in H^1_0(\Omega),
\end{equation}
where we need an additional condition to impose the value of $u$ on $\Gamma$.
The natural way to proceed, is to impose such boundary condition through a
Lagrange multiplier, resulting in the following weak problem: 

Given $f\in H^{-1}(\Omega)$ and $g\in H^{1/2}(\Gamma)$, find $u\in
H^1_0(\Omega),\,\lambda \in H^{-\frac12}(\Gamma)$ such that
\begin{subequations}
    \label{eq:model-problem-weak}
    \begin{align}
            \label{eq:model-problem-weak-a}
            &(\nabla u, \nabla v)_{\Omega} + \duality{\lambda, v}_{\Gamma} = (f,v)_{\Omega},\ \forall v\in H^1_0(\Omega), && \forall v \in H^1_0(\Omega)
            \\
            \label{eq:model-problem-weak-b}
            &\duality{u,q}_{\Gamma} = \duality{g,q}_{\Gamma}, && \forall q\in H^{-\frac12}(\Gamma).
    \end{align}
\end{subequations}

With appropriate conditions on the regularity of $\Gamma$, this problem admits a unique solution, and using integration by parts it is straightforward to show that
\begin{equation*}
    \lambda  = [\nabla u] \cdot n\, \ \textrm{in} \ H^{-\frac12}(\Gamma).
\end{equation*}

We use the following notations for Sobolev spaces in $\Omega$ and $\Gamma := \partial V$:
\begin{equation}
    \begin{aligned}
        \spaceV & := H^1_0(\Omega), \qquad &&\|u\|_{\spaceV} := \|u\|_{1,\Omega} \\
        \spaceQ & := H^{-\frac12}(\Gamma), \qquad &&\|\lambda\|_{\spaceQ} := \|\lambda\|_{-\frac12,\Gamma},
    \end{aligned}
\end{equation}
with which Problem \eqref{eq:model-problem-weak} can be represented in operator form as follows: 

given $f\in\spaceV',\,g\in\spaceQ'$ find $u\in\spaceV,\,\lambda\in\spaceQ$ such that
    \begin{subequations}\label{eq:model-problem-weak-abstract}
            \begin{align}
        &\duality{Au,v} + \duality{B^T \lambda, v} = \duality{f,v} && \forall v \in \spaceV
        \\
        &\duality{Bu,q} = \duality{g,q} && \forall q \in \spaceQ
    \end{align}
    \end{subequations}
where
\begin{align}
    \label{eq:defA}
    & A: \spaceV \equiv H^1_0(\Omega) \mapsto \spaceV' \equiv H^{-1}(\Omega)
    \ \mathrm{with} \ &&\duality{Au,v} = (\nabla u, \nabla v)_{\Omega}
    \\
    \label{eq:defB}
    & B: \spaceV \mapsto \spaceQ' \equiv H^{\frac12}(\Gamma)
    \ \mathrm{with} \ &&\duality{Bu,q} = \duality{u,q}_{\Gamma}
    \\
    \nonumber
    & B^T: \spaceQ \mapsto \spaceV'. &&
\end{align}

To demonstrate the well-posedness of \eqref{eq:model-problem-weak} we first address the main properties of the operators $A$ and $B$, which will be also central in the development of the reduced Lagrange multiplier approach.
\begin{theorem}[Infsup on $A$]
    \label{theo:infsup-A}
    The operator $A: \spaceV \mapsto \spaceV'$  is symmetric, and it satisfies the infsup condition, i.e., there exists a positive real number $\alpha>0$ such that
    \begin{equation}
        \infsup{0\neq u\in\spaceV}{0\neq v\in\spaceV} \frac{\duality{Au, v}}{\|u\|_{\spaceV} \|v\|_{\spaceV}} \geq \alpha > 0.
    \end{equation}
\end{theorem}
\begin{proof}
    The proof follows from the definition of $A$ and Poincarè inequality.
\end{proof}
\begin{theorem}[Infsup on $B$]
    \label{theo:infsup-B}
    The operator $B: \spaceV \mapsto \spaceQ'$ satisfies the infsup condition, i.e., there exists a positive real number $\beta_B > 0$ such that
    \begin{equation}
        \label{eq:insup-B}
        \infsup{0\neq q \in \spaceQ}{0\neq v \in \spaceV} \frac{\duality{Bv, q}}{\|u\|_\spaceV \|q\|_\spaceQ} \geq \beta_B > 0.
    \end{equation}
\end{theorem}

\begin{proof}
    The operator $B$ coincides with the trace operator. By the trace theorem~\cite{Mikhailov2013}, it is a bounded linear operator that coincides with the restriction operator for smooth functions (i.e., $B u = u|_{\Gamma}$ forall $u$ smooth), it admits a bounded right inverse, and therefore it satisfies, for some $\beta_0 > 0$, 
    \begin{equation}
        \label{eq:trace-inequalities}
       \begin{aligned}
           &\|B u\|_{\spaceQ'} \leq  \|B\|~ \|u\|_\spaceV & \forall u \in \spaceV \\
        & \|B u\|_{\spaceQ'} \geq  \beta_0 \|u\|_\spaceV \qquad &\forall u \in \spaceV \setminus \ker(B).
       \end{aligned}
    \end{equation}

By these conditions, it also follows that  the operator $B^T$ is linear, continuous and bounding. Moreover, $\ker(B^T) = \{0\}$.  Namely, there exist $0<\|B^T\|,\ \beta_{B} >0 $ such that
 \begin{align}
    \label{eq:B2a}
    \|B^Tq\|_{\spaceV'} &\leq \|B^T\| \|q\|_{\spaceQ},  &&\quad  \forall q \in \spaceQ
    \\
    \label{eq:B2b}
    \|B^Tq\|_{\spaceV'} &\geq \beta_{B} \|q\|_{\spaceQ}, &&\quad \forall q \in \spaceQ.
 \end{align}
Condition~\eqref{eq:B2b} is equivalent to~\eqref{eq:insup-B}, by the definition of the dual norm of $\spaceV'$, and by taking the infimum over $\spaceQ$.
\end{proof}

\begin{remark}
    For the forthcoming analysis, it is important to track the dependence of the
    constant $\beta_B$ on the size of slender inclusions $|V|$. Let us denote
    with $\epsilon$ the length of the smallest dimension of $V$. From simple
    scaling arguments, supported by results on the trace inequality
    \cite{MR1145843}, we note that the quantity $\|B\|$ is uniformly bounded
    with respect to $\epsilon$. Conversely,
    $\beta_0=\mathcal{O}(\epsilon^p)$ for some $p\geq 0$. We observe that
    $\|B^T\| \simeq \beta_0^{-1}$ and $\beta_B \simeq \|B\|^{-1}$. In
    conclusion, inequality \eqref{eq:insup-B} is robust with respect to the
    diameter of $V$.
\end{remark}

By Theorems~\ref{theo:infsup-A} and \ref{theo:infsup-B} it follows immediately that Problem~\ref{eq:model-problem-weak} is well posed and admits a unique solution $(u, \lambda)$ (see, e.g., \cite{Boffi2013}). By construction, the restriction of $u$ to  $\Omega\setminus V \subset \Omega$ satisfies problem~\eqref{eq:model-problem}.

\section{Reduced Lagrange multiplier formulation}\label{sec:reduced}

When the inclusion $V$ is slender, i.e., one or more of its characteristic dimensions are small compared to the measure of $\Omega$, it may be convenient to reformulate the problem on a subdomain of $V$ whose intrinsic dimension is smaller than $d$, i.e., a representative surface, curve, or point $\gamma \subseteq V$, see Figure \ref{fig:domain-bis}. We call $m$ the intrinsic dimension of $\gamma$, and we assume that $|V| = O(\epsilon^{d-m})$, where $\epsilon$ is the smallest characteristic dimension of $V$, and $\epsilon \ll |\Omega|$.

\begin{figure}[!h]
    \centering
    \includegraphics[page=1]{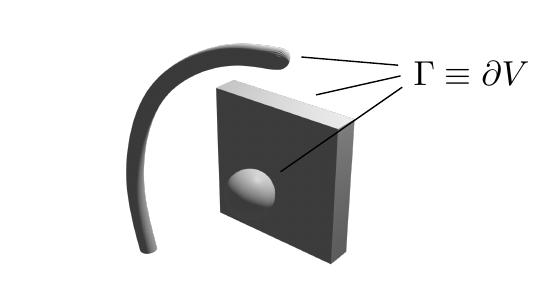}
    \hfill
    \includegraphics[page=2]{reduction}
    \caption{Example of dimensionality reduction from a general inclusion $V$ with boundary $\Gamma\equiv\gamma$ to a representative subset $\gamma$.}
    \label{fig:domain-bis}
\end{figure}

The main idea of such reformulation is rooted on the assumption that the relative measure of the domain $V$ w.r.t. to the measure of $\Omega$ allows one to accurately represent functions in $\spaceQ' = H^{1/2}(\Gamma\equiv\partial V)$ using a Sobolev space $\spaceW$ defined only on $\gamma$ (for example, $H^s(\gamma)^N$ for some $s \in [0,1]$ and for some $N\geq 1$). 

The abstract setting that enables us to perform such dimensionality reduction, is presented in the following assumption.

\begin{assumption}[Restriction operator]
    \label{ass:restriction}
    There exists a linear restriction operator $R:\spaceQ' \mapsto \spaceW'$ 
    and a positive constant $\beta_R > 0$ such that
    \begin{equation}
        \label{eq:infsup-R}
            \infsup{0\neq w \in \spaceW}{0\neq q' \in \spaceQ'} \frac{\duality{Rq', w}}{\|w\|_\spaceW \|q'\|_{\spaceQ'}} \geq \beta_R > 0.
    \end{equation}
\end{assumption}

Assumption \ref{ass:restriction} is equivalent to asking that the transposed operator $R^T: \spaceW \mapsto \spaceQ$ is a linear and bounding extension operator with trivial kernel, i.e., 
\begin{equation}
    \label{eq:RT-bounding}
    \| R^T w \|_{\spaceQ} \geq \beta_R \|w\|_{\spaceW} \qquad \forall w \in \spaceW.
\end{equation}

The operator $R^T$ defines a closed subspace $\spaceQR = \Im(R^T)  \subset \spaceQ$.  
The restriction of problem~\ref{eq:model-problem-weak} to $\spaceQR$ reads:

Given $g\in \spaceQ'$ , find $(u_R,\lambda_R)$ in $\spaceV\times \spaceQR$ such that 
\begin{equation}
    \label{eq:model-problem-restricted-weak-QR}
    \begin{aligned}
        & \duality{Au_R, v} + \duality{B^T \lambda_R, v} = \duality{f,v} && \forall v \in \spaceV\\
        &\duality{B u_R,q_R} = \duality{g,q_R} && \forall q_R \in \spaceQR.
    \end{aligned}
\end{equation}
or, equivalently:

Given $g\in \spaceQ'$ , find $(u_R,\Lambda)$ in $\spaceV\times \spaceW$ such that 
\begin{equation} 
    \label{eq:model-problem-restricted-weak-W}
    \begin{aligned}
        & \duality{Au_R, v} + \duality{B^T R^T\Lambda, v} = \duality{f,v} && \forall v \in \spaceV\\
        &\duality{R B u_R,w} = \duality{R g,w} && \forall w \in \spaceW,
    \end{aligned}
\end{equation}
where $\lambda_R = R^T \Lambda$. 
It is straightforward to see that $\lambda_R  = [\nabla u_R] \cdot n\, \ \textrm{in} \ \spaceQR$.

\begin{lemma}[Infsup on $RB$]
    \label{theo:infsup-RB}
    Under Assumption \ref{ass:restriction}, the operator $RB$ satisfies the infsup condition. More precisely, it holds
    \begin{equation}
        \label{eq:insup-BR}
        \infsup{0\neq w \in \spaceW}{0\neq v \in \spaceV} \frac{\duality{RBv, w}}{\|v\|_\spaceV \|w\|_\spaceW} \geq \beta_{R} \beta_{B}>0.
    \end{equation}
\end{lemma}

\begin{proof}
    The proof of \eqref{eq:insup-BR} follows immediately from Theorem~\ref{theo:infsup-B} and Assumption~\ref{ass:restriction}, observing that $\ker(B) \subseteq \ker(RB)$ and that:
    \begin{equation*}
        \| R B u \|_{\spaceW'} \geq \beta_R \| B u \|_{\spaceQ'} \geq \beta_R \beta_B \|u \|_\spaceV \qquad \forall u \in \spaceV\setminus\ker(R B).
    \end{equation*}
    As a result the desired inequality holds true.
\end{proof}

Standard results for mixed approximations (see, for example,~\cite{Boffi2013} Theorem 4.2.3, under the additional assumption that the operator $A$ is self-adjoint) allows one to conclude that there exists a unique solution $(u_R,\lambda_R) = (u_R, R^T\Lambda)$ to Problem~\eqref{eq:model-problem-restricted-weak-QR} (or, equivalently, to Problem~\eqref{eq:model-problem-restricted-weak-W}). Moreover, the solution of the problem is bounded by:

\begin{equation}\label{eq:stability}
    \begin{split}
     &\|u_R\|_{\spaceV} \leq \frac{1}{\alpha} \|f\|_{\spaceV'} + \frac{2\|A\|^\frac12}{\alpha^\frac12 \beta_{R}\beta_{B}} \|Rg\|_{\spaceW'}\,,
    \\
    &\|\Lambda\|_{\spaceW} \leq  \frac{2\|A\|^\frac12}{\alpha^\frac12 \beta_{R}\beta_{B}} \|f\|_{\spaceV'} + \frac{\|A\|}{\alpha (\beta_{R}\beta_{B})^2}\|Rg\|_{\spaceW'}\,. 
    \end{split}
\end{equation}

\subsection{Extension to Robin type transmission conditions}\label{sec:robin}

We now generalize problem \eqref{eq:model-problem} to embrace the case of Robin-type transmission conditions as follows,
\begin{subequations}
    \label{eq:model-problem-robin}
    \begin{align}
        \label{eq:model-problem-robin-a}
         -\Delta u &= f && \text{ in } \Omega\setminus V\\
        \label{eq:model-problem-robin-b}
         u &= 0 && \text{ on } \partial \Omega \\
        \label{eq:model-problem-robin-c}
        \kappa [\nabla u] \cdot n &=(u-g) && \text{ on } \Gamma 
    \end{align}
\end{subequations}
We notice that this problem embraces the one previously addressed. Precisely, it coincides with \eqref{eq:model-problem} for $\kappa \rightarrow 0$. In the case $\kappa \rightarrow \infty$ it represents Neumann-type transmission conditions.

The weak formulation of such problem, with enforcement of the transmission equation on $\Gamma$ by means of Lagrange multipliers reads as follows: find $u\in \spaceV \equiv H^1_0(\Omega),\,\lambda \in \spaceQ' \equiv H^{-\frac12}(\Gamma)$ such that
\begin{subequations}
    \label{eq:model-problem-weak-robin}
    \begin{align}
            \label{eq:model-problem-robin-weak-a}
            &(\nabla u, \nabla v)_{\Omega} + \duality{\lambda, v}_{\Gamma} = (f,v)_{\Omega},\ \forall v\in H^1_0(\Omega), && \forall v \in H^1_0(\Omega)
            \\
            \label{eq:model-problem-bobin-weak-b}
            &\duality{u,q}_{\Gamma} - \kappa\duality{\lambda,q}_{\Gamma} = \duality{g,q}_{\Gamma}, && \forall q\in H^{-\frac12}(\Gamma)\,.
    \end{align}
\end{subequations}

Problem \eqref{eq:model-problem-weak-robin} can be represented in operator form as the following perturbed saddle point problem: 
given $F\in\spaceV',\,G\in\spaceQ'$ find $u\in\spaceV,\,\lambda\in\spaceQ$ such that
\begin{subequations}\label{eq:model-problem-weak-robin-abstract}
      \begin{align}
        &\duality{Au,v} + \duality{B^T \lambda, v} = \duality{F,v} && \forall v \in \spaceV,
        \\
        &\duality{Bu,q} - \duality{C\lambda,q} = \duality{G,q} && \forall q \in \spaceQ,
    \end{align}  
\end{subequations}
where the operators $A$ and $B$ are defined as for \eqref{eq:model-problem-weak} in \eqref{eq:defA}, \eqref{eq:defB} and the operator $C$ takes the form,
\begin{equation}\label{eq:defC}
    C: \spaceQ \mapsto \spaceQ,
    \ \mathrm{with} \ \duality{C\lambda,q} = \kappa \duality{\lambda, q}_{\Gamma}.
\end{equation}

If $A$ and $B$ are continuous operators, $A$ is elliptic, $B$ satisfies the infsup condition of Theorem \ref{theo:infsup-B}, and $C$ is of the form \eqref{eq:defC}, 
then, owing to Theorem 4.3.2 of \cite{Boffi2013},
problem \eqref{eq:model-problem-weak-robin} admits a unique solution $u \in \spaceV,\,\lambda \in \spaceQ$ such that
\begin{align}
    \label{eq:stab-robin-1}
    \|u\|_{\spaceV} &\leq \frac{\beta_B^2+4\kappa\|A\|}{\alpha \beta^2}\|f\|_{\spaceV'}
    +\frac{2\|A\|^\frac12}{\alpha^\frac12\beta_B}\|g\|_{\spaceQ'},
    \\
    \label{eq:stab-robin-2}
    \|\lambda\|_{\spaceQ} &\leq \frac{2\|A\|^\frac12}{\alpha^\frac12\beta_B}\|f\|_{\spaceV'}
    +\frac{4\|A\|}{\kappa\|A\|+2\beta_B^2}\|g\|_{\spaceQ'}.
\end{align}

Similarly to the case of Dirichlet transmission conditions, we apply the restriction operator to problem \eqref{eq:model-problem-weak-robin}. Using the notation defined before, we obtain the following abstract problem: given $F\in\spaceV',\,G\in\spaceW'$ find $u_R\in\spaceV,\,\Lambda\in\spaceW$ such that
\begin{subequations}\label{eq:model-problem-weak-robin-reduced-abstract}
       \begin{align}
        &\duality{Au_R,v} + \duality{B^T R^T \Lambda, v} = \duality{F,v} && \forall v \in \spaceV,
        \\
        &\duality{RBu_R,w} - \duality{RCR^T\Lambda,w} = \duality{G,w} && \forall w \in \spaceW.
    \end{align}   
\end{subequations}
We note that in the particular case where $C$ is the identity, as in \eqref{eq:defC}, the perturbation term becomes $\kappa\duality{R^T\Lambda,R^Tw}$.
The operators $A$, $RB$ and $RCR^T$ satisfy the assumptions of Theorem 4.3.2 of \cite{Boffi2013}.
Then, the solution of the reduced problem $u_R,\Lambda$ also enjoys stability estimates analogous to \eqref{eq:stab-robin-1}-\eqref{eq:stab-robin-2}.

\section{Dimensionality reduction error}
\label{sec:dimensionality-reduction-error}

Let us now define and analyze the error due to the dimensionality reduction
induced by the operator $R$, namely $e_u = u-u_R,\ e_\lambda  = \lambda - R^T
\Lambda$. Subtracting equation \eqref{eq:model-problem-restricted-weak-W} from
\eqref{eq:model-problem-weak-abstract} we obtain,
\begin{equation}
    \label{eq:model-error-1}
    \begin{aligned}
        & \duality{A(u-u_R), v} + \duality{B^T (\lambda-R^T\Lambda), v} = 0 && \forall v \in \spaceV\\
        &\duality{B u,q} - \duality{B u_R, R^Tw}= \duality{g,q-R^Tw} && \forall q \in \spaceQ,\ w \in \spaceW.
    \end{aligned}
\end{equation}
For any $q_R \in \spaceQR$ we have the following orthogonality property:
\begin{equation*}
    \duality{B (u-u_R),q_R} = 0 \quad \forall q_R \in \spaceQR.
\end{equation*}
Since $\spaceQR$ is a closed subspace of $\spaceQ$, the space of orthogonal functions to $\spaceQR$, named $\spaceQ^\perp$, is such that $\spaceQ = \spaceQR \oplus \spaceQ^\perp$. Using the orthogonality property \eqref{eq:model-error-1} can be written as follows,
\begin{equation}
    \label{eq:model-error-2}
    \begin{aligned}
        & \duality{A(u-u_R), v} + \duality{B^T (\lambda-R^T\Lambda), v} = 0 && \forall v \in \spaceV\\
        &\duality{B (u - u_R),q^\perp} + \duality{B u_R, q^\perp}= \duality{g,q^\perp} && \forall q^\perp \in \spaceQ^\perp.
    \end{aligned}
\end{equation}
Problem \eqref{eq:model-error-2} becomes: find $e_u \in \spaceV$, $e_\lambda \in Q^\perp$ such that
\begin{equation}
    \label{eq:model-error-3}
    \begin{aligned}
        & \duality{A e_u, v} + \duality{B^T e_\lambda, v} = 0 && \forall v \in \spaceV\\
        &\duality{B e_u,q^\perp} = \duality{g-Bu_R,q^\perp} && \forall q^\perp \in \spaceQ^\perp.
    \end{aligned}
\end{equation}
Then, observing that the infsup stability of $B$ shown in Theorem \ref{theo:infsup-B} is satisfied by the pair of spaces $\spaceV,\,\spaceQ^\perp$, exploiting \cite[Theorem 4.2.3]{Boffi2013} applied to problem \eqref{eq:model-problem-weak-abstract} we obtain the following bounds:
\begin{equation}
    \label{eq:residual-reduced-continuous-new}
    \begin{split}
        &\|e_u\|_{\spaceV} \leq \frac{2\|A\|^\frac12}{\alpha^\frac12\beta_B}\|g-Bu_R\|_{\spaceQ'}\\
        &\|e_\lambda\|_{\spaceQ} \leq \frac{\|A\|}{(\beta_B)^2}\|g-Bu_R\|_{\spaceQ'},
    \end{split}
\end{equation}
where $\alpha$ is the coercivity constant of $A$, and $\beta_B$ is the inf-sup constant of $B$ \eqref{eq:insup-B}.
We note that all the constants of \eqref{eq:residual-reduced-continuous-new} are independent of $R$,
but they may depend on the size of $\Gamma$ through $\beta_B$.
The dimensionality reduction error and the influence of $R$ only appears through $\|g-Bu_R\|_{\spaceQ'}$. Estimates \eqref{eq:residual-reduced-continuous-new} can be regarded as a bound of the dimensionality reduction error with respect to the residual obtained by enforcing the boundary constraint $Bu=g$ on the closed subspace $\spaceQR$. Notice, however, that it is not enough to guarantee that $g \in \spaceQR'$ in order to eliminate the dimensionality reduction error, since the solution $u_R$ of the reduced problem \eqref{eq:model-problem-restricted-weak-W} does not necessarily coincide with the full order solution $u$, due to possible errors in the representation of the Lagrange multiplier $\lambda$ in the reduced dimensionality setting.

To obtain an estimate that considers also this error, we start form the following
equations, obtained again from a combination of \eqref{eq:model-problem-restricted-weak-W} and
\eqref{eq:model-problem-weak-abstract}:
\begin{equation*}
    \begin{split}
            \nonumber
        &\duality{A u_R,v} + \duality{B^TR^T\Lambda,v} 
        = \duality{A u,v} + \duality{B^T \lambda,v},\ \forall v\in \spaceV\,,
        \\
        &\duality{B u_R, R^T w} = \duality{Bu,R^T w},\  \forall w \in \spaceW\,,
    \end{split}
\end{equation*}
that is, for any $w \in \spaceW$,
\begin{equation*}
    \begin{split}
            &\duality{A (u_R -u ),v} + \duality{B^TR^T(\Lambda - w),v} 
        =\duality{B^T (\lambda - R^T w),v},\ \forall v\in \spaceV\,,
        \\
        &\duality{R B (u_R-u), w} = 0,\ \forall w \in \spaceW\,.
    \end{split}
\end{equation*}
Thanks to the stability of problem \eqref{eq:model-problem-restricted-weak-W}, namely inequalities \eqref{eq:stability}, we have
\begin{align}
    &\|u_R - u\|_{\spaceV} \leq \frac{1}{\alpha} \|B^T(\lambda - R^T w)\|_{\spaceV'}
    \\
    &\|\Lambda - w\|_{\spaceW} \leq \frac{2\|A\|^\frac12}{\alpha^\frac12 \beta_{R}\beta_{B}} \|B^T(\lambda - R^T w)\|_{\spaceV'}.
\end{align}

Finally, exploiting the continuity of $B^T$ and of $R^T$, using the triangle inequality and recalling that $w$ is a generic function in $\spaceW$, we obtain,
\begin{equation}
    \label{eq:error-reduced-continuous-new}
    \begin{split}
        &\|e_u\|_{\spaceV} \leq \frac{\|B^T\|}{\alpha} \inf\limits_{w \in \spaceW}\|\lambda - R^T w\|_{\spaceQ}\,,
        \\
        &\|e_\lambda\|_{\spaceQ} \leq \left(1+\frac{2\|A\|^\frac12\|B^T\|\|R^T\|}{\alpha^\frac12 \beta_{R}\beta_{B}}\right) \inf\limits_{w \in \spaceW} \|\lambda - R^T w\|_{\spaceQ}\,.
    \end{split}
\end{equation}
We note that, in contrast to \eqref{eq:residual-reduced-continuous-new}, the constants of \eqref{eq:error-reduced-continuous-new} depend on the properties of $R$.

We proceed similarly for the case of Robin transmission conditions.
By subtracting problem \eqref{eq:model-problem-weak-robin-reduced-abstract} from \eqref{eq:model-problem-weak-robin-abstract} and exploiting the representation $\spaceQ = \spaceQ^R \oplus \spaceQ^\perp$, we obtain that
\begin{align*}
    &\duality{A(u-u_R),v} + \duality{B^T(\lambda-R^T\Lambda),v} = 0, && \forall v \in \spaceV,
    \\
    &\duality{Bu,q} - \duality{Bu_R,R^Tw} - \kappa\duality{\lambda,q} + \kappa\duality{R^T\Lambda,R^Tw}
    = \duality{g,q-R^Tw}, && \forall q \in \spaceQ.
\end{align*}
From the latter equation, choosing $q \in \spaceQ^R$ we obtain the following relation,
\begin{equation*}
    \duality{B(u-u_R),q^R} - \kappa \duality{\lambda-R^T\Lambda,q^R} = 0, \ \forall q^R \in \spaceQ^R.
\end{equation*}
Furthermore, observing that for $\Lambda \in \spaceW$ we have $\duality{R^T\Lambda,q^\perp}=0$,
we obtain that the dimensionality reduction error related to the Robin-type transmission conditions satisfy the following problem: find $e_u \in \spaceV$, $e_\lambda \in \spaceQ^\perp$ such that
\begin{subequations}
       \begin{align*}
           &\duality{Ae_u,v} + \duality{B^Te_\lambda,v} =0, && \forall v \in \spaceV,
           \\
           &\duality{B e_u, q^\perp} - \duality{e_\lambda,q^\perp} = \duality{g-B u_R, q^\perp} && \forall v \in \spaceQ^\perp,
       \end{align*}
\end{subequations}
which shares the same structure of \eqref{eq:model-problem-weak-robin-abstract},
(with the exception that the second equation is tested on $\spaceQ^\perp$,
but the infsup stability of $B$ holds true also in this subspace)
and consequently shares the same stability property that is
\begin{align}
    \label{eq:stab-robin-err-1}
    \|e_u\|_{\spaceV} &\leq \frac{2\|A\|^\frac12}{\alpha^\frac12\beta_B}\|g-Bu_R\|_{\spaceQ'},
    \\
    \label{eq:stab-robin-err-2}
    \|e_\lambda\|_{\spaceQ} &\leq \frac{4\|A\|}{\kappa\|A\|+2\beta_B^2}\|g-Bu_R\|_{\spaceQ'}.
\end{align}

To conclude this section, we anticipate that we will use both \eqref{eq:residual-reduced-continuous-new} and \eqref{eq:error-reduced-continuous-new} to derive a priori estimates of the dimensionality reduction error with respect to $\epsilon$. However, we would like to warn the reader that the two results are not equivalently suited for this purpose. The residual type estimates do not require additional regularity to the functions on the right hand side. Conversely, the approximation type estimates leverage on the additional regularity of $\lambda$, that may not be guaranteed, to exploit optimal convergence with respect to $N$.

\section{Weighted restriction and extension operators}
\label{sec:isomorphism}

Let $\gamma \subseteq V$ be the lower dimensional representative domain
for $V$ (possibly the union of disjoint connected components). We assume that both $V$ and $\gamma$ are Lipschitz, and define a geometrical projection operator 
\begin{equation}
    \label{eq:definition-proj}
    \begin{aligned}
    & \proj :& \Gamma := \partial V &\mapsto &&\gamma \\
    & \proj^{-1}:& \gamma &\mapsto &&\partitionspace(\Gamma)
    \end{aligned}
\end{equation}
that maps uniquely each point on $\partial V$ to one point on the lower
dimensional $\gamma$.
In \eqref{eq:definition-proj}, we indicate with $\partitionspace(\Gamma)$ the
power set of $\Gamma$ (i.e., the set of all possible subsets of $\Gamma$), and
with $\proj^{-1}$ the preimage of $\proj$.  
For $s \in \gamma$, we will also indicate with $|\proj^{-1}(s)|$ the intrinsic
Hausdorff measure of the set $\proj^{-1}(s)$.
In particular the Hausdorff measure $\diff{\mathcal {H}(\Pi^{-1})}$ is such that
$|\proj^{-1}(s)| := \int_{\proj^{-1}(s)} \diff{\mathcal {H}(\Pi^{-1})}$ for all
$s \in \gamma$. 

We observe that, in principle, for different points $s \in \gamma$, the set
$\proj^{-1}(s)$ could have different intrinsic dimensionality (i.e., it could be
a curve, a surface, or a point). 
We will focus on the situation where the Hausdorff dimensionality of
$\proj^{-1}(s)$ is the same for all $s \in \gamma$. To simplify the notation, we
define
$$
D(s) := \proj^{-1}(s),\qquad \diff{D(s)} := \diff{\mathcal {H}(\Pi^{-1}(s))},
$$ 
and we assume that $|\proj^{-1}(s)| > 0$ and that $|\proj^{-1}(s)|$ is bounded
for almost all $s \in \gamma$, i.e., 
\begin{equation}
      0 < D_m\leq |D(s)|\leq D_M < \infty \qquad \forall s \in \gamma.
    \label{eq:assumption-measure}
\end{equation}

When no confusion can arise, we will also omit
from the notation the dependence of the
measure $\diff{D}$ and of the set $D$ on $s$.

Under the above assumptions, Fubini's theorem implies that the integral of any
absolutely integrable function $f$ over $\Gamma$ can be decomposed as
\begin{equation}
    \int_\Gamma f \diff{\Gamma} = 
    \int_\gamma \int_{D} f \diff{D} \diff{\gamma}.
\end{equation}

The projection $\Pi$ induces naturally, though $\Pi^{-1}$, an average operator for absolutely integrable functions $f$ on $\Gamma$ defined, for
each $s \in \gamma$, as the average of $f$ over the preimage $D(s)$:
\begin{equation}
    (\avg_0 f)(s) := \frac{1}{|D(s)|}
    \int_{D(s)} f \diff{D}(s) =: \left(\fint_{D} f \diff{D}\right)(s), \quad s \in \gamma.
    \label{eq:averageoperator}
\end{equation}

\begin{remark}[Examples]
    Two notable examples of projection operators are those induced by the extreme choices $\gamma = x_0 \in V$ (a single point) or $\gamma \equiv \Gamma$ (the
    full surface $\Gamma$). In the first case, all points on $\Gamma$ are projected
    to a single point $x_0$, and $\proj^{-1}(x_0) \equiv \Gamma$, leading to $\avg_0$
    being the classical average on $\Gamma$. In the second case, instead,
    $\proj^{-1}(s) = \{s\}$  for all $s \in \Gamma$, and the Hausdorff measure is
    the Dirac measure associated with $\Gamma$ at the point $s \in \Gamma$,
    i.e., $\avg_0 f$ is simply the pointwise evaluation of $f$.
\end{remark}

A natural extension operator from $\gamma$ to the whole surface $\Gamma$ can be
defined through the projection operator $\proj$, i.e., 
\begin{equation}
    (\ext_0 w)(x) := (w\circ \proj)(x), 
    \label{eq:extensionoperator}
\end{equation}
for any smooth $w$ on $\gamma$.
Clearly, the extension operator $\ext_0$ is the right inverse of the
average $\avg_0$:
\begin{equation}\label{eq:AE=I}
\avg_0 \ext_0 w = w, \qquad \forall w \in C^0(\overline{\gamma}),
\end{equation}
since the function $\ext_0 w$ associates to each set $D(s) =
\proj^{-1}(s)\subseteq \Gamma$ the constant value $w(s)$, whose average on
$D(s)$ coincides with $w(s)$.

These operators can be generalized to their \emph{weighted} counterparts by
defining
\begin{equation}
    \begin{split}
        (\ext_i w)(x) &:= \varphi_i(x) (w\circ \proj)(x) \\
        (\avg_i f)(s) &:= \fint_{D(s)}\varphi_i f \diff{D}(s),
    \end{split}
    \label{eq:weighted-extension-and-average}
\end{equation}
for a given choice of orthogonal weight functions $\varphi_i \in
H^1(\Gamma)\cap C^0(\Gamma)$ such that $\varphi_0 \equiv 1$ (generating the
definitions of $\avg_0$ and $\ext_0$ above) and such that
\begin{equation}
        (\avg_i\ext_j w)(x) := c_{i}\delta_{ij} w(x), \qquad \forall i,j \in [0,N), \ \forall w \in C^0(\overline{\gamma}),
    \label{eq:orthonormality}
\end{equation}
where $\delta_{ij}$ is the Kronecker delta, and $c_i$ are positive numbers. We
now work out some sufficient conditions that allow one to extend the operators
above to the Sobolev spaces $H^s(\Gamma)$ and $H^s(\gamma)$, respectively, for
$s \in [-1,1]$.

To simplify the exposition, we assume that $V$ is a single, simply connected, and non self-intersecting inclusion, and we assume that
\begin{assumption}[Isomorphism of $\hat{V}$]
    \label{ass:isomorphism}
    $V$ can be written as the image of an isomorphism 
    \begin{equation}
        \label{eq_Phi_transform}
        \Phi \colon \hat{V} \rightarrow V,
    \end{equation}
    where $\hat{V}$ is a reference domain with unit measure. We assume, moreover, that $\Phi$ satisfies the following hypotheses:
    \begin{enumerate}[i)]
        \item $\Phi \in C^{1}(\overline{\hat{V}})$, $\Phi^{-1}\in C^{1}(\overline{V})$;
        \item $0 < J_m \leq \det(\hat\nabla \Phi(\hat{x})) \leq J_M < \infty \qquad \forall \hat{x} \in \overline{\hat{V}}$;
        \item  $\refgamma$ is the pre-image of the $m$-th dimensional
        representative domain $\gamma$, i.e., $\refgamma := \Phi^{-1}(\gamma)$,
        and we assume that $\refgamma$ is a tensor product box containing the
        origin, aligned with the the last axes of the coordinates $\hat{x}$.
    \end{enumerate}
    
    The last hypothesis indicates that $\refgamma$ is a straight line directed
    along the $x_d$-axis for the cases where the dimension $m$ of $\refgamma$ is
    one, an axis aligned rectangle in the $\widehat{x_d}\times\widehat{x_{d-1}}$
    plane for the cases where the dimension $m$ of $\refgamma$ is two, and so
    on. Since $\refgamma$ contains the origin, any point $\hat{x} \in
    \hat{V}$ whose first $d-m$ components are zero belongs to $\hat{\gamma}$.
    Moreover, the inclusion domain $\hat{V}$ can be written as a tensor product
    domain of the form $\widehat{D} \times \refgamma$, and for each $\hat{s}\in \refgamma$, $\widehat{D}(\hat{s}) \equiv \widehat{D}$ is constant.
\end{assumption}
    
The tensor product structure of $\refGamma$ deriving from
Assumption~\ref{ass:isomorphism} allows one to define a reference projection
operator onto $\refgamma$ by the orthogonal projection on the last $m$ axes in
the reference coordinates $\hat{x}$ using the iso-morphism $\Phi$, i.e.:
\begin{equation}
    \label{eq:def-projection-Gamma}
   \begin{split}
    \refproj: & \hat{V} \mapsto \refgamma \\
        & \hat{x} \mapsto \sum_{i=d-m+1}^d (\hat{e}_i\otimes \hat{e}_i)\hat{x}\\
        \proj: & V \mapsto \gamma \\
        & x \mapsto \Phi( \refproj(\Phi^{-1}(x))).
   \end{split}
\end{equation}

For the reference extension and average operators
\begin{equation}
    \begin{split}
        \refext_0 \hat{w} &:= \hat{w}\circ \refproj\\
        \refavg_0 \hat{q} &:= \fint_{\widehat{D}} \hat{q} \diff{\widehat{D}},
    \end{split}
    \label{eq:reference-extension-and-average}
\end{equation}
it is possible to show that 
there exist two positive constants $\hat{C}_0^{\avg}$ and $\hat{C}_0^{\ext}$ such that
\begin{equation}
    \label{eq:reference-operators-inequalities}
    \begin{aligned}
        \| \refext_0 \hat{w} \|_{s, \refGamma} &\leq \hat{C}_0^{\ext} \|\hat{w}\|_{s, \refgamma} \qquad \forall \hat{w}\in H^s(\refgamma)\\
        \| \refavg_0 \hat{q} \|_{s, \refgamma} &\leq \hat{C}_0^{\avg} \|\hat{q}\|_{s, \refGamma} \qquad \forall \hat{q}\in H^s(\refGamma),
    \end{aligned}
\end{equation}
for $s=0$ and $s=1$, owing to the tensor product structure of $\hat{V}$. The result follows from an argument similar to 
\cite[Lemma 2.1 and Corollary 2.2]{Kuchta2021558}.

Similarly, one could pick a set of $N$ reference weight functions and derive
more general estimates for the corresponding weighted operators:

\begin{lemma}[regularity of reference weighted operators]
    \label{lem:regularity-reference-weighted-operators}
    Given a set of $N$ reference weight functions $\{\hat\varphi_i\}_{i=0}^{N}
    \in (C^0(\overline{\refGamma})\cap H^1(\refGamma))^{N*1}$, then the reference
    weighted operators
    \begin{equation}
        \label{eq:reference-weighted-operators}
        \begin{aligned}
            &\refext_i &\colon H^s(\refgamma) &\to &&H^s(\refGamma) \\
                 &  & \hat{w} &\mapsto &&\hat{\varphi}_i \hat{w}\circ \refproj,\\
                 \\
             &\refavg_i &\colon H^s(\refGamma) &\to &&H^s(\refgamma) \\
                  &  & \hat{q} &\mapsto &&\fint_{\hat{D}} \hat{q} \hat{\varphi}_i \diff{\hat{D}},
        \end{aligned}
    \end{equation}
    are bounded and linear operators for any $s\in [-1,1]$, i.e., there exist
    constants $\hat{C}_{i,s}^{\refavg}$ and $\hat{C}_{i,s}^{\refext}$ such that:
    \begin{equation}
        \label{eq:reference-weighted-operators-inequalities}
        \begin{aligned}
            \| \refext_i \hat{w} \|_{s, \refGamma} &\leq \hat{C}_{i,s}^{\refext} \|\hat{w}\|_{s, \refgamma} \qquad \forall i \in [0,N), \qquad \forall \hat{w}\in H^s(\refgamma),\\
            \| \refavg_i \hat{q} \|_{s, \refgamma} &\leq \hat{C}_{i,s}^{\refavg} \|\hat{q}\|_{s, \refGamma} \qquad \forall i \in [0,N), \qquad \forall \hat{q}\in H^s(\refGamma).
        \end{aligned}
    \end{equation}
\end{lemma}
\begin{proof}
    We begin by observing that, for a continous $\hat{w} \in C^0(\overline{\refgamma})$, we have 
    \begin{equation}
        \begin{split}
            \left\|\refext_i\left(\hat{w}\right)\right\|^2_{0,\refGamma} = 
            & \int_\refGamma \left(\hat{\varphi}_i \left(\hat{w}\circ\refproj\right) \right)^2 \diff{\refGamma} \\
            \leq &\|\hat{\varphi}_i\|^2_{0,\refGamma}  \ \|\hat{w} \circ \refproj \|_{0,\refGamma}^2\\  
            \leq & \|\hat{\varphi}_i\|^2_{0,\refGamma} \hat{D}^2  \| \hat{w} \|_{0,\gamma}^2, 
        \end{split}
        \label{eq:reference-l2-norm-extension}
    \end{equation}
    owing to the tensor product structure of $\refGamma=\hat{D}\times\refgamma$.

    Similarly, for  $w\in C^1(\refgamma)$, we have that 
    \begin{equation}
        \begin{split}
            |\refext_i \hat{w}|^2_{1,\refGamma} = 
            & \int_\refGamma \left( \hat{\nabla}_\refGamma \left( \hat{\varphi}_i \hat{w}\circ\refproj\right) \right)^2 \diff{\refGamma} \\
            \leq & |\hat{\varphi}_i|^2_{1,\refGamma}\|\hat{w} \circ \refproj \|_{0,\refGamma}^2 + \|\hat{\varphi}_i\|^2_{0,\refGamma} | \hat{w} \circ \refproj |_{1,\refGamma}^2\\   
            \leq & |\hat{\varphi}_i|^2_{1,\refGamma} |\hat{D}|^2 \|\hat{w} \|_{0,\refgamma}^2 + \|\hat{\varphi}_i\|^2_{0,\refGamma} |\hat{D}|^2 | \hat{w} |_{1,\refgamma}^2\\
            \leq & C \|\hat{\varphi}_i\|^2_{1,\refGamma}\|\hat{w} \|_{1,\refgamma}^2.
        \end{split}
        \label{eq:reference-h1-seminorm-extension}
    \end{equation}

    By a density argument, and interpolating the estimates above using the real
    method, we obtain
    \begin{equation}
        \|\refext_i \hat{w} \|_{s,\refGamma} \leq \hat{C}_{i,s}^{\refext} \|\hat{w}\|_{s, \refgamma}, \qquad \forall \hat{w} \in H^s(\refgamma),\qquad s\in[0,1],\qquad i \in [0,N),
        \label{eq:reference-extension-estimate}
    \end{equation}
    where the constants $\hat{C}_{i,s}^{\refext}$ depend on $i$, $\hat{\varphi}_i$, and $s$. This implies that $\refext_i$ is a bounded operator from $H^s(\refgamma)$ to $H^s(\refGamma)$ for all $s\in[0,1]$ and $i\in[0,N)$.

    We now observe that the following identity holds:
    \begin{multline}
        (\hat{w}, \refavg_i \hat{q})_{\refgamma} = \int_\refgamma \hat{w} \hat{\varphi}_i\left( \fint_{\hat{D}} \hat{q} \diff{\hat{D}} \right) \diff{\refgamma} = \int_\refgamma \hat{w} \hat{\varphi}_i\left(\frac{1}{|\hat{D}|} \int_{\hat{D}} \hat{q} \diff{\hat{D}} \right) \diff{\refgamma} \\
        = \int_\refGamma \refext_i\left(\frac{\hat{w}}{|\hat{D}|}\right) \hat{q} \diff{\refGamma} = \left(\refext_i\left(\frac{\hat{w}}{|\hat{D}|}\right) , \hat{q}\right)_{\refgamma},
        \label{eq:reference-weighted-operators-identity}
    \end{multline}
    and we conclude that $\refavg_i$ can be identified with the transpose of
    $\refext_i$ applied to $w/|\hat{D}|$, i.e., $\refavg_i$ is a bounded linear operator
    from $H^{-s}(\refGamma)$ to $H^{-s}(\refgamma)$ for the same $s \in [0,1]$ above
    by replacing the $L^2$ scalar product in the identification
    \eqref{eq:reference-weighted-operators-identity} with a duality pairing.

    The proof for $\refavg_i$ in the case $s=1$ follows a similar line:
    \begin{equation*}
        \begin{split}
            | \refavg_i \hat{q} |^2_{1, \refgamma} = &\int_\refgamma \left(\hat{\nabla}_\refgamma  \fint_{\hat{D}} \hat{\varphi}_i  \hat{q} \diff{\hat{D}} \right)^2\diff\refgamma \\
            \leq & \int_\refgamma \int_{\hat{D}} \left(\hat{\nabla}_\refGamma \left(\frac1{|\hat{D}|}  \hat{\varphi}_i  \hat{q}  \right)\right)^2\diff{\hat{D}} \diff\refgamma \\
            \leq  & C \|\hat{\varphi}_i\|^2_{1,\refGamma}\|\hat{q}\|^2_{1,\refGamma} \qquad \forall i \in [0,N), \qquad \forall \hat{q}\in H^1(\refGamma).
        \end{split}
    \end{equation*}

    By a density argument, and interpolating the estimate for $s=0$ and $s=1$ using the real method, we conclude that $\refavg_i$ is a bounded linear operator from $H^{s}(\refGamma)$ to $H^{s}(\refgamma)$ for all $s\in[0,1]$ and $i\in[0,N)$:
    \begin{equation}
        \|\refavg_i \hat{q} \|_{s,\refgamma} \leq C_{i,s}^{\refavg} \|\hat{q}\|_{s, \refGamma}, \qquad \forall \hat{q} \in H^s(\refGamma),\qquad s\in[0,1],\qquad i \in [0,N),
        \label{eq:reference-average-estimate}
    \end{equation}
    and again we use the identification \eqref{eq:reference-weighted-operators-identity}
    to conclude that $\refext_i$ is therefore also bounded and linear from
    $H^{-s}(\refgamma)$ to $H^{-s}(\refGamma)$ for all $s\in[0,1]$ and $i\in[0,N)$.
\end{proof}

Under Assumption~\ref{ass:isomorphism}, we can now consider the weighted
operators $\avg_i$ and $\ext_i$ arising from the projection operator induced by
$\Phi$, i.e., $\proj := \Phi\circ\refproj\circ\Phi^{-1}$ and weighted by
$\varphi_i := \hat{\varphi}_i\circ\Phi^{-1}$. Notice that, for any $q\in [0,1]$
and for $\hat A\subseteq \overline{\widehat{V}}$ and $A = \Phi(A)$, we have that
the following generalized scaling arguments hold:
    
\begin{equation}
    \label{eq:scaling}
    \begin{aligned}
        \| w \circ \Phi \|_{q, \hat A} &\leq |\Phi|^{q}_{1,\infty,\hat{V}} ~J_m^{-\frac12} ~\|w\|_{q, A}&\forall  w \in H^q(A), \\
        \| w \|_{q, A} &\leq |\Phi^{-1}|^{q}_{1,\infty,V}~ J_M^{\frac12} ~\|w \circ \Phi\|_{q, \hat A}  &\forall  w \in H^q(A).
    \end{aligned}
\end{equation}

Such arguments are common in the literature of high order finite element
methods, see, for example, ~\cite[Section 3.3]{Kawecki2020}, and can be
obtained, for a non-integer $q$ in $[0,1],$ by interpolating the estimates for
$q=0$ and $q=1$ using the real method. By using the push forward of the
reference basis $\hat{\varphi}_i$, we can now ensure similar regularity
properties also for $\avg_i$ and $\ext_i$:

\begin{theorem}[regularity of weighted operators]
    \label{theo:regularity-weighted-operators}
    Given a set of $N$ weight functions $\{\varphi_i\}_{i=0}^{N} \in
    (C^0(\overline{\Gamma})\cap H^1(\Gamma))^{N+1}$ and provided that
    Assumption~\ref{ass:isomorphism} holds, then the weighted operators
    \begin{equation}
        \label{eq:weighted-operators}
        \begin{aligned}
             &\avg_i &\colon H^s(\Gamma) &\to &&H^s(\gamma) \\
                  &  & w &\mapsto &&\fint_{D} w \varphi_i \diff{D},\\
                    \\
             &\ext_i &\colon H^s(\gamma) &\to &&H^s(\Gamma) \\
                  &  & f &\mapsto &&\varphi_i f\circ \proj,
        \end{aligned}
    \end{equation}
    are bounded and linear operators for any $s\in [-1,1]$, i.e., there exist
    constants $C_{i,s}^{\avg}$ and $C_{i,s}^{\ext}$ such that:
    \begin{equation}
        \label{eq:weighted-operators-inequalities}
        \begin{aligned}
            \| \avg_i q \|_{s, \gamma} &\leq C_{i,s}^{\avg} \|q\|_{s, \Gamma} \qquad \forall i \in [0,N), \qquad \forall q\in H^s(\Gamma)\\
            \| \ext_i w \|_{s, \Gamma} &\leq C_{i,s}^{\ext} \|w\|_{s, \gamma} \qquad \forall i \in [0,N), \qquad \forall w\in H^s(\gamma).
        \end{aligned}
    \end{equation}
    
\end{theorem}
\begin{proof}
    The proof follows from a combination of
    Lemma~\ref{lem:regularity-reference-weighted-operators} and the generalized
    scaling arguments~\eqref{eq:scaling-properties}, where the resulting
    constants can be shown to be bounded by 
    \begin{equation}
        \label{eq:constants-regularity-weighted-operators}
        \begin{aligned}
            C_{i,s}^{\ext} &\leq \hat{C}_{i,s}^{\refext}J_M^{\frac12} J_m^{-\frac12} 
            |\Phi|^s_{1,\infty,\refGamma} |\Phi^{-1}|^s_{1,\infty,\Gamma} \\
            C_{i,s}^{\avg} &\leq \hat{C}_{i,s}^{\refavg}J_M^{\frac12} J_m^{-\frac12} 
            |\Phi|^s_{1,\infty,\refGamma} |\Phi^{-1}|^s_{1,\infty,\Gamma}.
        \end{aligned}
    \end{equation}
\end{proof}

Under these rather general assumptions, and taking weight functions $\varphi_i$
that are orthogonal and such that the average and extension operators
satisfy the scaling property:
    \begin{equation}
        \label{eq:scaling-properties}
        \begin{aligned}
            \avg_i\ext_j w = c_i \delta_{ij} w,
        \end{aligned}
    \end{equation}
with $c_i$ positive constants, we can construct modal average and extension
operators that group together the individual average and extension operators,
i.e., we define
\begin{equation}
    \label{eq:definition-operator-R-transpose}
    \begin{aligned}
        R^T: & H^s(\gamma)^{N+1} &\mapsto &H^s(\Gamma) \\
            & w &\mapsto & \sum_{i=0}^{N}\ext_i w_i,
    \end{aligned}
\end{equation}
and
\begin{equation}
    \label{eq:definition-operator-P}
    \begin{aligned}
        P: & H^s(\Gamma)  &\mapsto & H^s(\gamma)^{N+1}\\
            & q &\mapsto & \{ c_i^{-1} \avg_i q \}_{i=0}^{N},
    \end{aligned}
\end{equation}
with the property that $P$ is a left inverse for $R^T$ by construction ($P
R^T=I_{H^s(\gamma)^{N+1}}$), and $(R^T P)^2=R^T P$ is a projection operator from
$H^s(\Gamma)$ to $H^s(\Gamma)$, that only retains $N$ modes of $q \in H^s(\Gamma)$ (the projection of $q$ onto the $\varphi_i$ basis on $\Gamma$). 
Let us consider the space $\spaceWN$ defined below, where the particular case $\spaceWZ$ coincides with $\spaceW$,
\begin{equation}
    \label{eq:spaceW}
    \spaceWN := (H^{\frac12}(\gamma))^{N+1}, \qquad \|w\|^2_{\spaceWN} := \sum_{i=0}^{N} \|w_i\|^2_{\frac12, \gamma},
\end{equation}
then $R^T$ satisfies the inf-sup condition, i.e.,  Assumption~\ref{ass:restriction}:

\begin{theorem}[Modal extension operator]
    \label{theo:reference-extension}
    Under the same assumptions of
    Theorem~\ref{theo:regularity-weighted-operators}, the extension operator
    $R^T: \spaceW \mapsto \spaceQ'$ defined in
    \eqref{eq:definition-operator-R-transpose}   
    satisfies the infsup condition~\eqref{eq:infsup-R} for $N \geq 0$, that is, there
    exists a positive constant $\beta_R$ such that, for any $w \in \spaceW$, we
    have
    \begin{equation}
        \label{eq:infsup-RT}
        \|R^T w \|_{\spaceQ} \geq \beta_R \|w\|_{\spaceW}.
    \end{equation}
\end{theorem}
\begin{proof}
    The operator $R^T$ posseses the left inverse $P$, defined in
    \eqref{eq:definition-operator-P}, which is bounded and linear owing to
    Theorem~\ref{theo:regularity-weighted-operators} for $s=-1/2$. Existence of
    a bounded left inverse of $R^T$ is equivalent to the \emph{inf-sup}
    condition \eqref{eq:infsup-RT}.
\end{proof}

\section{The Fourier extension operator for 1D-3D coupling}\label{sec:fourier}

As a concrete example of a general extension operator, we consider the case where the domain $V$ is isomorphic to a cylinder with unit measure through the mapping $\Phi$, and we set $\refgamma$ to be the reference cylinder centerline, which we assume to be aligned with the axis $\hat{e}_3$, i.e., $\refgamma = \{0\}\times\{0\}\times[0,1]$. We denote with $\widehat{(\cdot)}$ the functions and space coordinates on the reference domain $\hat{V}$, and we define a projection to the centerline $\gamma$ through $\Phi$ and the orthogonal projection on the $\hat{e}_3$-axis in the following way:
\begin{equation}
    \label{eq:def-projection-Gamma-cylinder}
   \begin{split}
    \refproj: & \hat{V} \mapsto \refgamma \\
        & \hat{x} \mapsto (\hat{e}_3\otimes \hat{e}_3)\hat{x}\\
        \proj: & V \mapsto \gamma \\
        & x \mapsto \Phi( \refproj(\Phi^{-1}(x))).
   \end{split}
\end{equation}

With this geometric structure in mind, we define a set of harmonic basis
functions that satisfy the assumptions of
Lemma~\ref{lem:regularity-reference-weighted-operators} by building them on the
reference domain $\hat{V}$, i.e., we set $\hat{\varphi}_i$ to be such that
\begin{equation}
    \label{eq:fourier-extension}
    \hat{\varphi_i}(\hat{x}) = \hat{\varphi}_i(\hat{x}_1, \hat{x}_2)
\end{equation}
where the functions $\hat{\varphi}_i$ are harmonic, constant along the $\hat{e}_3$ direction, and form an orthogonal set of basis functions in $L^2(\hat{B})$,
where $\hat{B}$ is the unit ball in $\mathbb{R}^{d-1}$ that represents the cross section of the unit cylinder $\hat{V}$. 

We use Fourier modes and cylindrical harmonics to define 
$\hat{\varphi}_0(\hat{x}_1, \hat{x}_2)=1$ and $\hat{\varphi}_i^*(\hat{x}_1, \hat{x}_2)$, i.e., for $1 \leq i \leq n$ and $*=c,s$, we set (in cylinder coordinates)
\begin{equation}
\label{eq:fourier-modes}
\hat{\varphi}_{i}^c(\hat{\rho}, \hat{\theta}) = \hat{\rho}^i \cos(i\hat{\theta}),
\quad
\hat{\varphi}_{i}^s(\hat{\rho}, \hat{\theta}) =  \hat{\rho}^i \sin(i\hat{\theta}),
\end{equation}
where $\hat{x}_1 = \hat{\rho}\cos(\hat{\theta})$, $\hat{x}_2 = \hat{\rho}\sin(\hat{\theta})$.
Then, with a little abuse of notation we obtain the following weighted projectors,
\begin{equation*}
    \refavg_0 \hat q =\fint_{\partial\hat{B}} \hat q \diff{\partial\hat{B}},
    \quad
    \refavg_i^* \hat q =\fint_{\partial\hat{B}} \hat q \hat{\varphi}_{i}^* \diff{\partial\hat{B}},
    \quad *=c,s,
    \quad \forall \hat q \in H^s(\hat \Gamma)\,.
\end{equation*}
The definitions of these weighted average operators on the physical domain follow similarly and will be used later on.

Using the Fourier modes up to the order $n$, the total number of modes become $N=2n+1$.
We note that for $n=N-1=0$ the extension operator $R^T$  extends a function $w$ defined on $\gamma$ on the entire domain $V$, in a constant way on iso-surfaces of $\Phi$ at constant ${\hat{x}_3}$. Its left inverse operator $P$ takes the average of a function on sections of $\Gamma$ and uses that value to construct a function on $\spaceW$. For $n,N>1$, the passage from $\gamma$ to $\Gamma$  entails the projection of higher order moments, using more than one degree of freedom on each cross section of the cylinder.

\subsection{An example: a 1D fiber embedded into a 3D domain}\label{sec:fiber}

We consider a narrow fiber, $V$, embedded into the domain $\Omega$. Assuming that the fiber cross sectional radius, named $\epsilon$, is small with respect to the characteristic size of the whole domain (comparable to the unit value), we aim to analyze how the dimensionality reduction error, namely $e_u = u-u_R, \ e_\lambda = \lambda - R^T\Lambda$, scales with respect to the radius of the fiber.
{\color{black} The analysis presented here is an extension to the 3D-1D case of the one developed for the Poisson problem with small holes in \cite{boulakia:hal-03501521}.}

For simplicity, let us consider a rectilinear fiber $V=\{\mathbf{x}\in \mathbb{R}^3: x_1=\rho\cos\theta,\ x_2=\rho\sin\theta,\ x_3=z, \ (\rho,\theta,z) \in (0,\epsilon) \times [0,2\pi] \times (0,L)\}$ isomorphic to the unit cylinder $\hat V$ through the transformation $\Phi: \rho = \epsilon \hat \rho,\ \theta = \hat \theta,\ z = L \hat z$. It is straightforward to see that $\mathrm{det}(\hat \nabla \Phi (\hat x))= J_m = J_M = \epsilon L$ and that $|\Phi|_{1,\infty,\refGamma}=\epsilon L, \ |\Phi^{-1}|_{1,\infty,\Gamma} = (\epsilon L)^{-1}$. We notice that in this case the constant $|\Phi|^{-s}_{1,\infty,\refGamma} J_m^{\frac12} |\Phi^{-1}|^{-s}_{1,\infty,\Gamma} J_M^{-\frac12}$ for $s=1/2$ does not depend on $\epsilon$ and $L$.

In the previous sections, the dimensionality reduction error has been bounded in two possible ways: in \eqref{eq:residual-reduced-continuous-new} on the basis of the residual obtained by projecting the boundary constraint on $\spaceQR$; in \eqref{eq:error-reduced-continuous-new} using the approximation error of $\spaceQR$ with respect to $\spaceQ$. In what follows we address both approaches for this particular case.

\subsubsection{Model error bound based on the residual}\label{sec:error-analysis-res}
Adopting the first approach, we analyze how $\|g-Bu_R\|_{\spaceQ'}$ scales with respect to $\epsilon$. Here we analyze the main mechanism that governs the decay of the residual when the radius of the fiber shrinks. Let $\tilde V \supset V$ be a fiber of radius $\delta > \epsilon$ and let be $\partial \tilde B \supset \partial B$ the cross sections of $\tilde V$ and $V$ respectively. As both cylinders have constant cross sections, we omit to denote the axial coordinate at which the cross section is evaluated.
Let $v \in H^1_0(\Omega)$ be the weak solution of $-\Delta v = f$ in $\Omega$ with $v=g$ on $\partial\Omega$, with $f\in L^2(\Omega)$, $g \in H^\frac12(\partial \Omega)$. It is known that $\|v\|_{1,\Omega} \leq C \left(\|f\|_{0,\Omega}+ \|g\|_{\frac12,\partial\Omega}\right)$. Let us finally set the technical assumption $\mathrm{supp}(f) \cap \tilde V = \emptyset$, as a result of which $v$ is harmonic on $\tilde V$ and we can represent $v$ as follows,
\begin{equation}\label{eq:harmonic}
    v(\rho,\theta,x_3) = \avg_0 v + \sum_{i=1}^\infty 
    \left(\frac{\rho}{\delta}\right)^i [(\avg_i^c v)|_{\partial \tilde B}(x_3) \cos (i\theta) 
    + (\avg_i^s v)|_{\partial \tilde B}(x_3) \sin (i\theta)]
\end{equation}
where $(\avg_i^* v)|_{\partial \tilde B}(x_3)=\int_{\partial \tilde B} v \varphi_i^*(\delta,\theta) \delta d\theta$ and $\varphi_i^*$ are defined in \eqref{eq:fourier-modes}.
We note that the projector $\avg_i^*$ with $*=c,s$ is related to $R^T P v$ previously defined.
In particular, with the exception of constant scaling factors, $R^T P v$ coincides with $\sum_{i=0}^N\ext_i^*\avg_i^*$ on $\partial \tilde B$.
Taking the same representation with respect to $V$ and comparing term by term, we obtain,
\begin{equation*}
    (\avg_i^* v)|_{\partial B}(x_3) = \left(\frac{\epsilon}{\delta}\right)^i (\avg_i^* v)|_{\partial \tilde B}(x_3).
\end{equation*}
Furthermore we have,
\begin{equation}\label{eq:fourier-bound-tilde}
    \|\avg_i^* v|_{\partial \tilde B}\|_{0,\gamma} \leq \|v\|_{0,\partial \tilde V} \leq C \|v\|_{1,\Omega} \leq C \left(\|f\|_{0,\Omega}+ \|g\|_{\frac12,\partial\Omega}\right),
\end{equation}
with constants independent of $\epsilon$ and thus we conclude that for any $0<\epsilon<\delta$ we have
\begin{equation}\label{eq:fourier-bound}
    \|\avg_i^* v|_{\partial B}\|_{0,\gamma} \leq C \left(\frac{\epsilon}{\delta}\right)^i 
    \left(\|f\|_{0,\Omega}+ \|g\|_{\frac12,\partial\Omega}\right).
\end{equation}

In the analysis, we neglect for simplicity the error arising for the projection of the right hand side $g$. More precisely, we assume that $g \in \spaceQR$, as a result $\avg^*_i g |_{i=n+1}^\infty =0$, $*=c,s$. Owing to the second of \eqref{eq:model-problem-restricted-weak-QR} we obtain that $\avg^*_i (Bu_R-g) |_{i=n+1}^\infty =\avg^*_i Bu_R |_{i=n+1}^\infty$. Observing that $\spaceQR$ and $\spaceQ^\perp$ are orthogonal spaces, and reminding that $g\in \spaceQR$ and $(g-Bu_R) \in \spaceQ^\perp$, we obtain that $\|g - Bu_R\|_{\spaceQ'}=\|Bu_R^\perp\|_{\spaceQ'}$, where $Bu_R^\perp$ is the projection of $Bu_R$ onto $\spaceQ^\perp$.

The next step is to show that the coefficients of the Fourier expansion of $u_R^\perp|_{\Gamma}$ satisfy the following property:
\begin{align}
    \nonumber
    &u_R^\perp|_{\Gamma} = \\
    \label{eq:harmonic-lemma-1}
    &\quad \sum_{i=n+1}^\infty \left(\frac{\epsilon}{\delta}\right)^i [(\avg_i^c u_R)|_{\partial \tilde B}(x_3) \cos (i\theta) + (\avg_i^s u_R)|_{\partial \tilde B}(x_3) \sin (i\theta)],
    \\
    \label{eq:harmonic-lemma-2}
    &\|(\avg_i^c u_R)|_{\partial \tilde B}(x_3)\|_{0,\gamma},\,
    \|(\avg_i^c u_R)|_{\partial \tilde B}(x_3)\|_{0,\gamma}
    \leq C \left(\|f\|_{0,\Omega}+ \|g\|_{\frac12,\partial\Omega}\right),
\end{align}

We first observe that $u_R \in H^1_0(\Omega)$ is harmonic in $V$, then we have 
\begin{equation*}
u_R^\perp(\rho, \theta, z) =  \displaystyle \sum_{n = n+1}^{\infty} \rho^{i}[A_i(x_3) \cos(i\theta) + B_i(x_3) \sin(i\theta)],
\end{equation*}
and it is also harmonic in the annulus $\tilde V \setminus V$ such that
\begin{multline*}
u_R^\perp(\rho, \theta, z) =  \displaystyle \sum_{i = n+1}^{\infty} \{[C_i(x_3) \rho^{i} + D_i(x_3)(\rho^{-i})] \cos(i\theta) \\
+ [E_i(x_3) \rho^{i} + F_i(x_3)(\rho^{-i})]\sin(i\theta)\}.
\end{multline*}
By applying the matching conditions on the jump of $u_R$ and its gradients across $\Gamma$ we obtain the following constraints on the coefficients $A_i,B_i,C_i,D_i,E_i,F_i$:
\begin{equation*}
\begin{cases}
A_i \epsilon^{i} = C_i \epsilon^{i} + D_i \epsilon^{-i} = 0, & \\
B_i \epsilon^{i} = E_i \epsilon^{i} + F_i \epsilon^{-i} = 0 & \\
A_i \epsilon^{i} = C_i \epsilon^{i} - D_i \epsilon^{-i} = 0, & \\
B_i \epsilon^{i} = E_i \epsilon^{i} - F_i \epsilon^{-i} = 0, & \\
\end{cases}
\end{equation*}
which imply that $A_i=C_i, \ B_i=D_i, \ D_i=F_i=0$ for $i>n+1$.
Now using expression \eqref{eq:harmonic} we have that
$A_i=C_i=\delta^{-i} (\avg_i^c u_R)_{\partial \tilde B}, \ 
B_i=D_i=\delta^{-i} (\avg_i^s u_R)_{\partial \tilde B}$ 
and we can write $u_R$ on $\tilde V$ as follows,
\begin{equation*}
    u_R^\perp(\rho,\theta,z) = \sum_{i=n+1}^\infty 
    \left(\frac{\rho}{\delta}\right)^i [(\avg_i^c u_R)|_{\partial \tilde B}(x_3) \cos (i\theta) 
    + (\avg_i^s u_R)|_{\partial \tilde B}(x_3) \sin (i\theta)],
\end{equation*}
that proves \eqref{eq:harmonic-lemma-1} by taking $\rho=\epsilon$.
In addition, \eqref{eq:harmonic-lemma-2} follows from \eqref{eq:fourier-bound-tilde}.

The final step is to estimate $\|Bu_R^\perp\|_{\spaceQ'}$ using \eqref{eq:harmonic-lemma-1} and \eqref{eq:harmonic-lemma-2}. 
Let us recall that for any $v\in H^\frac12(0,2\pi)$ we have,
\begin{equation*}
\mathrm{if} \
v = \avg_0v + \sum_{i=1}^\infty \sum_{*=c,s} \avg^*_i v
\ \mathrm{then} \
\|v\|_{\frac12,(0,2\pi)}^2=(\avg_0)^2 + \sum_{i=1}^\infty \sum_{*=c,s} (1+i)(\avg^*_i)^2\,.
\end{equation*}
We proceed as follows,
\begin{multline*}
    \|Bu_R^\perp\|_{\spaceQ'} \leq \left[\sum_{i=n+1}^\infty (1+i)\left(\frac{\epsilon}{\delta}\right)^{2i}
    \left(\|(\avg_i^c u_R)|_{\partial \tilde B}(x_3)\|_{0,\gamma}^2
    + \|(\avg_i^s u_R)|_{\partial \tilde B}(x_3)\|_{0,\gamma}^2 \right)\right]^\frac12
    \\
    \leq C \left(\|f\|_{0,\Omega}+ \|g\|_{\frac12,\partial\Omega}\right) \left(\frac{\epsilon}{\delta}\right)^{n+1}
    \left[\sum_{i=n+1}^\infty (1+i) \left(\frac{\epsilon}{\delta}\right)^{2(i-n-1)}\right]^\frac12
    \\
    \leq C \left(\|f\|_{0,\Omega}+ \|g\|_{\frac12,\partial\Omega}\right) \left(\frac{\epsilon}{\delta}\right)^{n+1}\,.
\end{multline*}
In conclusion, if $R$ represents the projection of $\spaceQ$ 
on the first $n$ cross-sectional Fourier modes defined on a narrow cylinder $\Gamma$ of radius $\epsilon$, under the restrictive assumption $\mathrm{supp}f \cap \tilde V = \emptyset$ for any $\tilde V \supset V$ we have
\begin{equation}
    \label{eq:residual-reduced-continuous-new-fourier}
    \begin{split}
        &\|e_u\|_{\spaceV} \leq C \left(\frac{\epsilon}{\delta}\right)^{n+1} \frac{\|A\|^\frac12}{\alpha^\frac12\beta_B} \left(\|f\|_{0,\Omega}+ \|g\|_{\frac12,\partial\Omega}\right),\\
        &\|e_\lambda\|_{\spaceQ} \leq C \left(\frac{\epsilon}{\delta}\right)^{n+1} \frac{\|A\|}{(\beta_B)^2} \left(\|f\|_{0,\Omega}+ \|g\|_{\frac12,\partial\Omega}\right),
    \end{split}
\end{equation}
where the constants of \eqref{eq:residual-reduced-continuous-new-fourier} do not depend on $\epsilon$.

A few remarks about the previous analysis are in order. First, we observe that it does not require regularity assumptions other that the minimal regularity ensured by the well posedness analysis of the full and reduced problems. 

Second, it highlights that the dimensionality reduction error is affected by the distance of $V$ from the boundary of from any other forcing term, quantified by means of the parameter $\delta$. In other words, it shows that if $\delta$ decreases, then the projection of higher modes is required to maintain a desired level of error.

Third, we notice that the results until section \ref{sec:isomorphism} are general with respect to the shape of $\Gamma$, provided that some regularity assumptions are satisfied by the mapping $\Phi$. Conversely, the analysis presented here strongly depend on the assumption that the inclusion $V$ is a cylinder of circular section. For example, on a generic section it would be no longer true  that the projection on $N$ Fourier modes implies that 
    \begin{equation*}
        u_R^\perp(\rho, \theta, z) =  \displaystyle \sum_{i = n+1}^{\infty} \rho^{i}[A_i(x_3) \cos(i\theta) + B_i(x_3) \sin(i\theta)].
    \end{equation*}

\subsubsection{Model error bound based on the approximation error}
We now address the analysis based on inequalities \eqref{eq:error-reduced-continuous-new}.
We observe that
\begin{equation*}
    \inf\limits_{w \in \spaceWN}\|\lambda - R^T w\|_{\spaceQ}
    \leq \|\lambda - R^T P \lambda \|_{L^2(\Gamma)}\,,
    \quad \forall \lambda \in L^2(\Gamma)\,.
\end{equation*}
Using the scaling argument \eqref{eq:scaling} with $q=0$, we obtain
\begin{equation*}
    \|\lambda - R^T P \lambda \|_{L^2(\Gamma)} 
    \leq J_M^{\frac12} 
    \left\|\hat\lambda - \left(\refavg_0\hat\lambda\hat\varphi_0 + \sum_{i=1}^n\sum_{*=c,s}\refavg_i^*\hat\lambda\hat\varphi_i^*\right)\right\|_{L^2(\partial \hat V)}\,.
\end{equation*}
We now apply the approximation properties of trigonometric polynomials, see for example \cite{canuto2007spectral}. Precisely, for any $0 \leq q$, for any $\hat v \in H^q(0,2\pi)$ we have
\begin{equation*}
    \left\|\hat v(\theta) - \left[\refavg_0\hat v(\theta)  
    + \sum_{i=1}^n \left(\refavg_i^c\hat v\cos(i\theta)
    + \refavg_i^s\hat v\sin(i\theta) \right)
    \right]\right\|_{L^2(0,2\pi)}
    \leq C n^{-q}|v|_{H^q(0,2\pi)}\,.
\end{equation*}
Then, assuming $\hat \lambda \in H^q (\partial \hat V)$ and extending the previous classical error bound on the whole $\partial \hat V$ and  we have,
\begin{equation*}
    \left\|\hat\lambda - \left(\refavg_0\hat\lambda(\hat x_3)\hat\varphi_0(\hat\rho,\hat\theta) + \sum_{i=1}^n\sum_{*=c,s}\refavg_i^*\hat\lambda(\hat x_3)\hat\varphi_i^*(\hat\rho,\hat\theta)\right)\right\|_{L^2(\partial \hat V)}
    \\
    \leq \hat C n^{-q} \|\hat \lambda\|_{H^q(\partial \hat V)}\,,
\end{equation*}
where the constant $\hat C$ is clearly independent of $\epsilon$.
Then, using \eqref{eq:scaling}, we map back the right hand side to the physical domain $\Gamma$,
\begin{equation*}
    \|\hat \lambda\|_{H^q(\partial \hat V)} 
    \leq |\Phi|^{q}_{1,\infty,\hat V} J_m^{-\frac12} 
    \|\lambda\|_{H^q(\Gamma)}\,.
\end{equation*}
Provided that $\lambda \in H^p(\Gamma)$, putting together the previous estimates and reminding that the constant $|\Phi|^{q}_{1,\infty,\hat V}  = \mathcal{O} (\epsilon^{q})$, while $J_M^{\frac12} J_m^{-\frac12}$ is independent of $\epsilon$, we obtain
\begin{equation}\label{eq:approximation-reduced-continuous-new-fourier}
    \|\lambda - R^T P \lambda \|_{\spaceQ} 
    \leq C \left(\frac{\epsilon}{n}\right)^q \|\lambda\|_{H^q(\Gamma)}\,,
\end{equation}
where the constant $C$ is independent of $\epsilon$.

It is obvious that \eqref{eq:approximation-reduced-continuous-new-fourier}
describes a much faster convergence of the dimensionality reduction error with $\epsilon \rightarrow 0, n,N\rightarrow \infty$ than \eqref{eq:residual-reduced-continuous-new-fourier}, provided that the Lagrange multiplier is regular enough, namely $\lambda \in H^p(\Gamma)$ with $p>1$. In absence of this additional regularity, we rely on \eqref{eq:residual-reduced-continuous-new-fourier}.

\section{Numerical approximation of the 1D-3D coupling}\label{sec:approximation}

Let us consider the discrete counterpart of problem \eqref{eq:model-problem-restricted-weak-W}, discretized by means of the finite element method. We develop the numerical discretization in the particular case already addressed in section \ref{sec:fiber}, namely $V$ is a rectilinear fiber $V=\{\mathbf{x}\in \mathbb{R}^3: x_1=\rho\cos\theta,\ x_2=\rho\sin\theta,\ x_3, \ (\rho,\theta,x_3) \in (0,\epsilon) \times [0,2\pi] \times (0,L)\}$. The general case of a curvilinear fiber can be addressed by discretizing it as a collection of piecewise linear segments, each treated separately as discussed in what follows. We note that in the discrete case using the mapping $\Phi: \hat V \mapsto V$ is impractical because it affects also the computational meshes. For this reason, in the discrete case we work only on the physical domain.

Let $X_h^k(\Omega) \subset \spaceV$ be the space of Lagrangian finite elements of polynomial order $k$ defined on a family of quasi-uniform meshes $\mathcal{T}_h^\Omega$. Under the assumption that $\gamma$ is a straight line, for the discretization of the Lagrange multiplier space we consider a family of one-dimensional partitions of $\gamma$ and we define a finite element space $X_h^k(\gamma) \subset \spaceWZ$ of piecewise polynomials of order $k$ defined on it. Let $N$ be the total numeber of modes, i.e.$N=2n+1$, we introduce $X_h^{k,N}(\gamma)= (X_h^k(\gamma))^{N+1} \subset \spaceWN$.
The discretization of problem \eqref{eq:model-problem-restricted-weak-W} reads as follows:
given $g\in \spaceQ'$ , find $(u_{R,h},\Lambda_h)$ in $X_h^k(\Omega) \times X_h^{k,N}(\gamma)$ such that 
\begin{equation} 
    \label{eq:discrete-problem}
    \begin{aligned}
        & \duality{Au_{R,h}, v_h} + \duality{B^T R^T\Lambda_h, v_h} = \duality{f,v_h} && \forall v \in X_h^k(\Omega)\,,\\
        &\duality{R B u_{R,h},w_h} = \duality{R g,w_h} && \forall w_h \in X_h^{k,N}(\gamma)\,.
    \end{aligned}
\end{equation}

Before proceeding, let us define as $\pi_h^k(\Omega)$ a stable and conforming interpolation operator $\space V \mapsto X_h^k(\Omega)$. For example, we choose the Scott-Zhang operator \cite{scott1990finite}. Similarly, let $\pi_h^k(\gamma)$ be the $L^2$ projector $\spaceWZ \mapsto X_h^k(\gamma)$. With little abuse of notation we will omit sometimes to denote the domain of application, if clear from the context. 

For the particular case of section \ref{sec:fourier}, let be $\spaceMN=\mathrm{span}\{\varphi_0(x),\varphi_i^*(x)\}_{i=1}^{n}$ $*=c,s$, the space of basis functions on the cross section of $V$. We use the functions $\varphi_i$ to construct the average operator, $P$ and extension operator $R^T$ according to \eqref{eq:definition-operator-P} and \eqref{eq:definition-operator-R-transpose}, respectively. In particular, we chose the basis functions as the Fourier modes defined in \eqref{eq:fourier-modes}. 

The following Lemma shows that to obtain sufficient conditions for the stability of the method, we need conformity restrictions between the partitions of $\Omega$ and $\gamma$, as well as between the polynomial order $k$ and the number of modes $n$.
To formulate these restrictions, we introduce the domain $V_h$ defined as the collection of all the elements $K \in \mathcal{T}_h^\Omega$ that intersect $V$. Moreover, let $V_{\epsilon + h}$ be the cylinder of radius $\epsilon +h$ with centerline $\gamma$. According to the definition of $h$ we have $V \subset V_h \subset V_{\epsilon+h}$.
{\color{black} This analysis is based on the results obtained in \cite{Kuchta2021558,boulakia:hal-03501521} and partially extends them to a more general framework.}

\begin{lemma}\label{lem:conformity}
    If the basis functions of $\spaceMN$ are as in \eqref{eq:fourier-modes},
    if $k \geq n$ and if the meshes of the 1D and the 3D domains are conforming, namely the faces of the elements $K \in \mathcal{T}_h^\Omega \cap V_h$ are co-planar with the normal plane to $\gamma$ at the vertices of $\mathcal{T}_h^\gamma$, then we have $X_h^k(\gamma) \times \spaceMN \subset X_h^k(V_h)$.
\end{lemma}

\begin{proof}
    Let us take the function $v \in X_h^k(\gamma) \times \spaceMN$ such that
    \begin{equation*}
        v= v_{h,c}^0(x_3) + \sum_{i=1}^n \Big(
        v_{h,i}^c(x_3)\left(\frac{\rho}{\epsilon}\right)^{i}\cos(i\theta) 
        + v_{h,i}^s(x_3)\left(\frac{\rho}{\epsilon}\right)^{i}\sin(i\theta) \Big)
    \end{equation*}
    where $v_{h,0}(x_3),\,v_{h,i}^c(x_3),\,v_{h,i}^s(x_3)  \in X_h^k(\gamma)$.
    
    We observe that the trigonometric polynomials  
    $\left(\frac{\rho}{\epsilon}\right)^{i} \cos(i\theta), 
    \ \left(\frac{\rho}{\epsilon}\right)^{i} \sin(i\theta)$ for $1\leq i \leq n$, we can be written as polynomials of $x_1$ and $x_2$ of degree smaller than or equal to $i$, where $(x_1,x_2)$ satisfies $(x_1, x_2)= (\rho \cos(\theta), \rho \sin(\theta))$. This can be achieved, for example, by using Chebyshev polynomials. 

    Then, owing to the assumption that the faces $\mathcal{T}_h^\Omega$ are co-planar with the normal plane to $\gamma$ at the vertices of $\mathcal{T}_h^\gamma$, we conclude that $v$ is also a piecewise polynomial function of order $k$ conforming to any element $K \in \mathcal{T}_h^\Omega \cap V_h$, namely $v \in X_h^k(V_h)$.
\end{proof}

For the analysis of \eqref{eq:discrete-problem}, we observe that the second equation can be equivalently rewritten as $\duality{B u_{R,h}, R^T w_h} = \duality{g, R^T w_h}$. Then the stability of the problem is related to the inf-sup condition stated in the following Lemma.

\begin{lemma} 
Under assumptions of Lemma \ref{lem:conformity} there exists a constant $$\beta_h^{n,\epsilon} = C \beta_R {\left(1 + \frac{h}{\epsilon}\right)^{-n}},$$ where $C>0$ is a constant independent of $h,n,\epsilon$, such that for all $w_h \in X_h^k(\gamma)$, 
\begin{equation*}
\sup \limits_{v_{h} \in X_h^k(\Omega)} \frac{\duality{B v_h, R^T w_h}}{\|v_{h}\|_{1, \Omega}} \geq \beta_h^{n,\epsilon} \|w_h\|_{-\frac{1}{2}, \gamma}\,.
\end{equation*}
\end{lemma} 

\begin{proof}
Let $q \in \spaceQ^\prime=H^\frac12(\partial V)$ be given. 
Let $x_3$ be the axial coordinate along $\gamma$.
Let $\avg_{0}(x_3)$ and $\avg^c_{i}(x_3),\,\avg^s_{i}(x_3)$, for all $i \geq 1$ be weighted the averaging operators introduced before.
Then, we consider $v^q \in H^1_0(\Omega)$ such that inside $V_{\epsilon +h}$ it is given by, 
\begin{multline}\label{eq:discrete-liftng-e+h}
v^q = \pi_h^k(\gamma) \avg_0q(x_3) 
\\
+ \sum_{i = 1}^{n} \Big( \pi_h^k(\gamma) \avg_i^c q (x_3)\left(\frac{\rho}{\epsilon}\right)^{i} \cos(i\theta) + \pi_h^k(\gamma) \avg_{i}^s q (x_3) \left(\frac{\rho}{\epsilon}\right)^{i} \sin(i\theta) \Big)
\end{multline}
and outside $V_{\epsilon + h}$, $v^q$ is defined as the harmonic lifting in $H^1_0(\Omega)$. 
Let us note that $v^q \in X_h^k(\gamma) \times \spaceMN$ in the cylinder $V_{\epsilon +h}$. Then, owing to Lemma \ref{lem:conformity}, $v^q \in X_h^k(V_h)$.
We now show the following properties of $v^q$. 

First, for any function $w \in X_h^k(\gamma) \times \spaceMN$,
$$w=w_{h,0}(x_3) + \sum_{i=1}^n \Big(w_{h,i}^c(x_3)\left(\frac{\rho}{\epsilon}\right)^{i}\cos(i\theta) + w_{h,i}^s(x_3)\left(\frac{\rho}{\epsilon}\right)^{i}\sin(i\theta) \Big)$$
exploiting that the basis of $\spaceMN$ is orthonormal and that $\pi_h^k(\gamma)$ is an orthogonal projection, we obtain that
\begin{multline*}
    \duality{v^q|_{\partial V}, w} 
    = \int_\gamma \Big[\pi_h^k \avg_0 q w_{h,0}
    +\sum_{i=1}^n \Big(\pi_h^k \avg_i^c q w_{h,i}^c
    + \pi_h^k \avg_i^s q  w_{h,i}^s \Big) \Big] dx_3
    \\
    = \int_\gamma \Big[\avg_0q w_{h,0}
    +\sum_{i=1}^n \Big( \avg_i^c w_{h,i}^c
    + \avg_i^s q  w_{h,i}^s \Big) \Big] dx_3
    = \duality{q, w}\,.
\end{multline*}

Second, equation \eqref{eq:discrete-liftng-e+h} defines a a well posed lifting operator $\partial V_{\epsilon+h} \rightarrow \Omega$. Thanks to the stability of such operator we obtain that the following inequality, with a constant $C$ independent of $\epsilon,n$
\begin{equation*}
    \|v^q\|_{1, \Omega} \leq C \|v^q\|_{\frac{1}{2}, \partial V_{\epsilon + h}}\,.
\end{equation*}
 Moreover taking $v^q$ on $\partial V_{\epsilon + h}$ we get, 
 \begin{equation*}
     v^q = \pi_h^k \avg_0q + \sum_{i = 1}^{n} \pi_h^k \avg_i^c \left(1 + \frac{h}{\epsilon}\right)^{i}\cos(i\theta) + \pi_h^k \avg_i^s \left(1 + \frac{h}{\epsilon}\right)^{i} \sin(i\theta)\,.
 \end{equation*}
As a result we obtain,
\begin{multline*}
\|v^q\|_{\frac{1}{2}, \partial V_{\epsilon + h}} 
\\
\leq C \left[ \|\pi_h^k \avg_0 q\|_{L^2(\gamma)}^2 + \sum_{i = 1}^{n} (1+i) \left(1 + \frac{h}{\epsilon}\right)^{2i} \left( \|\pi_h^k \avg_i^c q\|_{L^2(\gamma)}^2 + \|\pi_h^k \avg_i^s q\|_{L^2(\gamma)}^2 \right) \right]^{\frac{1}{2}} 
\\ 
\leq C \left(1 + \frac{h}{\epsilon}\right)^{n} \left[ \|\pi_h^k \avg_0q\|_{L^2(\gamma)}^2 
+ \sum_{i=1}^{n} (1+ i) \left( \|\pi_h^k \avg_i^c q\|_{L^2(\gamma)}^2 + \|\pi_h^k \avg_i^s q\|_{L^2(\gamma)}^2 \right) \right]^{\frac{1}{2}}
\\
\leq C \left(1 + \frac{h}{\epsilon}\right)^{n} \|q\|_{\frac{1}{2}, \partial V}\,. 
\end{multline*}
Combining the previous inequalities we get,
\begin{equation*}
    \|v^q\|_{1, \Omega} \leq C \left(1 + \frac{h}{\epsilon}\right)^{n} \|q\|_{\frac{1}{2}, \partial V}\,.
\end{equation*}

Finally, let us chose $v_h^q = \pi_h^k(\Omega) v^q$. Owing to the properties of $\pi_h^k(\Omega)$ we have that $v_h^q=v^q$ in $V_h$ and that $\|v_h^q\|_{1,\Omega} \leq C \|v^q\|_{1,\Omega}$, where $C$ is a positive constant independent of $\epsilon,h,n$. Exploiting the surjectivity of the mapping $q \mapsto v^q \mapsto v_h^q$ we then obtain,
\begin{equation*}
    \sup \limits_{v_{h} \in X_h^k(\Omega)} \duality{v_{h}, R^T w_h} = \sup \limits_{q \in H^\frac12(\partial V)} \duality{q, R^T w_h}\,.
\end{equation*}
Recalling that the operator $R^T$ is linear and bounding with constant $\beta_R$, i.e. \eqref{eq:RT-bounding}, the previous inequalities directly imply that
\begin{multline*}
\sup \limits_{v_{h} \in X_h^k(\Omega)}\frac{\duality{v_{h}, R^T w_h}}{\|v_{h}\|_{1, \Omega}}  
\geq \displaystyle C{\left(1 + \frac{h}{\epsilon}\right)^{-n}}  \sup \limits_{q \in H^\frac12(\partial V)} \frac{\duality{q, R^T w_h}}{\|q\|_{\frac{1}{2}, \partial V}}
 \\
= \displaystyle C{\left(1 + \frac{h}{\epsilon}\right)^{-n}}
 \|R^T w_h\|_{-\frac{1}{2}, \partial V} 
 \geq C \beta_R {\left(1 + \frac{h}{\epsilon}\right)^{-n}}
\|w_h\|_{-\frac{1}{2}, \gamma}\,.
\end{multline*}
\end{proof}

Owing to the previous results and according to Theorem 5.2.2 of \cite{Boffi2013}, we obtain the following a-priori error estimates,
\begin{align*}
    \|u_R - u_{R,h}\|_{\spaceV} &\leq 
    \left(\frac{2\|A\|}{\alpha} + \frac{2\|A\|^\frac12\|RB\|}{\alpha^\frac12\beta_h^{n,\epsilon}}\right)\|E_{u,h}\|_{\spaceV}
    + \frac{\|RB\|}{\alpha} \|E_{\Lambda,h}\|_{\spaceWN}\,,
    \\
    \|\Lambda-\Lambda_h\|_{\spaceWN} &\leq
    \left(\frac{2\|A\|^\frac32}{\alpha^\frac12 \beta_h^{n,\epsilon}}
    +\frac{\|A\|\|RB\|}{(\beta_h^{n,\epsilon})^2}\right)\|E_{u,h}\|_{\spaceV}
    +\frac{3\|A\|^\frac12\|RB\|}{\alpha\beta_h^{n,\epsilon}} \|E_{\Lambda,h}\|_{\spaceWN}\,,
\end{align*}
where $\|E_{u,h}\|_{\spaceV}$ $\|E_{\lambda,h}\|_{\spaceWN}$ are the approximation errors of the selected finite element spaces,
\begin{equation*}
    \|E_{u,h}\|_{\spaceV} := \inf\limits_{v_h \in X_h^k(\Omega)} \|u-v_h\|_{\spaceV}\,,
    \quad
    \|E_{\Lambda,h}\|_{\spaceWN} := \inf\limits_{w_h \in X_h^k(\gamma)} \|\Lambda-w_h\|_{\spaceWN}\,.
\end{equation*}
We note that the approximation errors $\|E_{u,h}\|_{\spaceV}$ $\|E_{\lambda,h}\|_{\spaceWN}$ may not scale optimally with respect to $h$ because of the lack of global regularity of the solution on $\Omega$. Since the solution exhibits low regularity across the interface, strategies to mitigate this drawback may include using conforming meshes to the interface or graded meshes in the neighborhood of it.

We conclude this section with a synthesis of the previous analyses about the modeling error and the approximation error.
Putting together the error estimates in $\epsilon,N$ and $h$ and using the triangle inequality we obtain that there exist positive constants $C_i,\,i=1,\ldots,6$ independent of $\epsilon,n,h$, such that
\begin{equation}
    \label{eq:mainerrorestimate1}
    \|u - u_{R,h}\|_{\spaceV} \leq 
     C_1\left(\frac{\epsilon}{\delta}\right)^{n+1}
    +C_2\left(1+\frac{h}{\epsilon}\right)^n\|E_{u,h}\|_{\spaceV}
    +C_3\|E_{\Lambda,h}\|_{\spaceWN},
\end{equation}
\begin{multline}
    \label{eq:mainerrorestimate2}
     \|\lambda-\Lambda_h\|_{\spaceQ} \leq
     C_4\left(\frac{\epsilon}{\delta}\right)^{n+1}
    +C_5\left(1+\frac{h}{\epsilon}\right)^{2n}\|E_{u,h}\|_{\spaceV}
    \\
    +C_6\left(1+\frac{h}{\epsilon}\right)^{n}\|E_{\Lambda,h}\|_{\spaceWN}. 
\end{multline}
These results are particularly interesting because they highlight the interplay between the modeling error and the approximation error and they provide guidelines to balance suitably these two error components of the proposed method.

\section{Numerical examples}\label{sec:numerics}

The main objective of this section is to illustrate by means of selected numerical tests the interplay of the parameters $h,n,\epsilon$, the mesh characteristic size, the dimension of the Lagrange multiplier space and the size of the inclusion, respectively, on the whole approximation error of the proposed approach, formally represented in \eqref{eq:mainerrorestimate1}-\eqref{eq:mainerrorestimate2}.

The provided examples have been implemented using the open source library
\texttt{deal.II}~\cite{ArndtBangerthDavydov-2021-a,dealII94,Sartori2018,Maier2016}. In
particular, we use bi- and tri-linear finite elements for the approximation of
the solution and of the Lagrange multiplier in the full order method.

\subsection{Two dimensional examples}
\label{sec:numerics-dirichlet-2d}

\paragraph{$h$-convergence}
We start by considering the cross section of a cylindrical vessel embedded in a cubic domain, where Dirichlet boundary conditions are applied on the boundary of the vessel, and some manufactured boundary conditions are imposed on the boundary of the cube to recover a known manufactured solution.

The corresponding two dimensional setting we consider is that of a square $\Omega = [-1,1]^2$, with a circular inclusion $V \equiv B_\epsilon(0)$, where $\epsilon$ is the radius of the cross section.

In particular we impose boundary conditions so that the resulting solution is harmonic in $\Omega\setminus\Gamma$, i.e., 
\begin{equation}
    \label{eq:manufactured-solution-D1}
    \begin{aligned}
        -\Delta u & =  0 && \mathrm{in} \ \Omega\setminus \Gamma \equiv [-1,1]^2\setminus \partial B_\epsilon(0),\\
        u &= -\ln(x^2+y^2)/2  && \mathrm{on} \ \Gamma \cup \partial \Omega, \\
        \lambda & = \frac{1}{\epsilon} = [\nabla u]\cdot n && \mathrm{on} \  \Gamma \equiv \partial B_{\epsilon}(0).\\
    \end{aligned}
\end{equation}

For this particular problem, the manufactured solution is a truncated fundamental solution, and it represents one of the simplest examples of solutions around a circular inclusion (or a cylindrical one in three dimensions). The major characteristic of this solution is that it is harmonic in the entire domain $\Omega\setminus\Gamma$, with a constant jump on the gradient on $\Gamma$ which increases as the radius of the inclusion decreases, and it has a constant value $u = -\ln(\epsilon)$ on the boundary of the inclusion.

\begin{figure}[t]
    \includegraphics[width=0.45\textwidth]{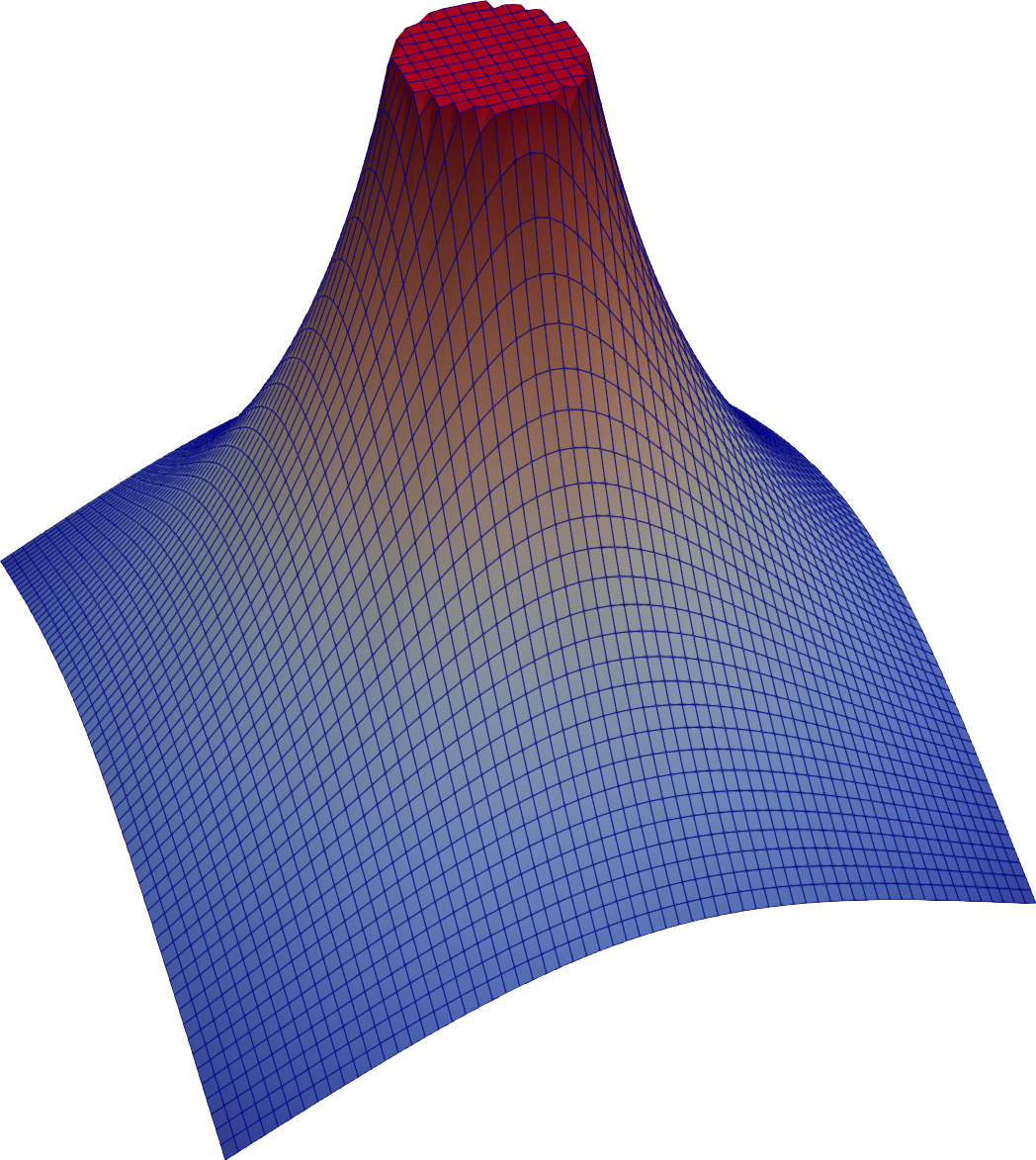}
    \hfill
    \includegraphics[width=0.45\textwidth]{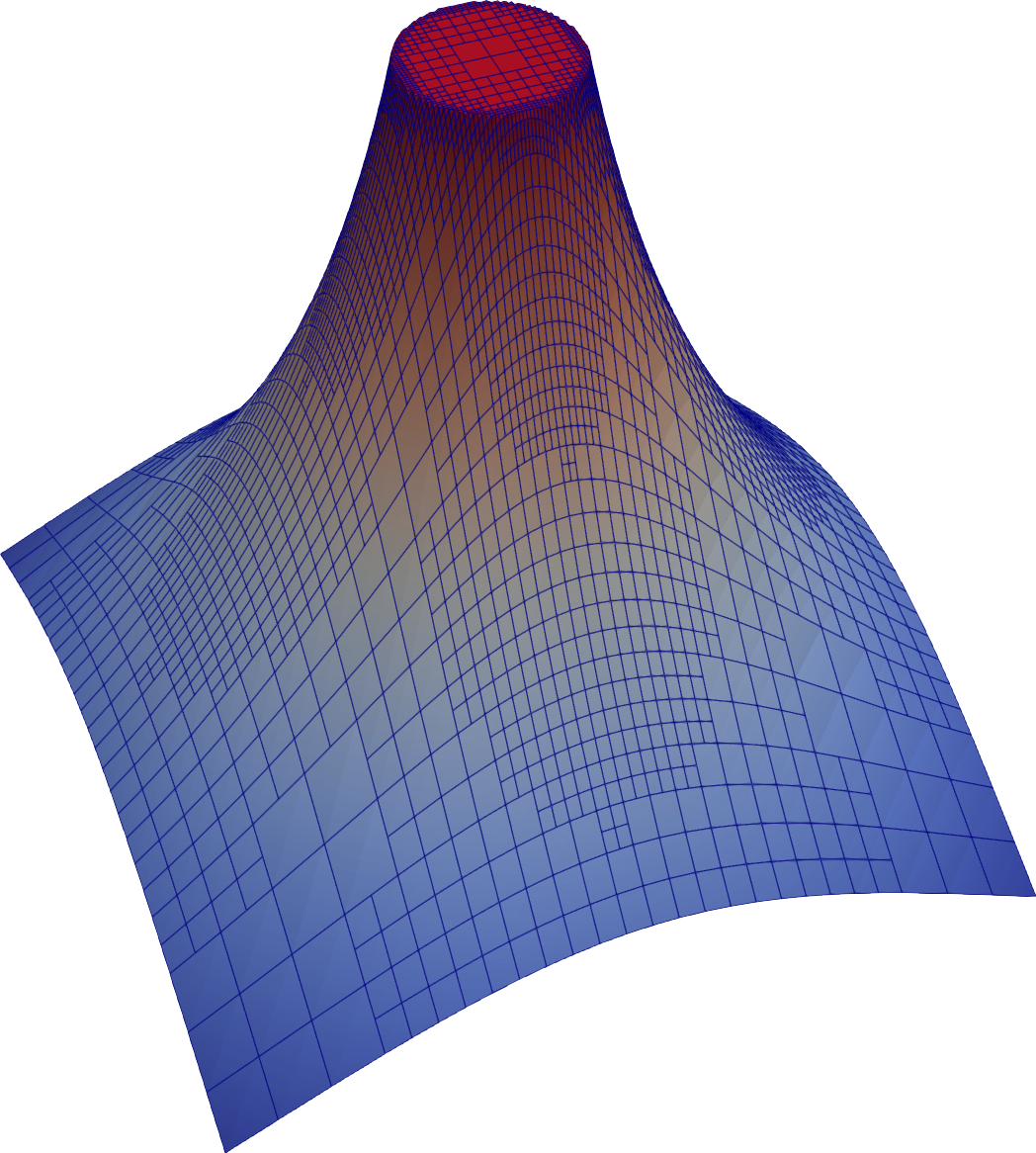}
    \caption{Example for the numerical solution with global refinements for problem \eqref{eq:manufactured-solution-D1} (left) and with local refinements (right), using a single Fourier mode and $r=0.2$.}
    \label{fig:fundamental_D1}
\end{figure}

Such characteristics make this one of the simplest manufactured solution since both the solution $u$ and the Lagrange multiplier $\lambda$ are constant on $\Gamma$, allowing one to test the exact solution of the problem when using just one Fourier mode. Precisely, for this articular case the error on the right hand side of \eqref{eq:error-reduced-continuous-new} is null. For this reasons, this is the ideal case to test the $h$-convergence of the method, addressed in Figure~\ref{fig:fundamental_D1_error}.


\begin{figure}
\includegraphics{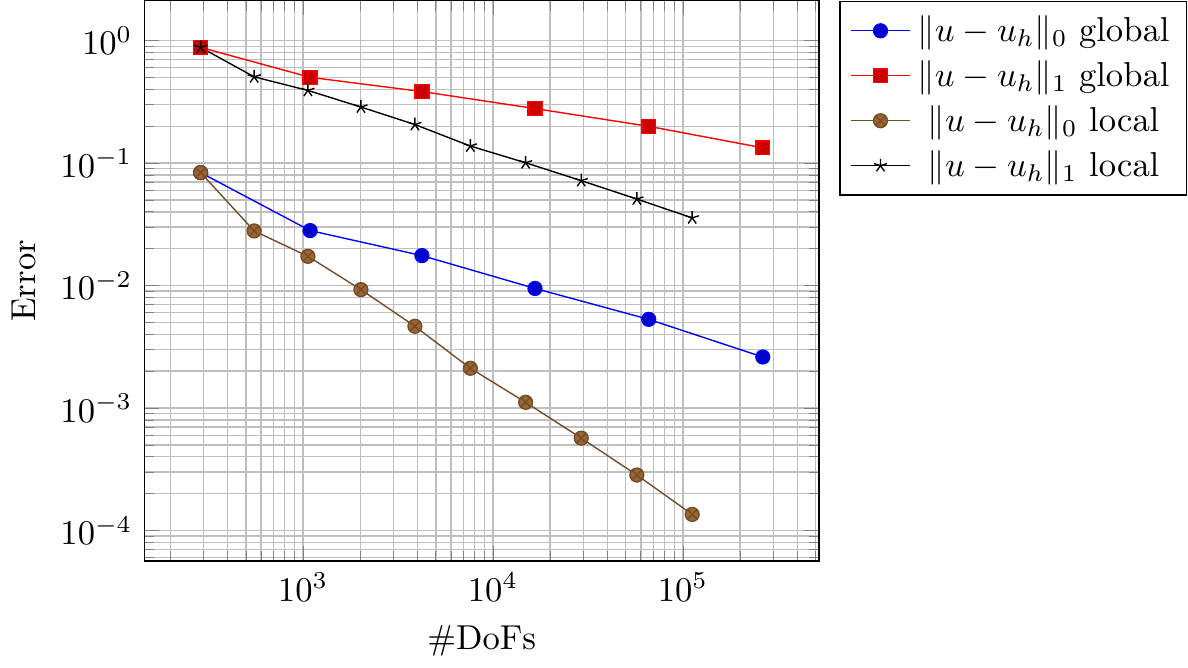}
\caption{Error of the numerical solution with global refinements (left) and with local refinements (right), using a single Fourier mode for problem~\eqref{eq:manufactured-solution-D1}, when $r=0.2$.}
\label{fig:fundamental_D1_error}
\end{figure}

The error of the numerical solution with global refinements (left) and with adaptive local refinements (right), using a single Fourier mode is shown in Figure~\ref{fig:fundamental_D1_error}. The error is computed in the $L^2$ and $H^1$ norms. In the global refinement case, the orders of convergence are $1.5$ for the $L^2$ error, and $0.5$ for the $H^1$ error, as expected from the global regularity of the solution, even though optimal error estimates may be recovered by measuring the error with weighted norms~\cite{HeltaiRotundo-2019-a}. 

\paragraph{The role of $n,N$}
Adaptive finite element methods offer optimal error convergence rates of $2$ for the $L^2$ error and $1$ for the $H^1$ error~\cite{Cohen2012}, providing an excellent combination of accuracy and efficiency.

Although Figure~\ref{fig:fundamental_D1_error} seems to suggest that the error of the numerical solution is acceptable even with just one Fourier mode, one should carefully notice that this is generally not the case in realistic scenarios. An example about the importance of the parameter $n$ is given by the solution of the following manufactured problem:
\begin{equation}
    \label{eq:manufactured-solution-D2}
    \begin{aligned}
        -\Delta u & =  0 && \mathrm{in} \ \Omega\setminus \Gamma \equiv [-1,1]^2\setminus \partial B_\epsilon(0),\\
        u &= x \frac{\epsilon^2}{x^2+y^2}  && \mathrm{on} \ \Gamma \cup \partial \Omega, \\
        \lambda & = \frac{x}{\epsilon} = [\nabla u]\cdot n && \mathrm{on} \  \Gamma \equiv \partial B_\epsilon(0),\\
    \end{aligned}
\end{equation}
which has exact solution equal to $u(x) = x$ inside the inclusion $V$, and $u(x) = x \epsilon^2/(x^2+y^2)$ outside of the inclusion (see Figure~\ref{fig:fundamental_D2}, right).

\begin{figure}
    \includegraphics[width=.45\textwidth]{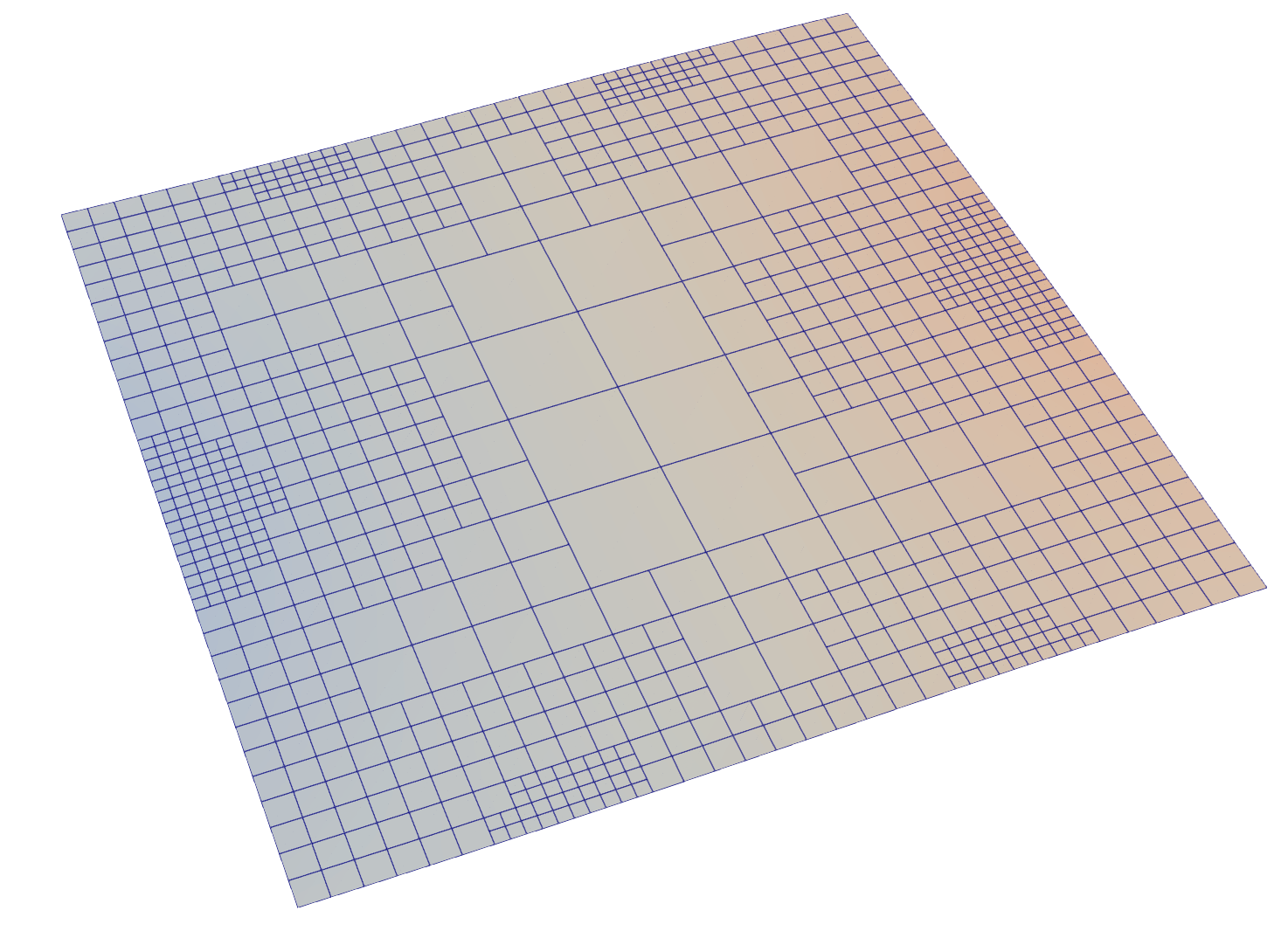}
    \hfill
    \includegraphics[width=.45\textwidth]{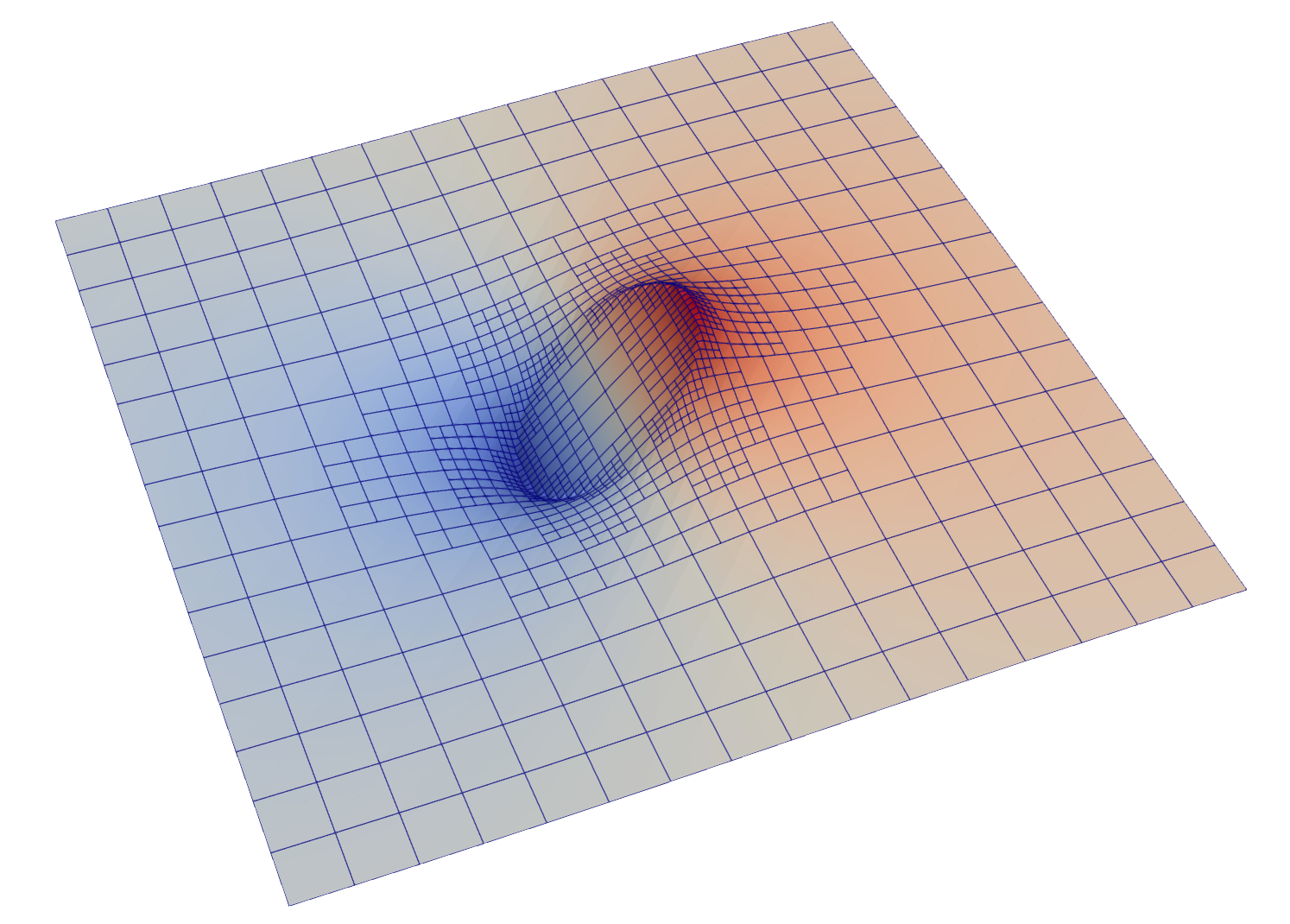}
    \caption{Example for a numerical solution where the dimensionality reduction error is significant, using a single Fourier mode (left) and two fourier modes (right), for problem \eqref{eq:manufactured-solution-D2}, when $r=0.2$.}
    \label{fig:fundamental_D2}
\end{figure}

When using a Fourier extension with a single Fourier mode, the numerical scheme fails to capture the solution (which has zero average on the boundary of the vessel) (see Figure~\ref{fig:fundamental_D2} left). At least two Fourier modes (three cylindrical harmonics in total) are required to obtain an acceptable solution (see Figure~\ref{fig:fundamental_D2} right). 

This example is extremely relevant and illustrative for the applications of this method.
In fact, the solution reminds of a particle exposed to a shear flow field. On one side the inclusion is subject to forces in one direction and on the other side to the opposite direction. 
This test case clearly shows that the approximation of the interface conditions with only one Fourier mode (i.e. $n=0,N=1$) would not be sufficient to model the rotation of the particle. The approximation based on three modes ($n=1,N=3$) would completely resolve this issue.

\begin{figure}
    \includegraphics[width=.45\textwidth]{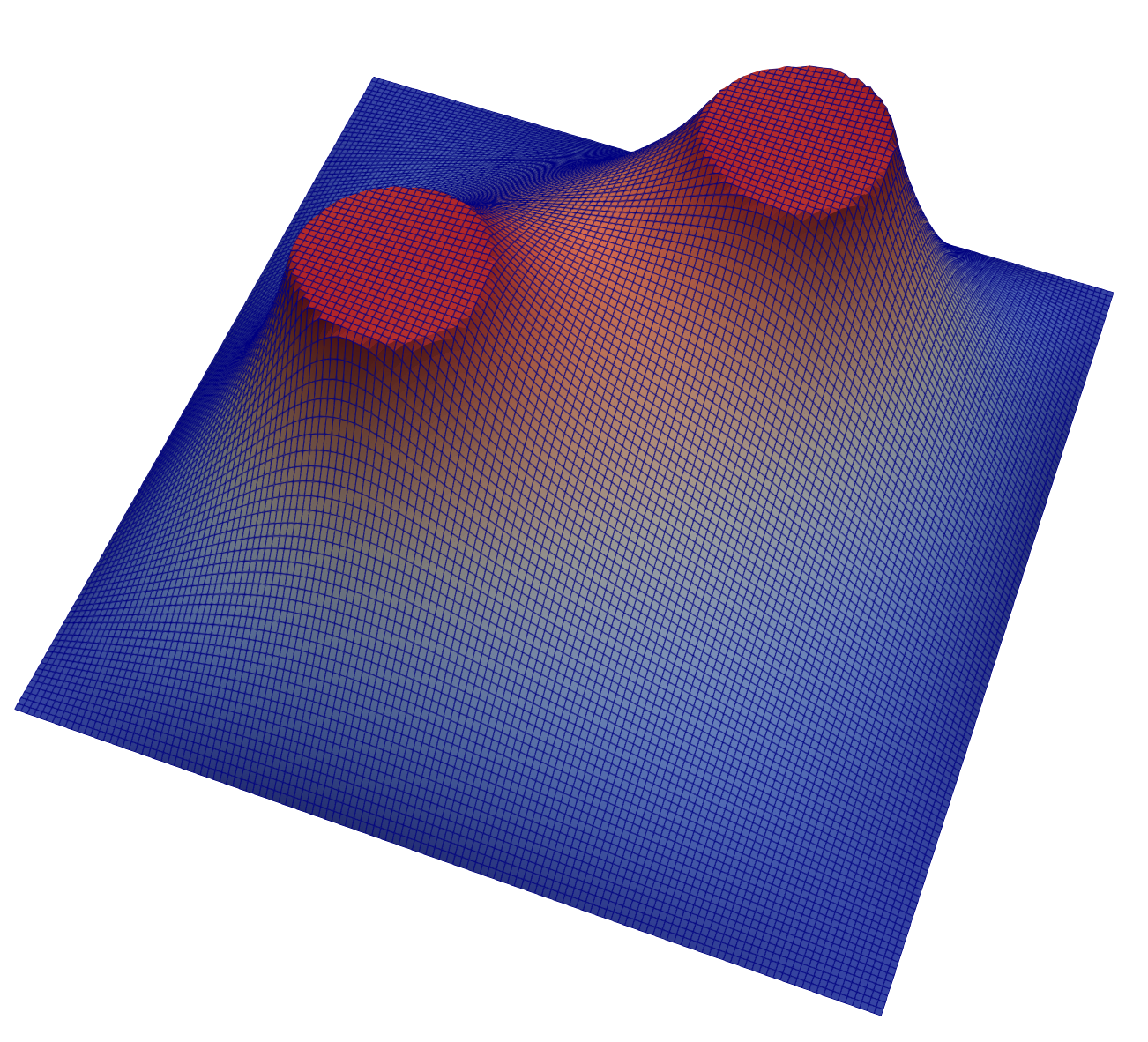}
    \hfill
    \includegraphics[width=.45\textwidth]{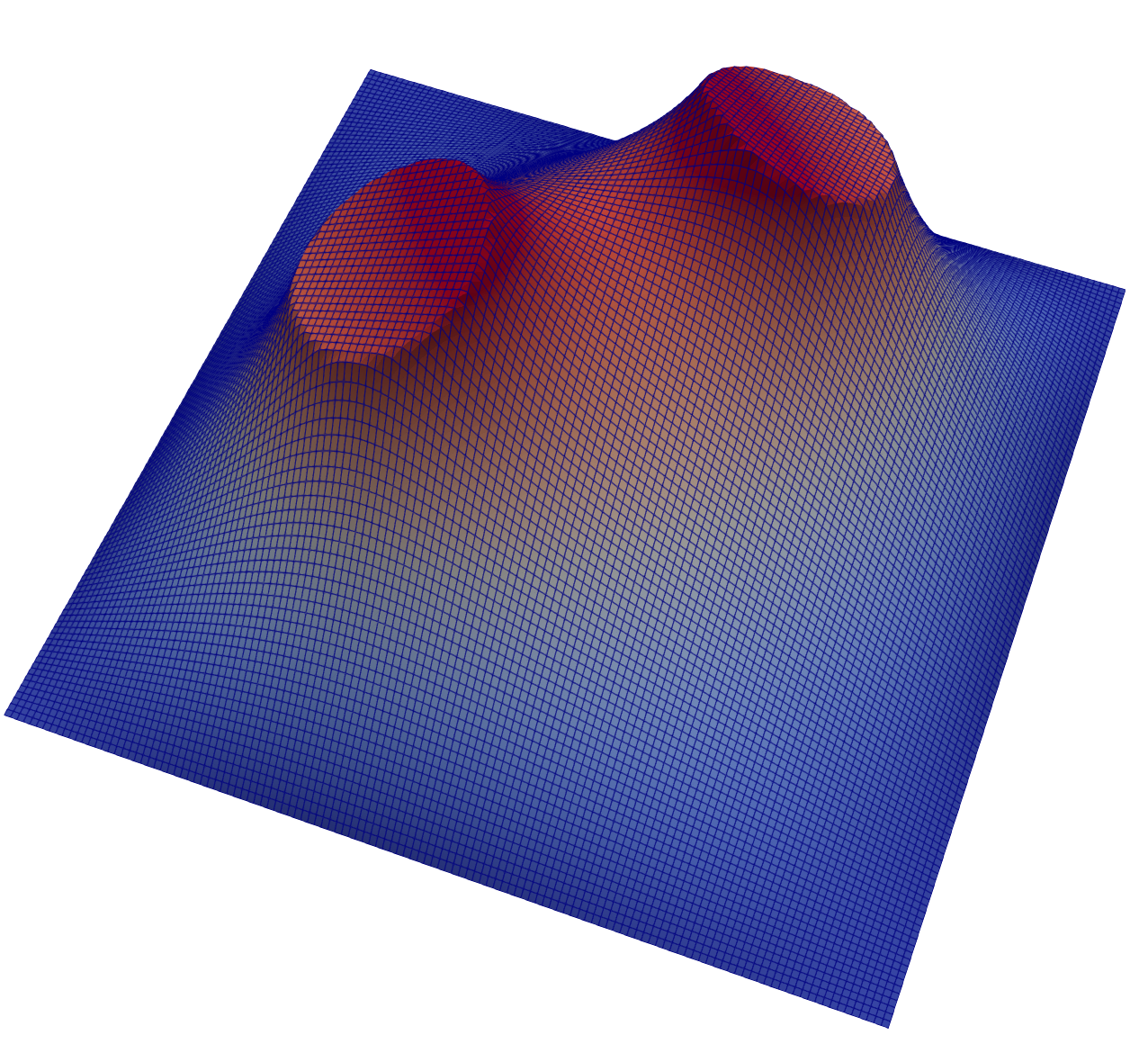}

    \includegraphics[width=.45\textwidth]{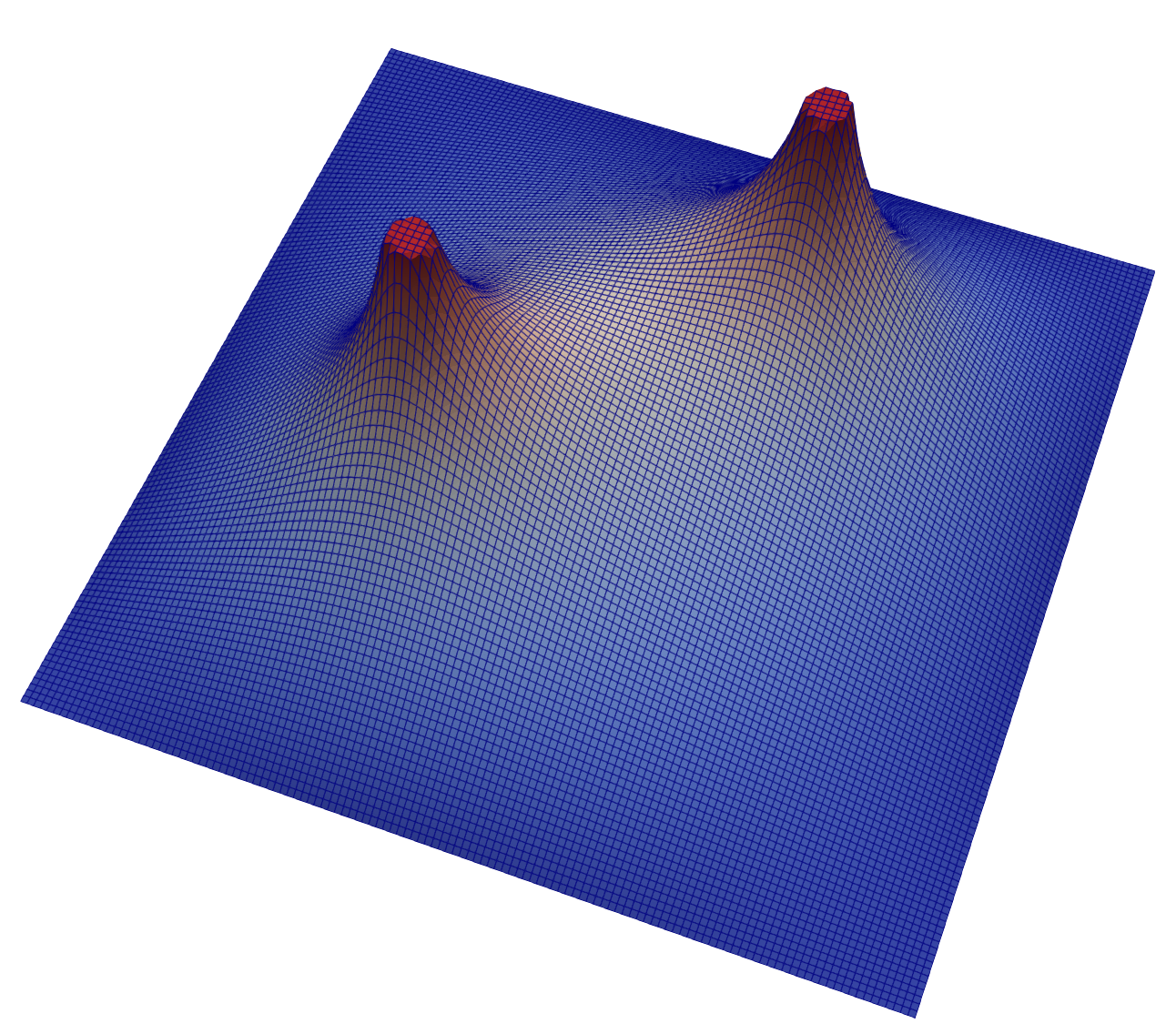}
    \hfill
    \includegraphics[width=.45\textwidth]{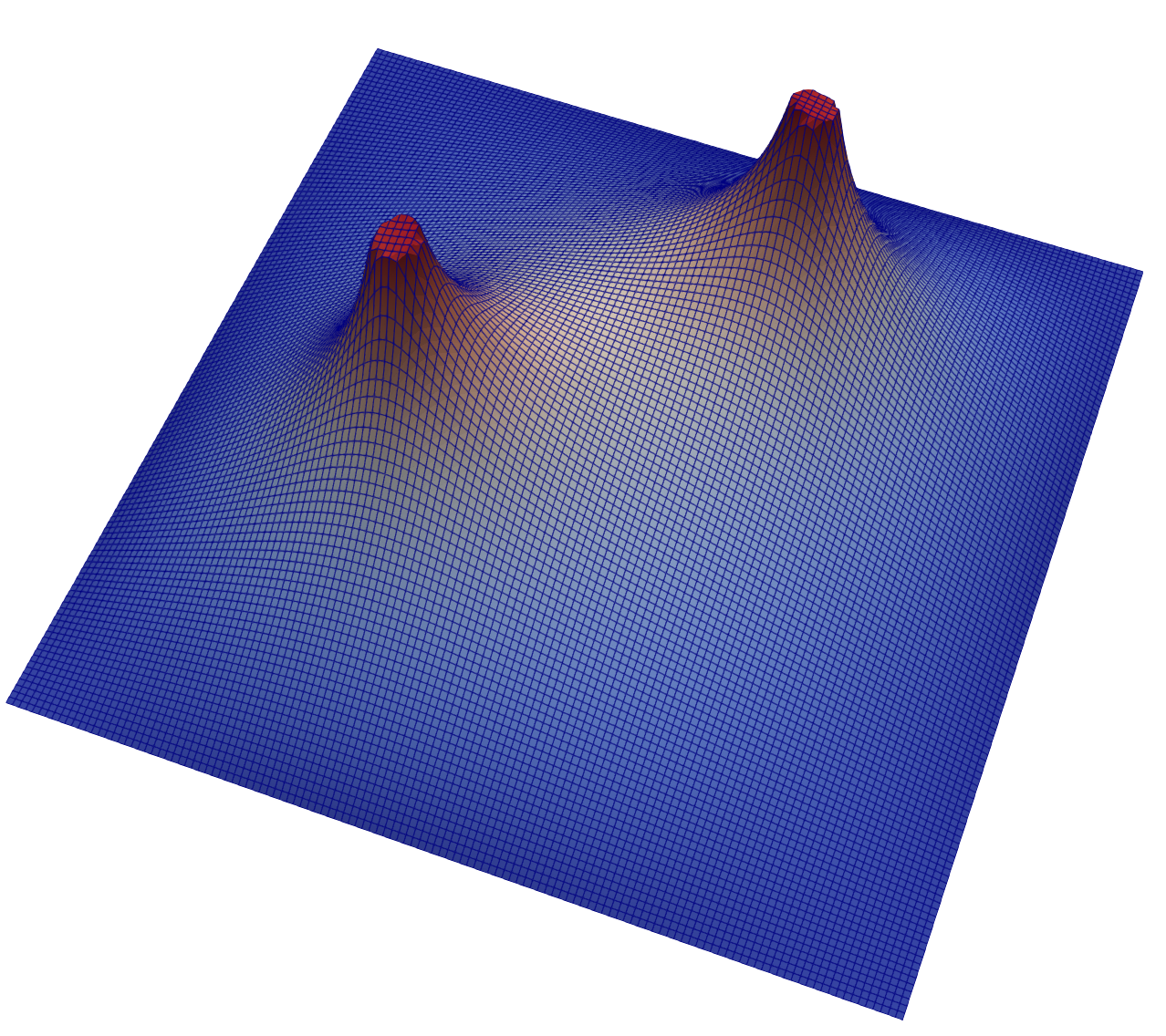}

    \includegraphics[width=.45\textwidth]{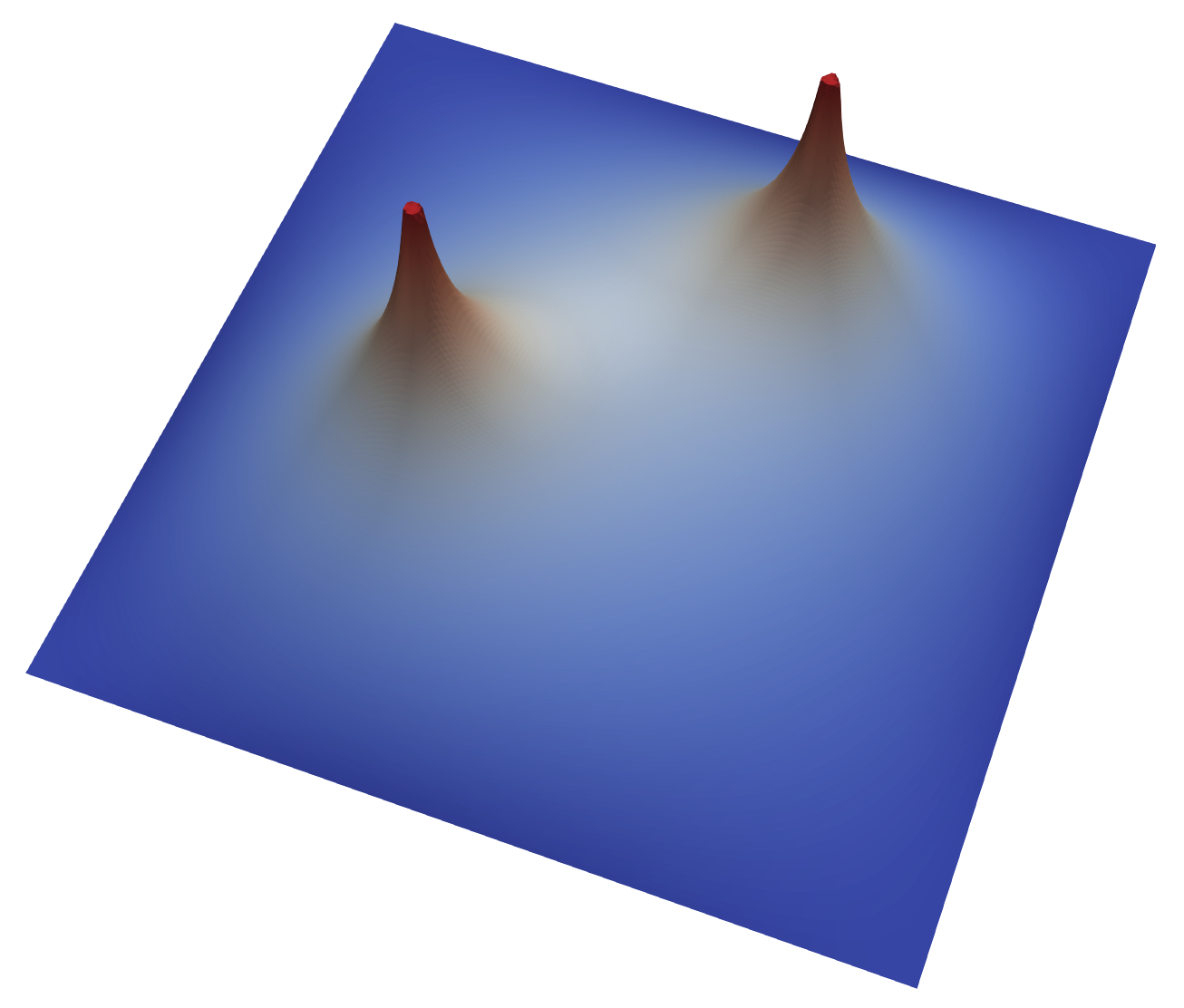}
    \hfill
    \includegraphics[width=.45\textwidth]{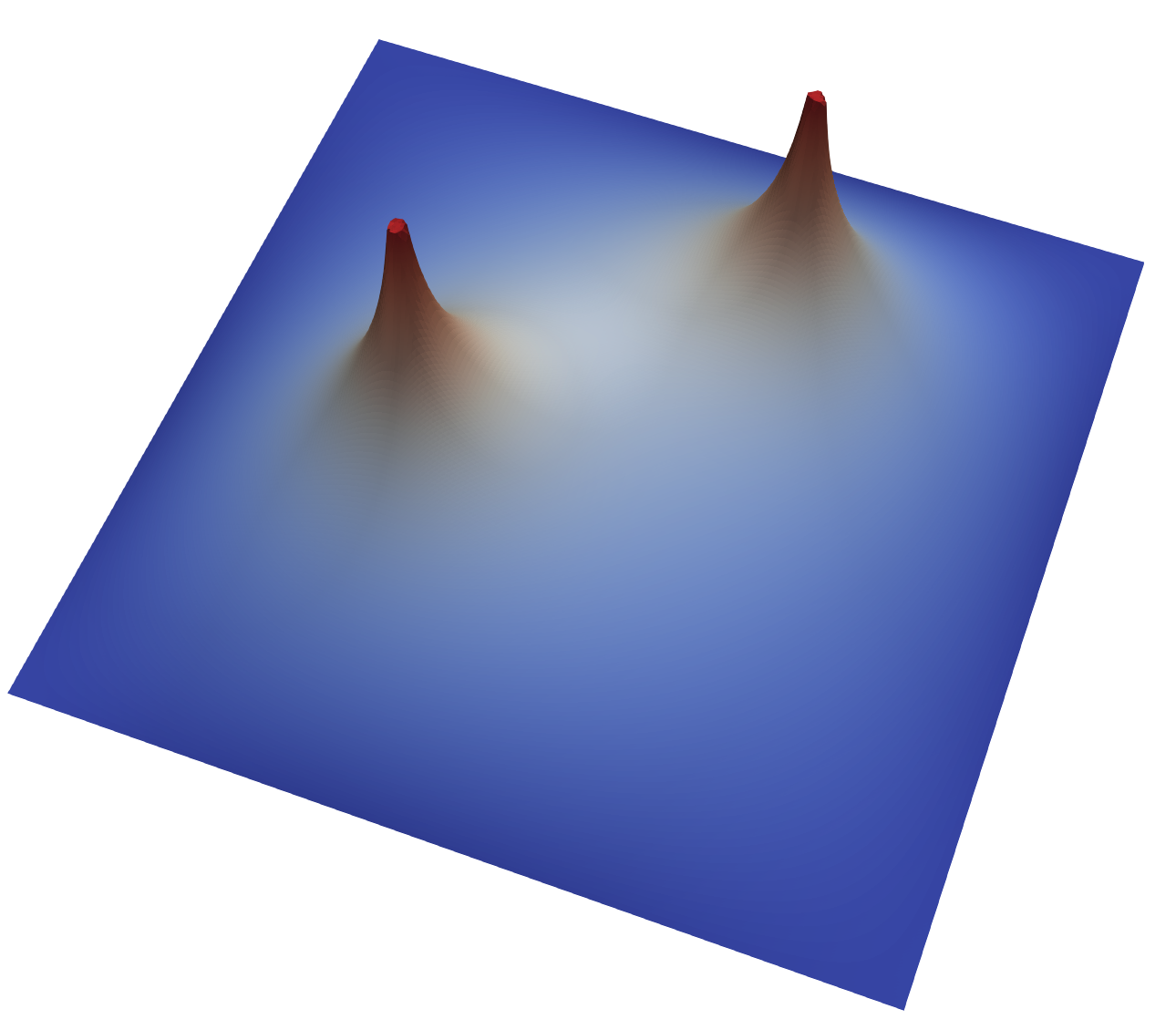}

    \caption{two inclusions with constant boundary condition. Full Lagrange multiplier solution (left) VS single Fourier mode solution (right), for decreasing inclusion radius $r=0.2$ (top), $r=0.1$ (center), and $r=0.05$ (bottom).}
    \label{fig:two_inclusions}
\end{figure}

\paragraph{$\epsilon$-convergence}
In Figure~\ref{fig:two_inclusions}, where we replicate in spirit the example of problem~\eqref{eq:manufactured-solution-D1} for two inclusions with different radii. We impose $u=1$ on the boundary of the inclusions, and $u=0$ on the boundary of the domain $\Omega$. This setting reflects more closely what happens in realistic scenarios, where the boundary conditions on the vessels are dictated by the solution of auxiliary problems solved in dimension one, and extended (constantly) on $\Gamma$.

Recalling that the Lagrange multiplier represents here the jump of the gradients at the interface, we see that while one Fourier mode (i.e. $n=0,N=1)$ would suffice to represent exactly the solution, it fails to capture the Lagrange multiplier when the radius of the vessel is non-negligible, leading to a solution where only the average is equal to the desired value on $\Gamma$. For this particular case, a solution obtained with five Fourier modes (not shown here) is indistinguishable from the full order solution (see Figure~\ref{fig:two_inclusions} top left). For smaller inclusions (see Figure~\ref{fig:two_inclusions} center right and bottom right), the solution obtained with a single Fourier mode is significantly closer to the full order solution.

\paragraph{The combined effect of $h,n,\epsilon$}
The interplay of the three parameters $h,n,\epsilon$ is studied rigorously in Figure~\ref{fig:fundamental_D3_error} and~\ref{fig:fundamental_D3_error_ref} for a single inclusion of variable size with a non trivial solution. In particular, we solve the following problem:
\begin{equation}
    \label{eq:manufactured-solution-D3}
    \begin{aligned}
        -\Delta u =&  0 && \mathrm{in} \ \Omega\setminus \Gamma \equiv [-1,1]^2\setminus \partial B_{\epsilon}(0),\\
        u =& 2 x^{3} - x^{2} - 6 x y^{2} + x + y^{2} + 1 && \mathrm{on} \ \Gamma \equiv \partial B_{\epsilon}(0), \\
        u = &  \frac{2 \epsilon^{6} x \left(x^{2} - 3 y^{2}\right)}{\left(x^{2} + y^{2}\right)^{3}} + \frac{\epsilon^{4} \left(- x^{2} + y^{2}\right)}{\left(x^{2} + y^{2}\right)^{2}} \\
        &+ \frac{\epsilon^{2} x}{x^{2} + y^{2}}+ \frac{\log{\left(x^{2} + y^{2} \right)}}{2 \log{\left(\epsilon \right)}} && \mathrm{on} \ \partial \Omega, \\
    \end{aligned}
\end{equation}
where the expression of the boundary conditions on $\Gamma$ and on $\partial \Omega$ coincide with non trivial harmonic solutions of both the interior and the exterior problems. We solve the problem for $\epsilon=\{0.2, 0.1, 0.05, 0.025\}$ with varying mesh size $h=2/(2^i)$, for $i=\{6,7,8,9,10\}$, and we compare the $L^2$ and $H^1$ errors (with respect to the available exact solution) in the solution obtained with a variable number of Fourier modes $N=1,3,5,7,9, (n=0,1,2,3,4)$. The results are shown in Figures~\ref{fig:fundamental_D3_error} and~\ref{fig:fundamental_D3_error_ref}, where we (partially) confirm numerically the estimate presented in the final error estimates \eqref{eq:mainerrorestimate1}-\eqref{eq:mainerrorestimate2}.



\begin{figure}[ht]
    \includegraphics{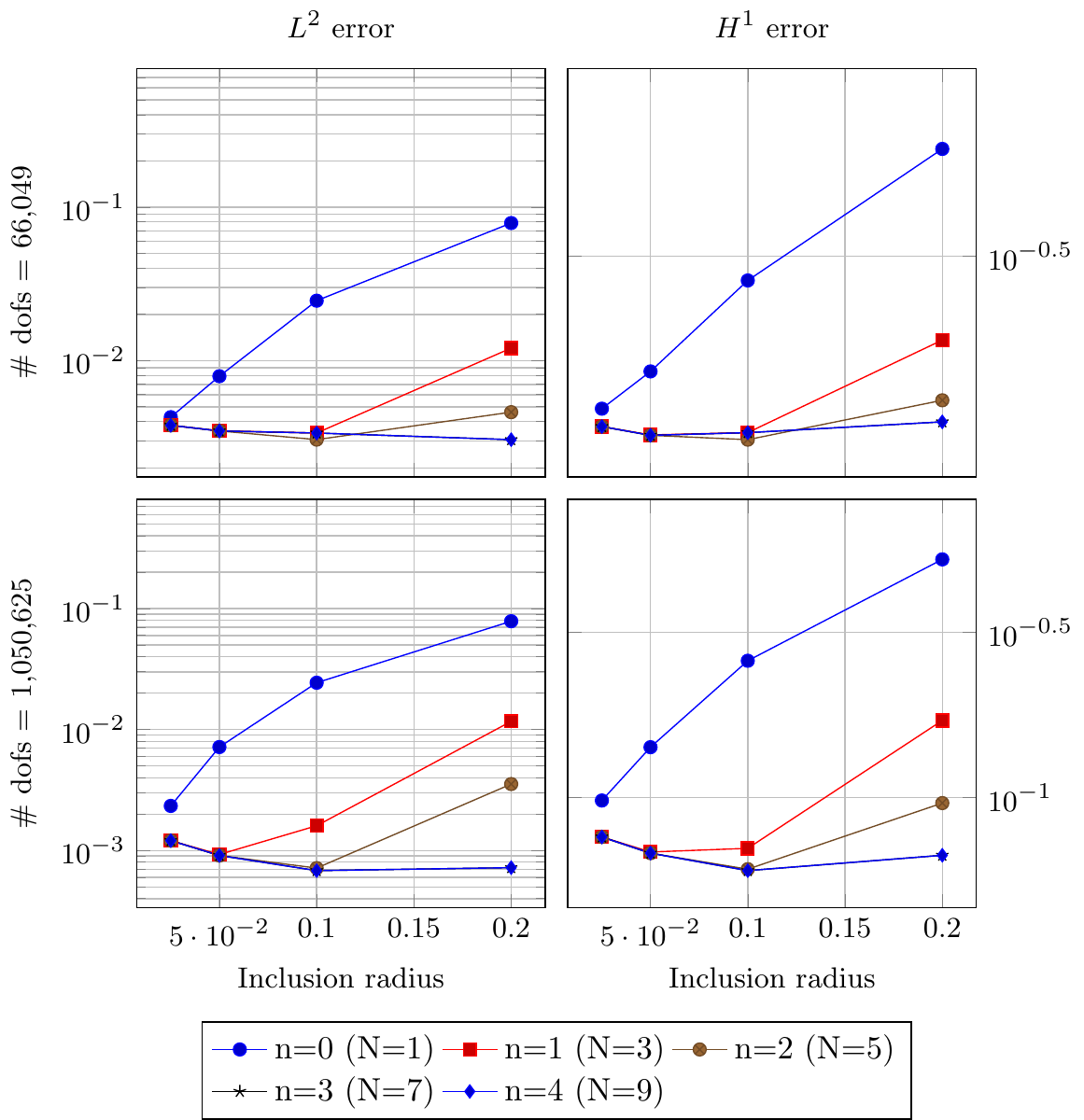}

    \caption{$L^2$ error (left column) and $H^1$ error (right column) in the solution of problem~\eqref{eq:manufactured-solution-D3} for different values of the radius $r$ and different number of Fourier modes $N$, in two different refinements: \# dofs = 66,049 (top) and \# dofs = 1,050,625 (bottom).}
    \label{fig:fundamental_D3_error}
\end{figure}





\begin{figure}[!ht]
    \includegraphics{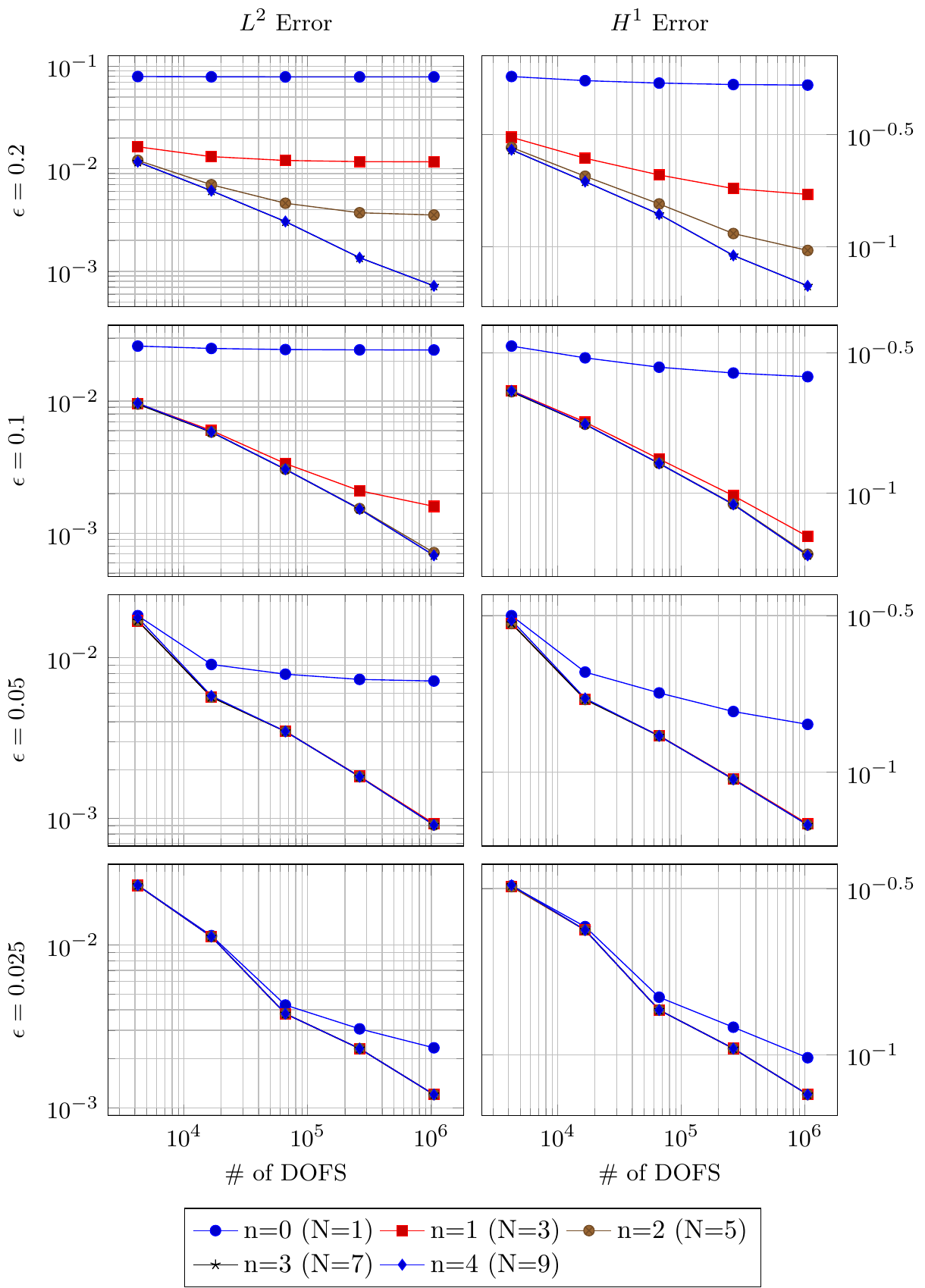}

    \caption{$L^2$ error (left column) and $H^1$ error (right column) with respect to the number of degrees of freedom in the solution of problem~\eqref{eq:manufactured-solution-D3} for different values of the radius $r$ and different number of Fourier modes $N$.}
    \label{fig:fundamental_D3_error_ref} 
\end{figure}   

\begin{figure}
    \includegraphics[width=.3\textwidth]{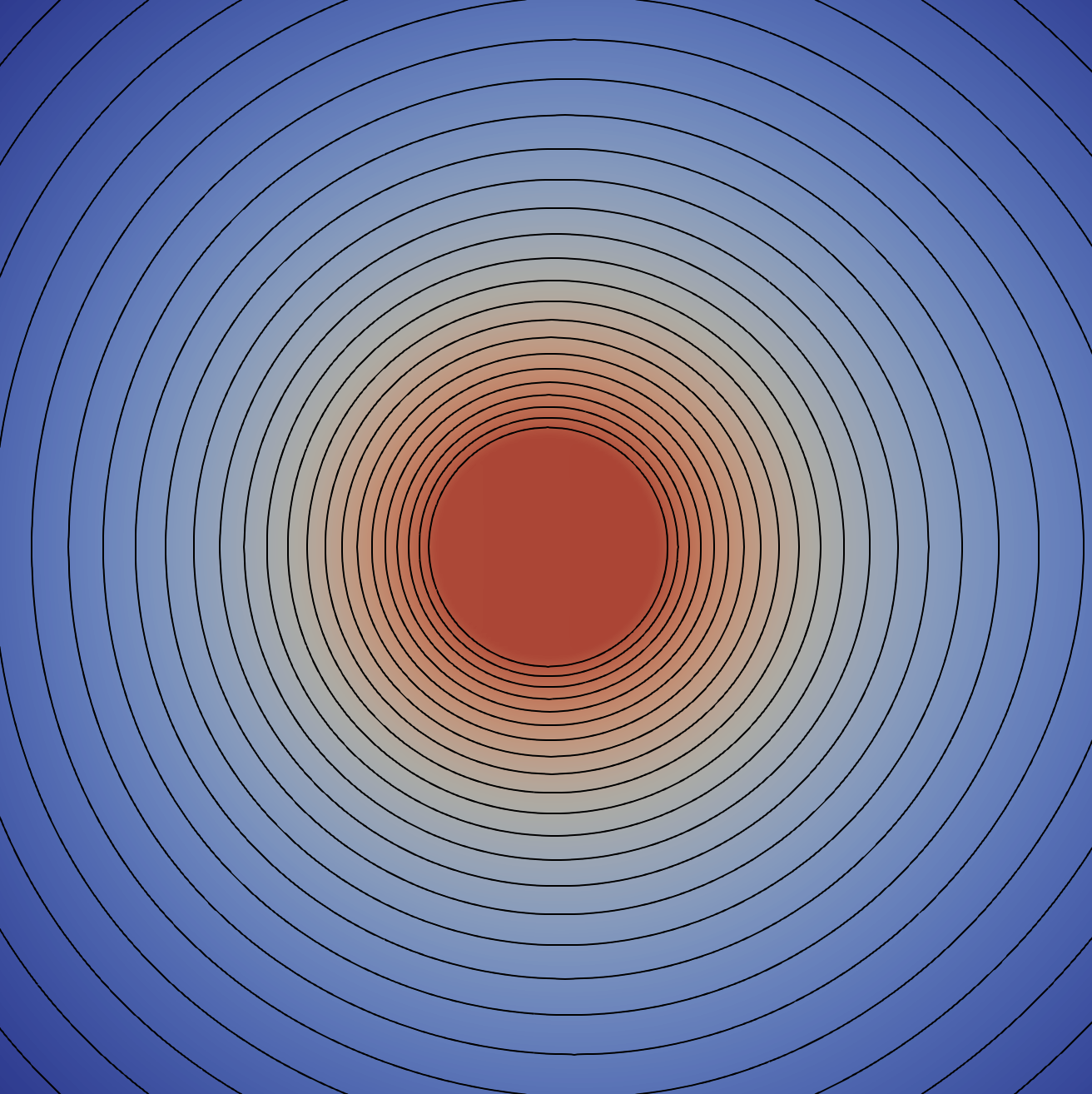}
    \hfill 
    \includegraphics[width=.3\textwidth]{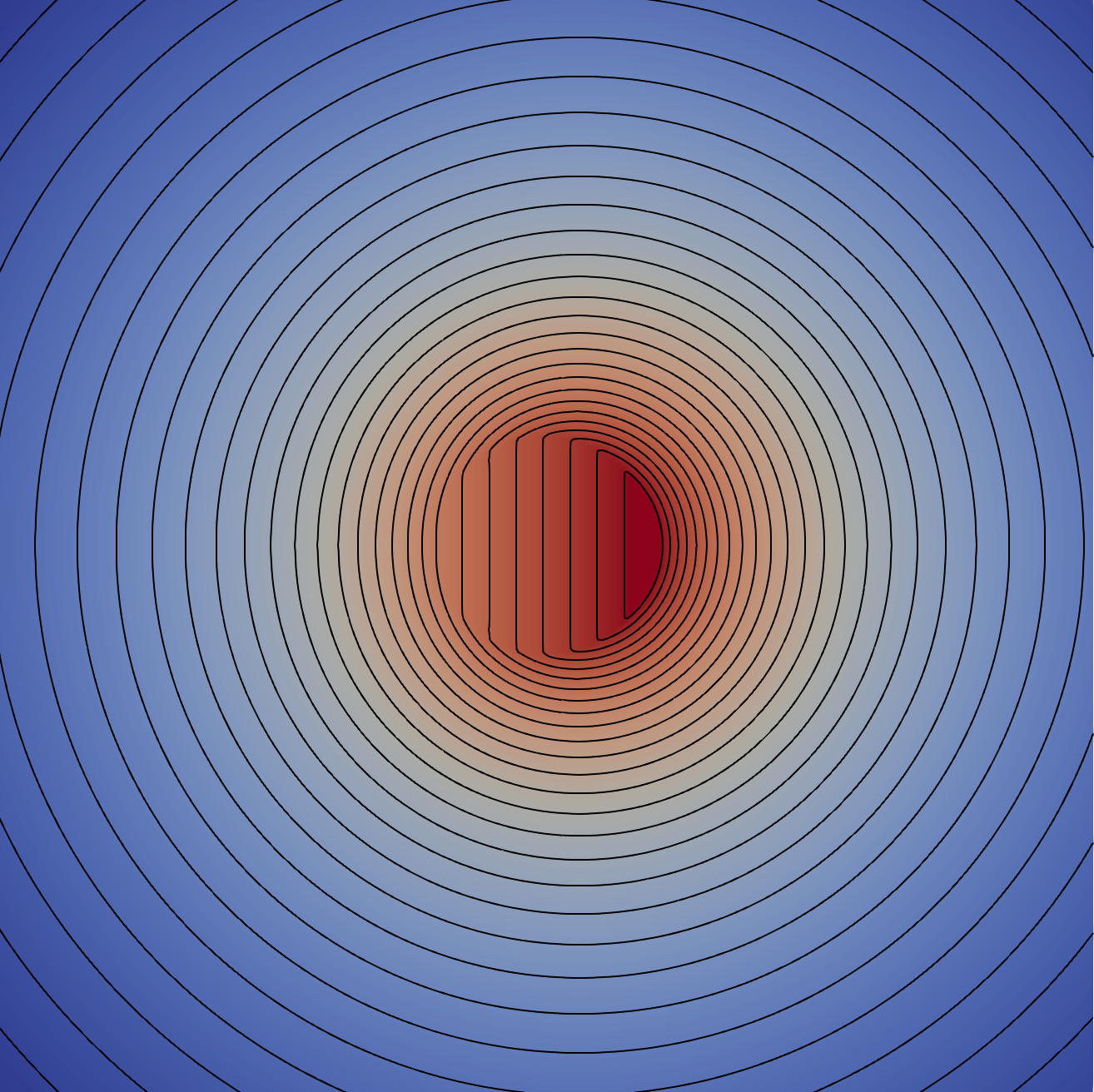}
    \hfill 
    \includegraphics[width=.3\textwidth]{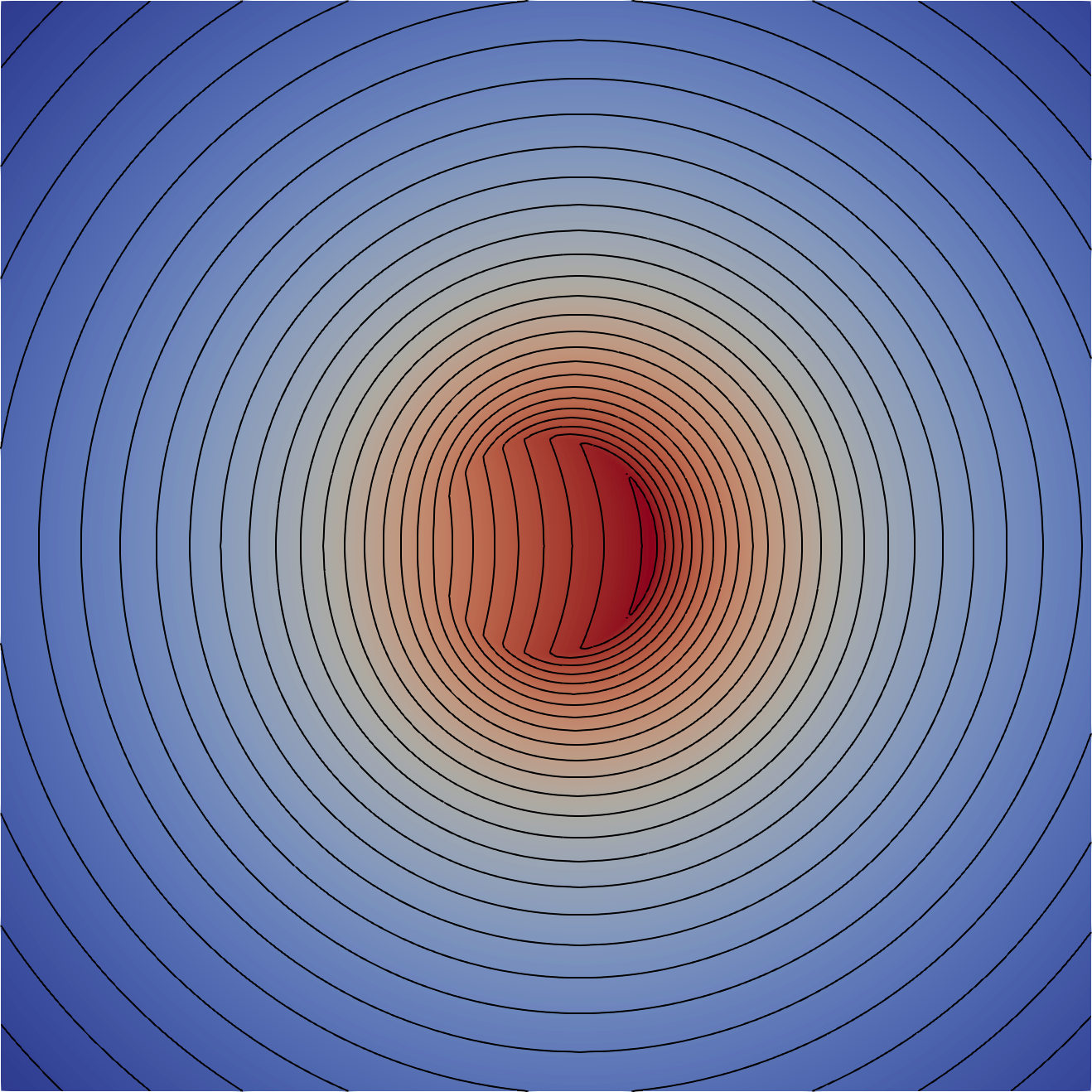}

    \vspace{.05\textwidth}
    
    \includegraphics[width=.3\textwidth]{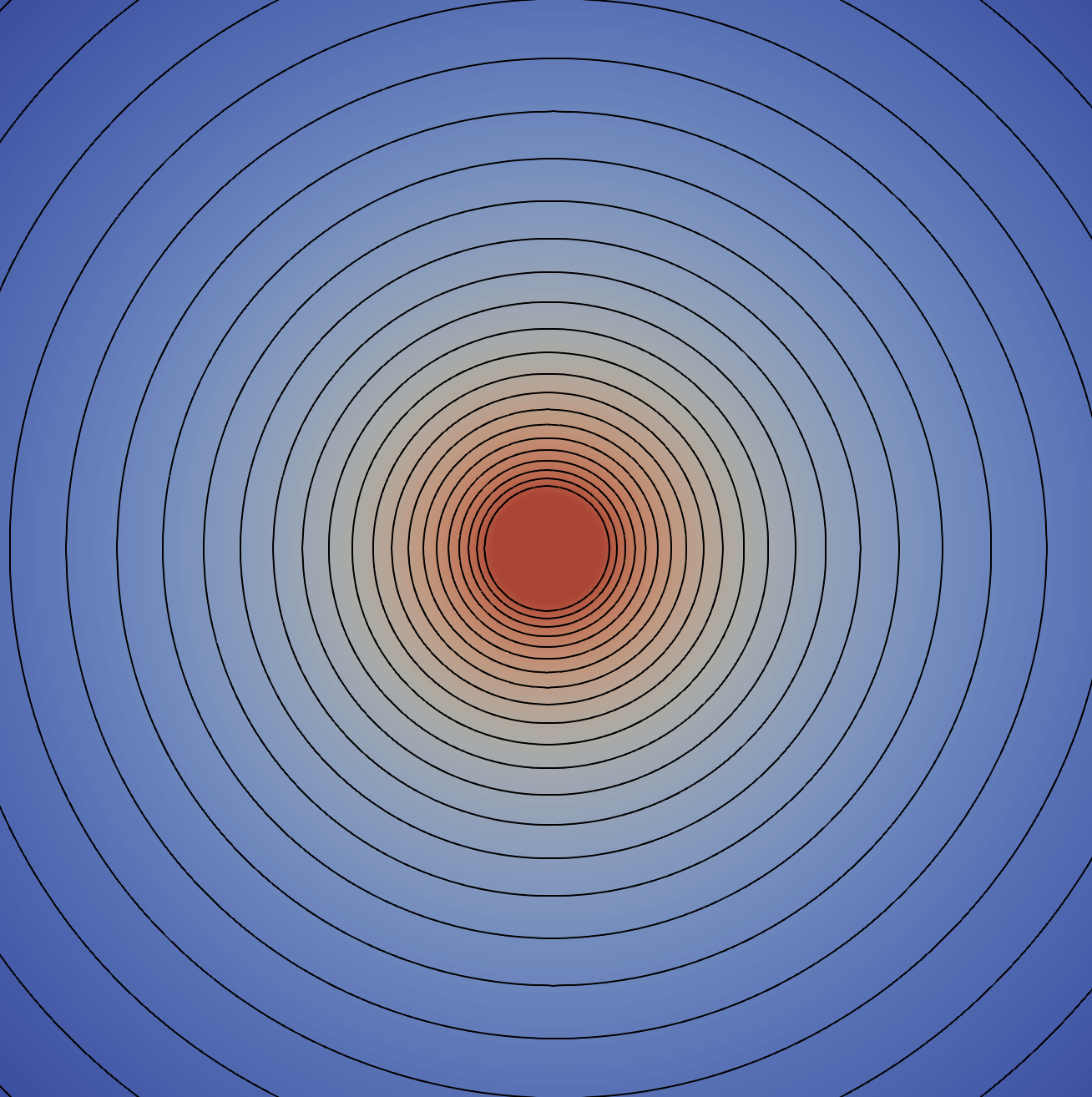}
    \hfill 
    \includegraphics[width=.3\textwidth]{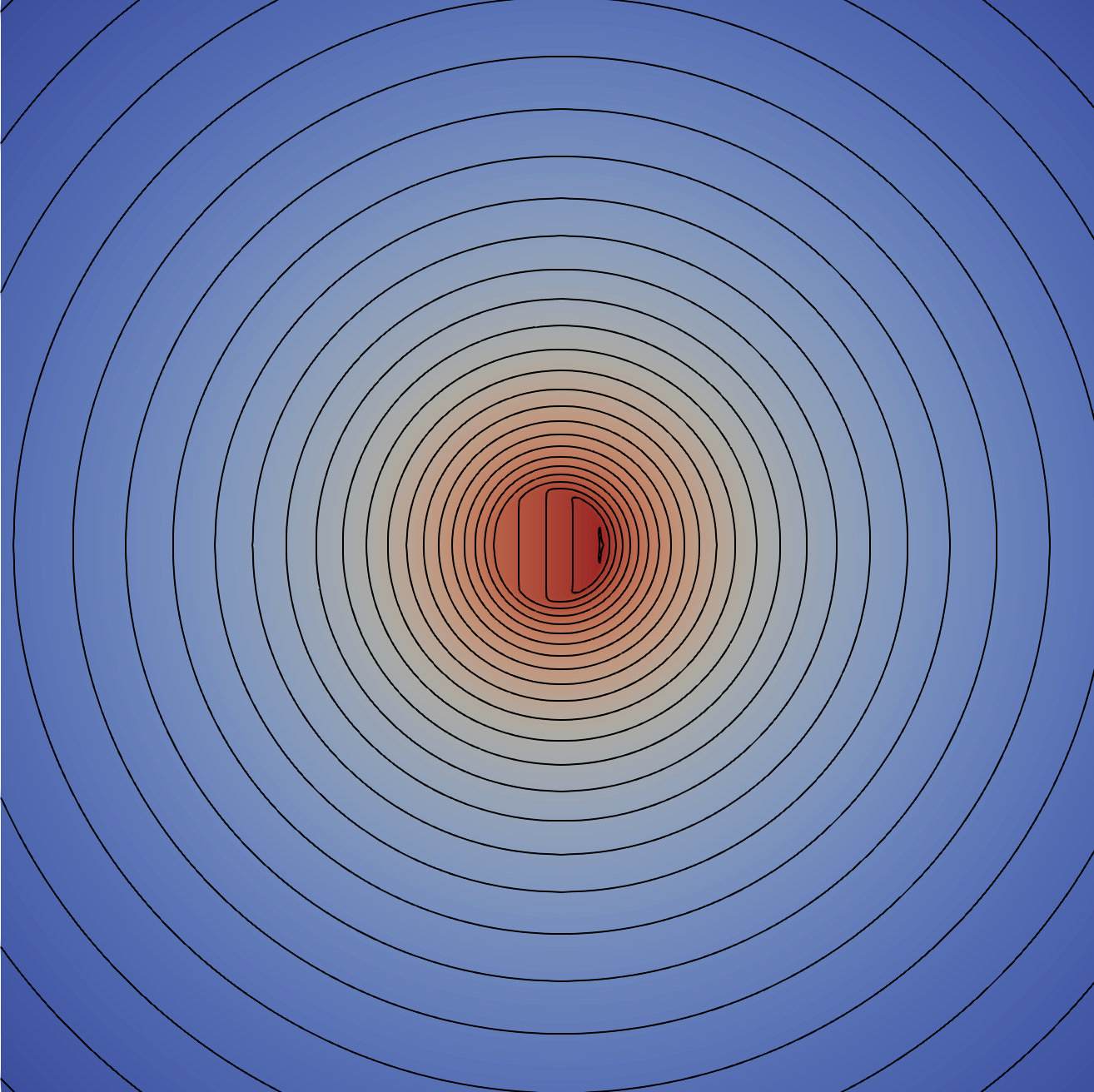}
    \hfill 
    \includegraphics[width=.3\textwidth]{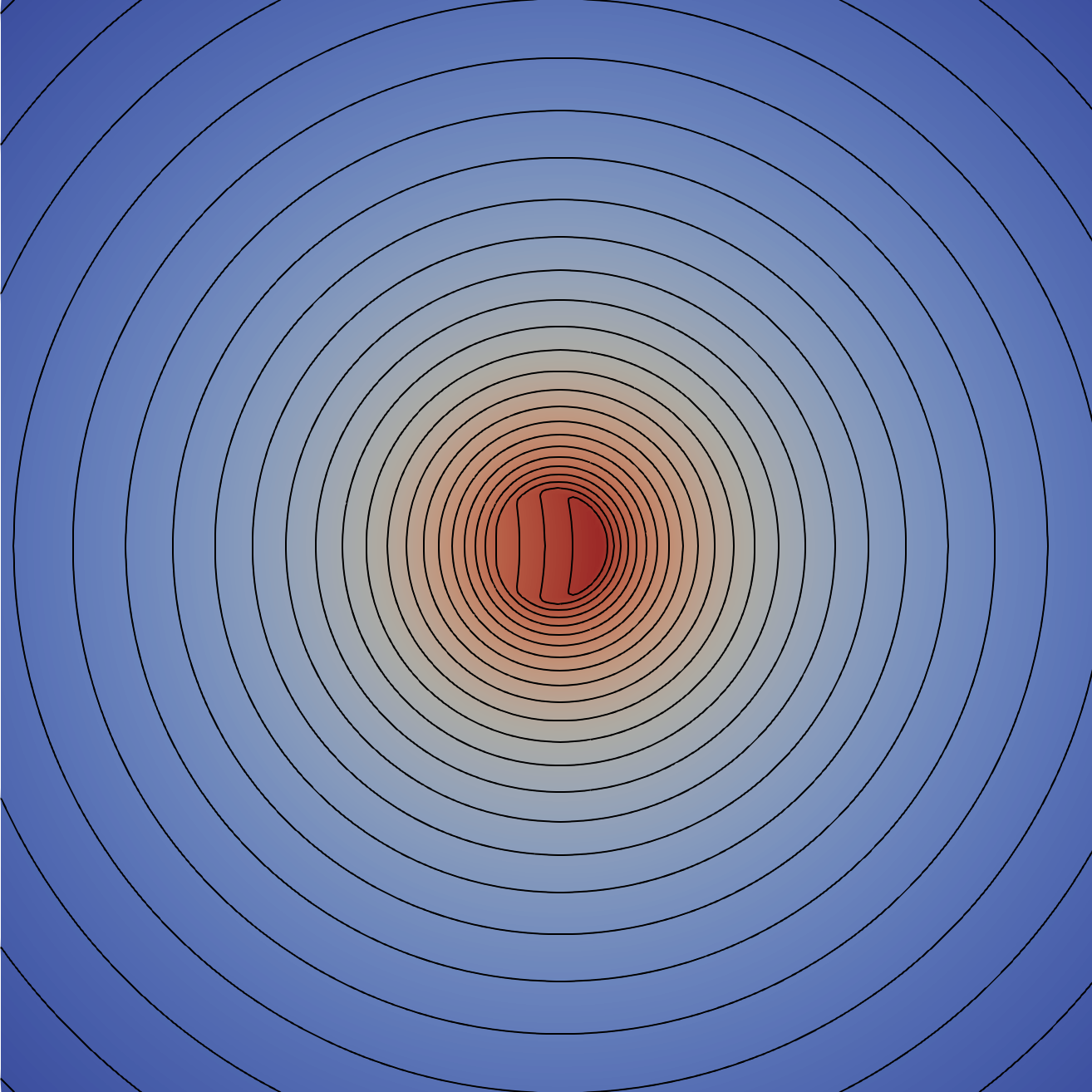}

    \vspace{.05\textwidth}

    \includegraphics[width=.3\textwidth]{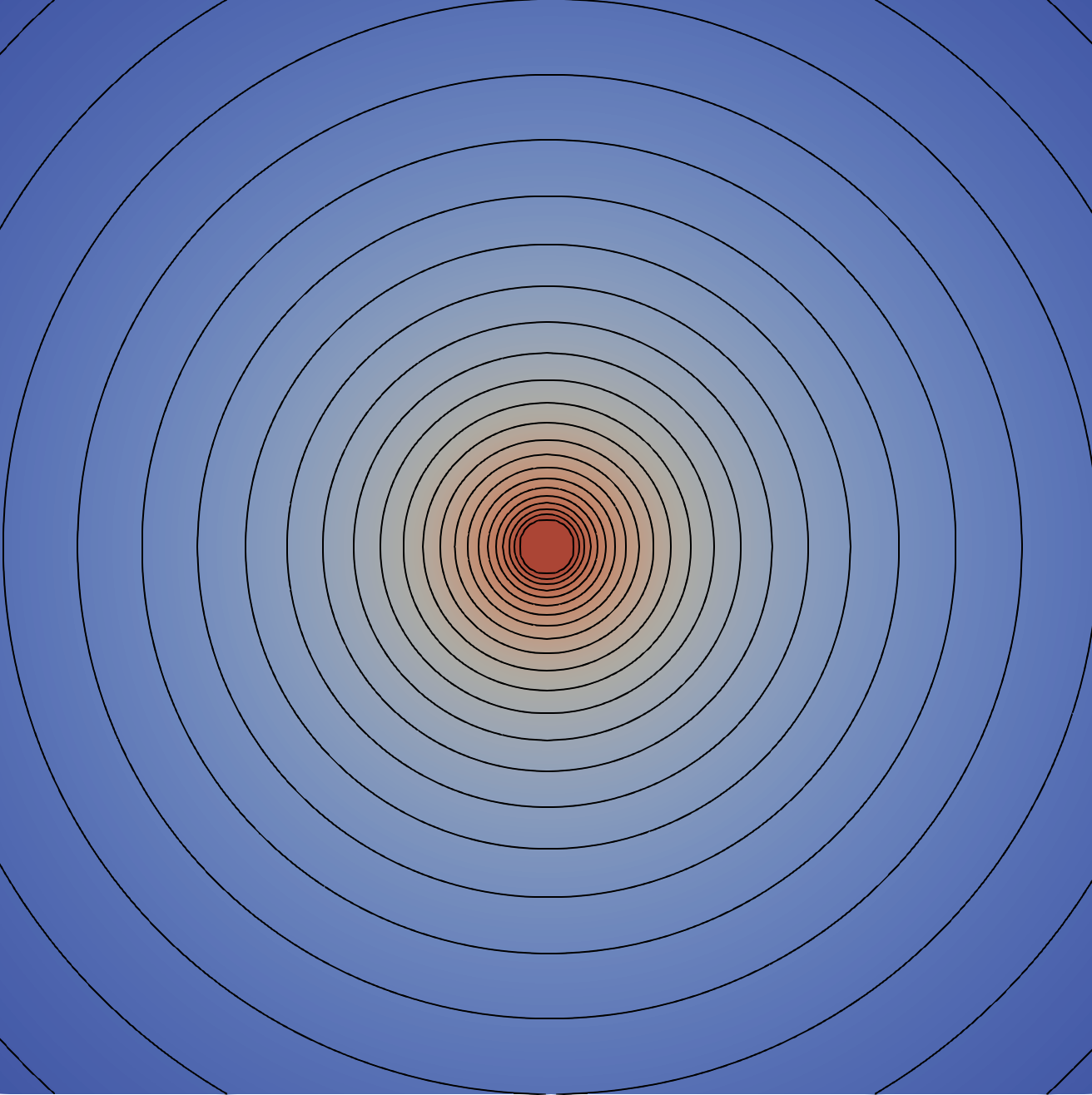}
    \hfill 
    \includegraphics[width=.3\textwidth]{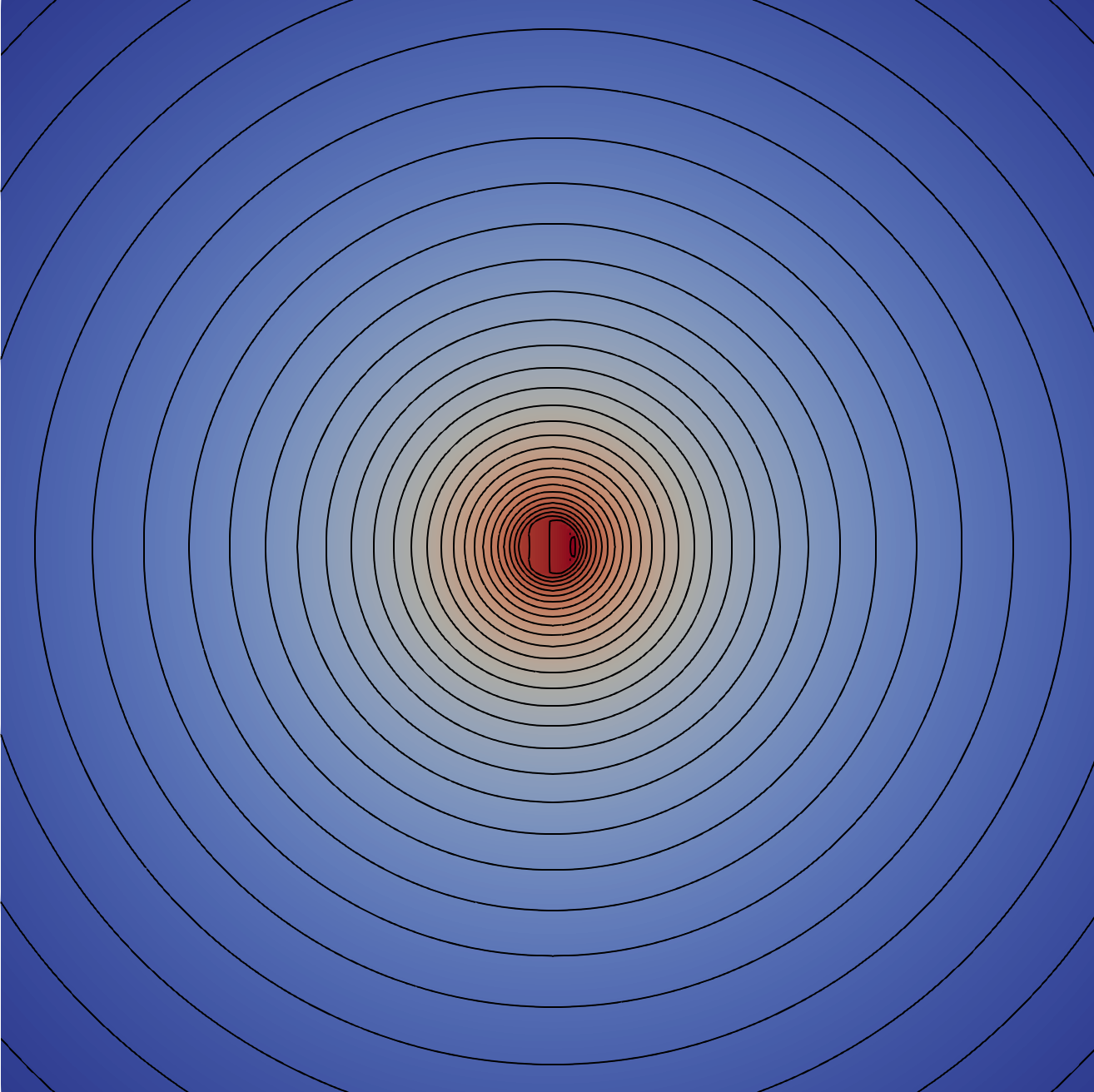}
    \hfill 
    \includegraphics[width=.3\textwidth]{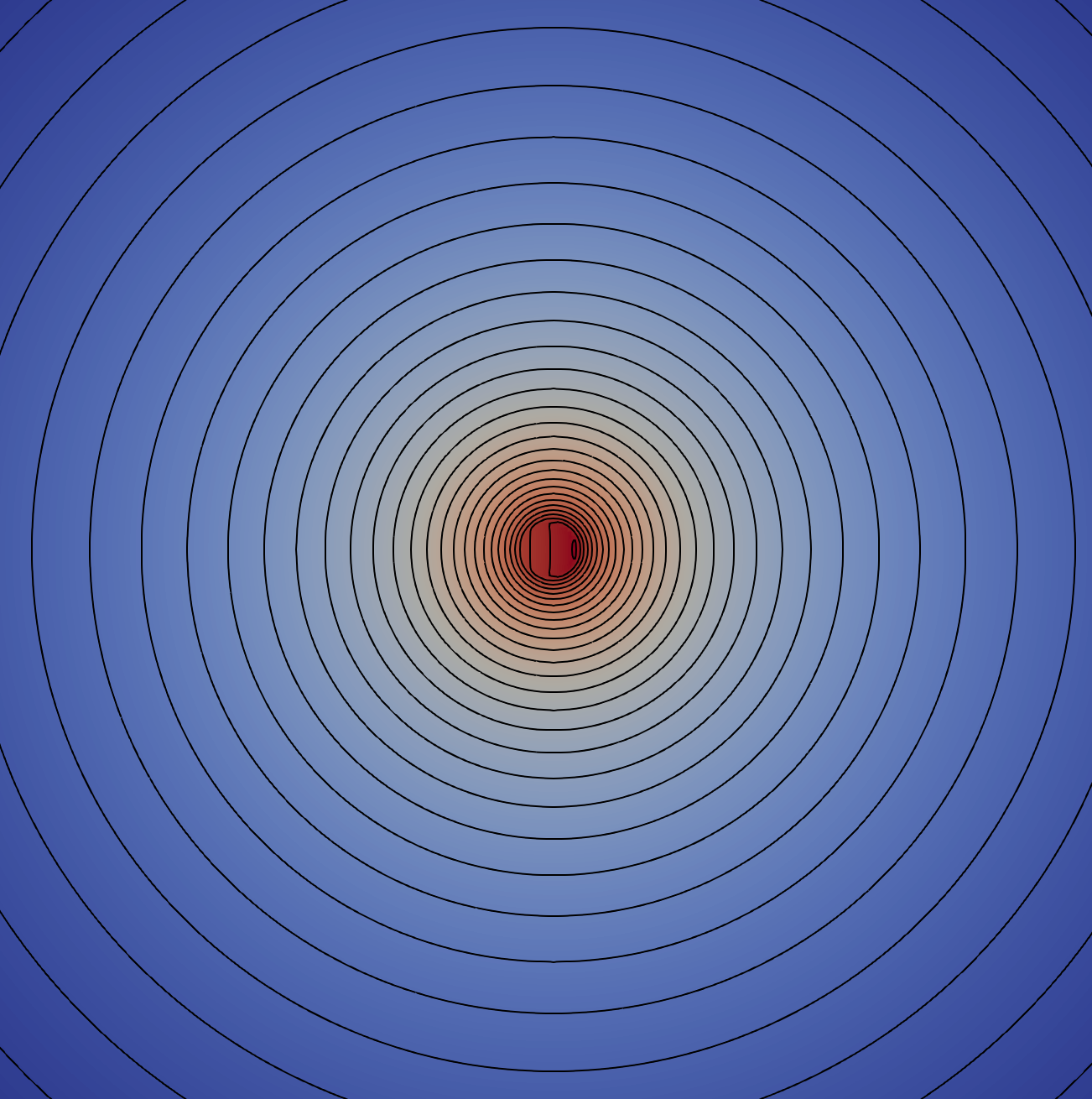}
    
    \caption{Contour plot of the numerical solution of problem~\eqref{eq:manufactured-solution-D3} for different values of the radius $r$ ($r=0.2$ top, $r=0.1$ center, and $r=0.05$ bottom) and different number of Fourier modes $N$ ($N=1$ left, $N=2$ center, and $N=3$ right).}
    \label{fig:fundamental_D3}
\end{figure}

In the top row of Figure~\ref{fig:fundamental_D3_error}, obtained for fixed $h=2/2^8$ and variable $\epsilon \in [5\cdot 10^{-2}, 0,1]$, the numerical discretization error (proportional to $h$) dominates over the dimensionality reduction error for any $n>0$. Only for $n=0$ we observe the expected linear decay of the whole error with $\epsilon$, in agreement with \eqref{eq:mainerrorestimate1}. Conversely, the error plot for $n=4$ is almost flat, confirming that the dimensionality reduction error is negligible in this case, if compared to the numerical approximation one.

The scenario of the bottom row of Figure~\ref{fig:fundamental_D3_error}, obtained using $h=2/2^{12}$, is more interesting. At least for the interval $\epsilon \in [0.1,0.2]$ we see that there is a clear decay of the whole error with $\epsilon$, confirming that in this regime the dimensionality reduction error is larger than the numerical approximation one. Interestingly, and in agreement with \eqref{eq:mainerrorestimate1}-\eqref{eq:mainerrorestimate2}, we see the effect of the parameter $n$ in the error decay rate. Precisely, the decay for $n=1$ is larger than the one observed with $n=0$.

The analysis at different levels of refinement (i.e. $h$-convergence) is shown in
Figure~\ref{fig:fundamental_D3_error_ref}. There, we highlight the transition between two main regimes. For values of $\epsilon \in [0.1,0.2]$ and $n=0,1$ the $h$-convergence of the scheme is heavily polluted by the dimensionality reduction error, as predicted by the theory. Conversely, for all radii and for $n \geq 3$ the effect of the dimensionality reduction error disappears, in fact we converge to the full-order solution (not shown here, but indistinguishable up to the sixth digit of accuracy from the solutions with $n=3$ and $n=4$). When the radius decreases, the number of modes that are necessary to achieve the same accuracy decreases -- as predicted -- and in particular we observe that all error curves tend to overlap as $\epsilon\to 0$ and all of them exhibit the optimal $h$-convergence rate of the scheme.

Overall, these tests show that the proposed approach offers full control on the dimensionality reduction error and numerical approximation error, allowing us to optimally balance these two components in the different scenarios where the inclusion is fully resolved or not resolved by the computational mesh.

\subsection{Three dimensional examples}
\label{sec:numerics-dirichlet-3d}

In this section we present some numerical results for the Dirichlet problem~\eqref{eq:inclusions-dirichlet-3d} in three dimensions. We start with a qualitative test that mimics the two inclusions problem presented in Figure~\ref{fig:two_inclusions}. We set $\Omega \equiv [-1,1]^3$, and choose $V$ composed of three non-aligned cylinders of varying radii $r$ and height $0.5$, as shown in Figure~\ref{fig:domain} for $r=0.2$: 
\begin{equation}
    \label{eq:inclusions-dirichlet-3d}
    \begin{aligned}
        -\Delta u =&  0 && \mathrm{in} \ \Omega\setminus \Gamma \equiv [-1,1]^3\setminus \Gamma,\\
        u =& 1 && \mathrm{on} \ \Gamma, \\
        u = &  0 && \mathrm{on} \ \partial \Omega. \\
    \end{aligned}
\end{equation}

In this case we do not have access to the exact solution, however we can provide a qualitative analysis of the numerical results by observing Figure~\ref{fig:results-3-inclusions}. 

\begin{figure}

    \includegraphics[width=.44\textwidth]{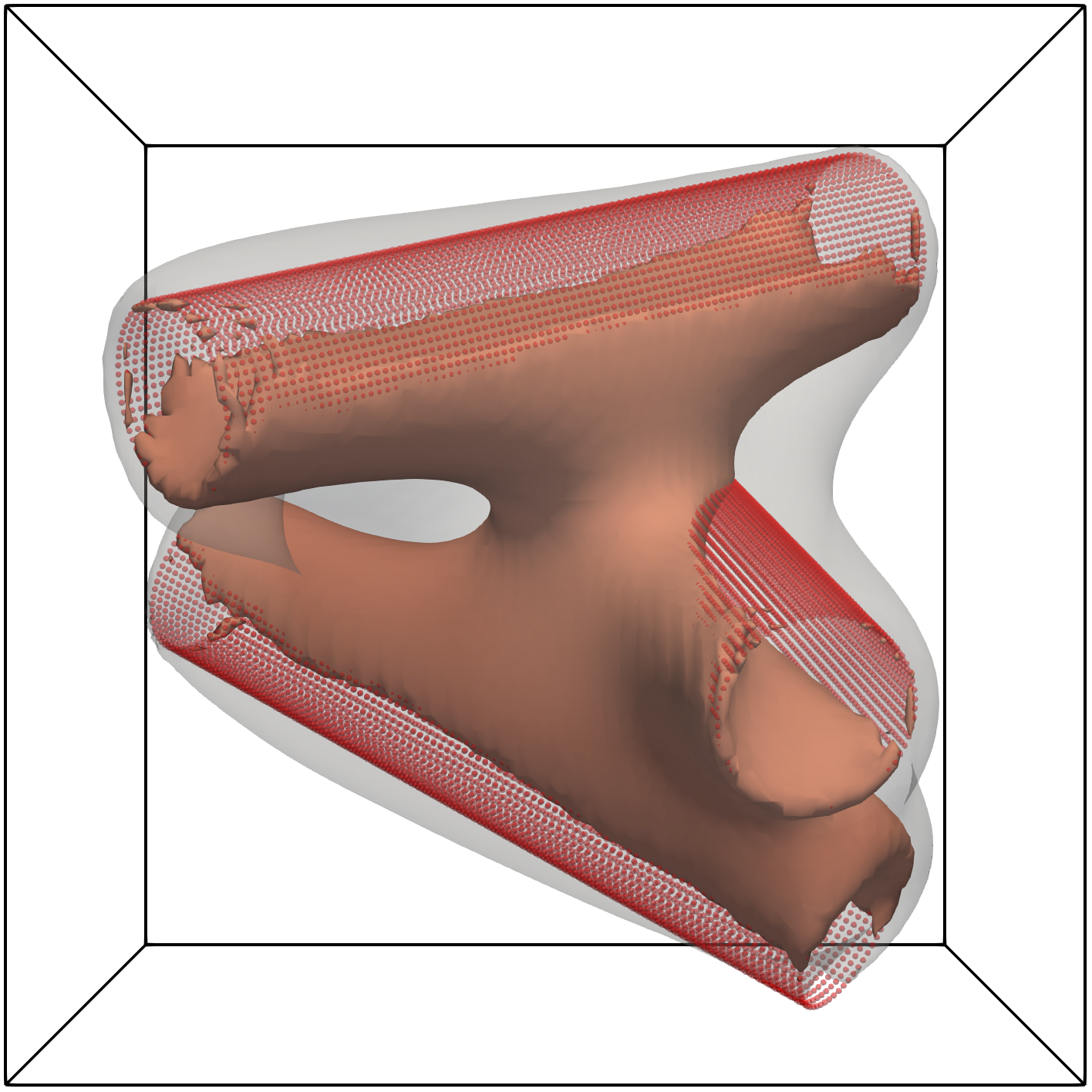}
    \hfill
    \includegraphics[width=.44\textwidth]{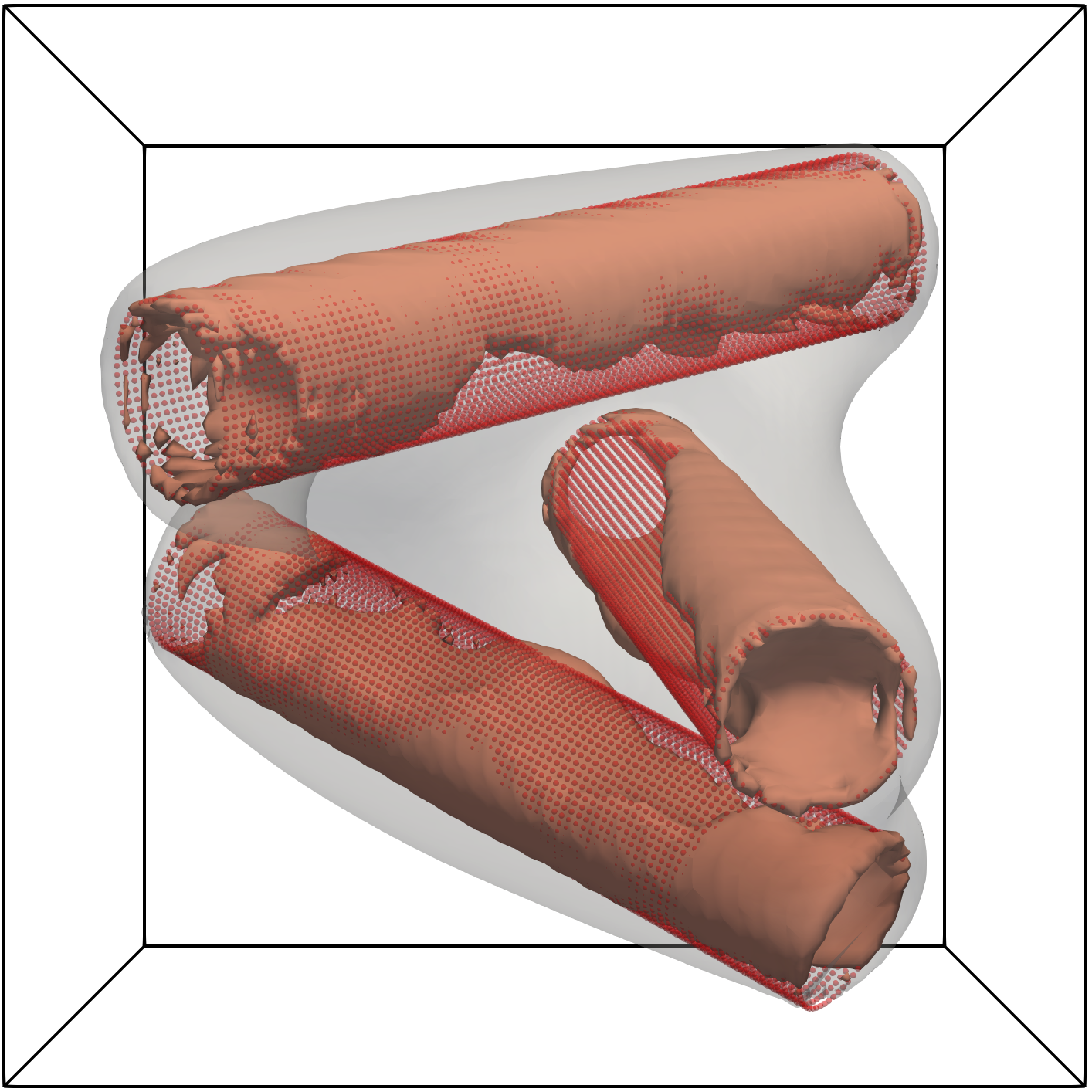}

    \includegraphics[width=.44\textwidth]{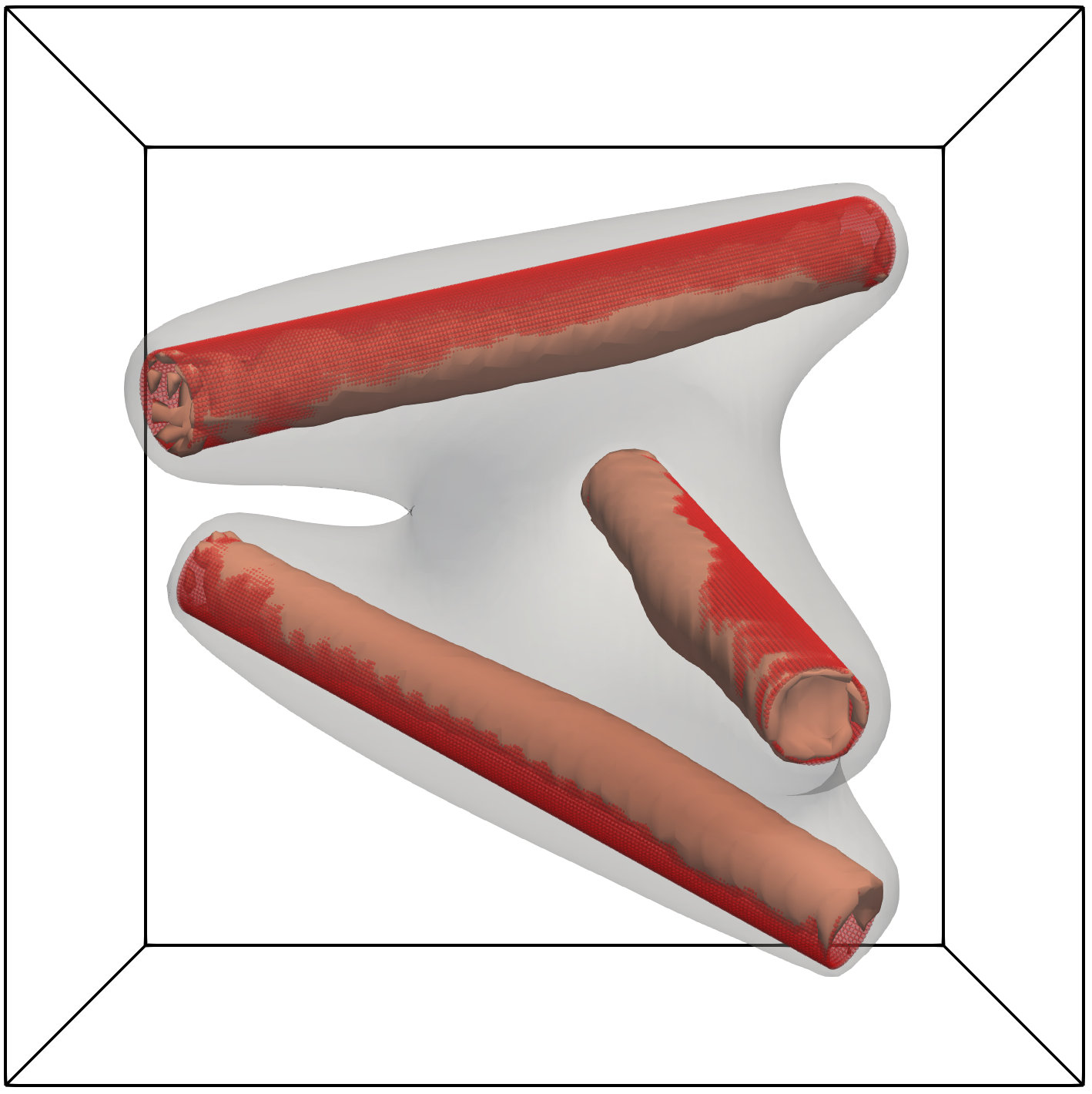}
    \hfill
    \includegraphics[width=.44\textwidth]{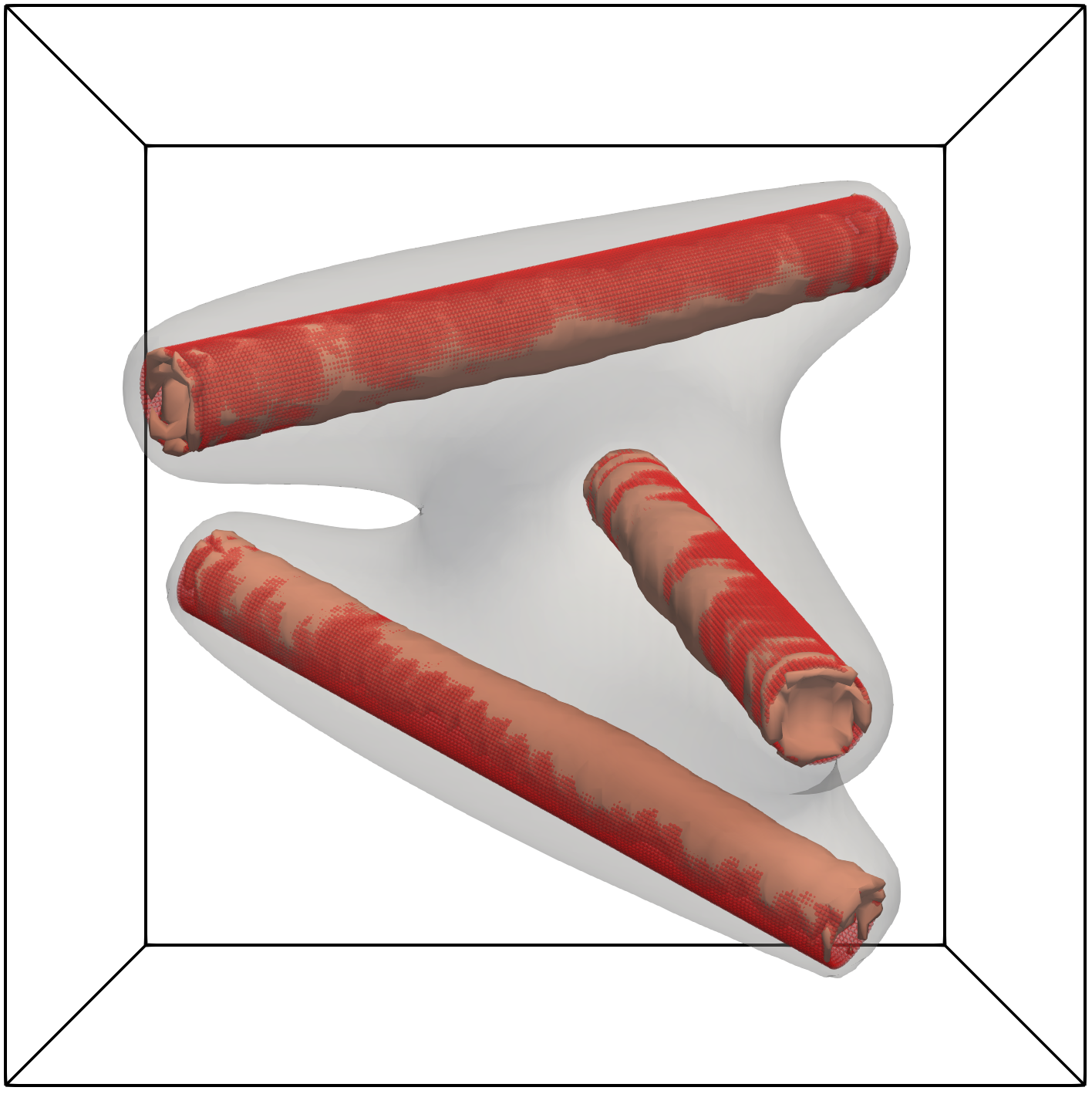}

    \includegraphics[width=.44\textwidth]{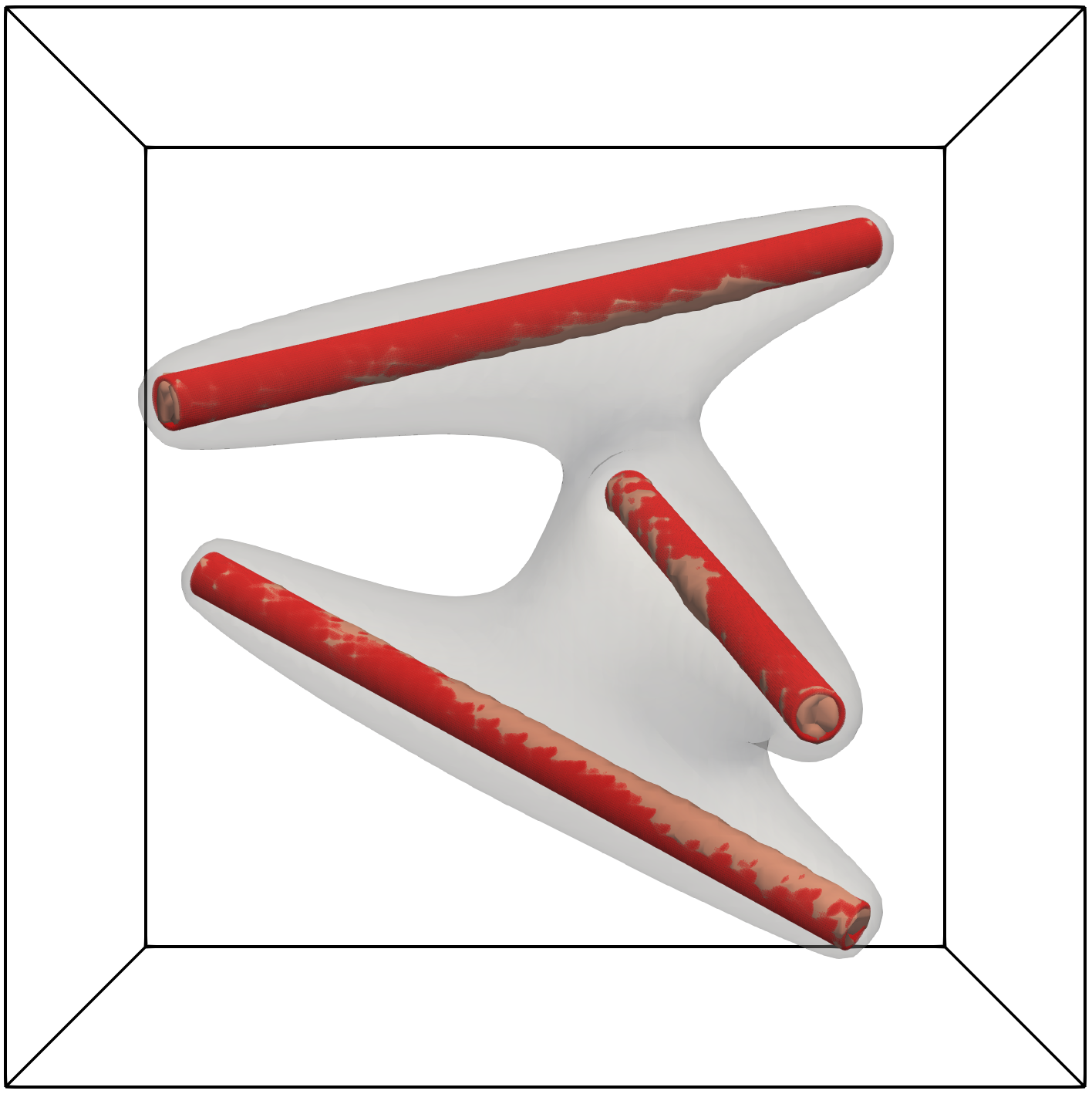}
    \hfill
    \includegraphics[width=.44\textwidth]{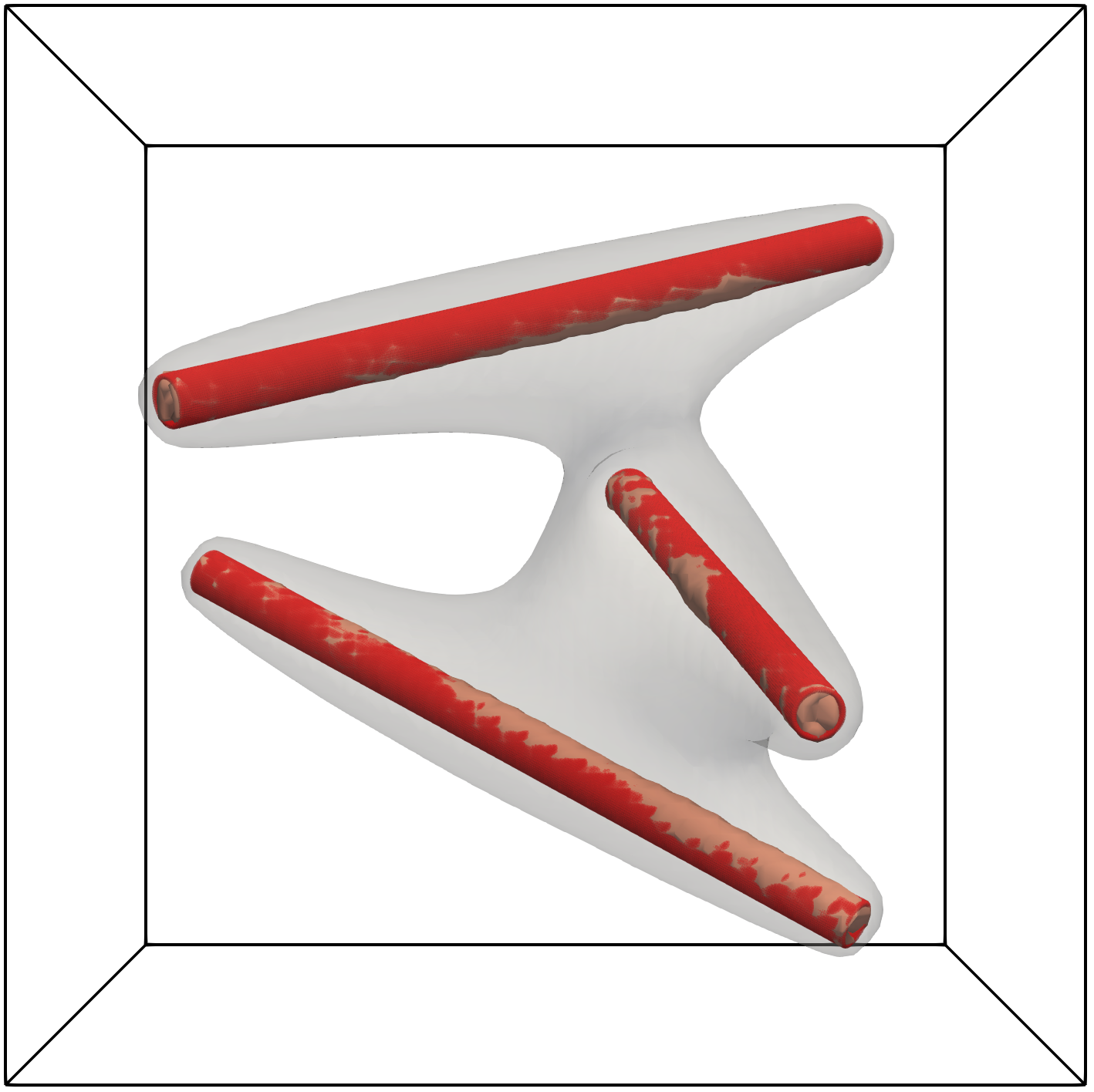}
    
    \caption{Numerical results for Problem~\eqref{eq:inclusions-dirichlet-3d} with one Fourier modes (left) and three Fourier modes (right), for inclusion radius $r=0.2$ (top), $r=0.1$ (center), and $r=0.05$ (bottom). The plots show iso-surfaces with values $u=1$ (red) and $u=0.5$ (light grey) of the solution.}
    \label{fig:results-3-inclusions}
\end{figure}

When the background mesh resolution is sensibly smaller than the radius of the
inclusions, a single Fourier mode is not enough to capture the behavior of the
solution around the inclusion, showing a large area of the solution inside the
domain $\Omega$ where the density is sensibly larger than one, violating the
maximum principle that would dictate a maximum value of the solution equal to
one. However, when the mesh resolution is comparable with the inclusion radius,
the solution is well approximated even with a single Fourier mode.

\begin{figure}
    \includegraphics[width=.45\textwidth]{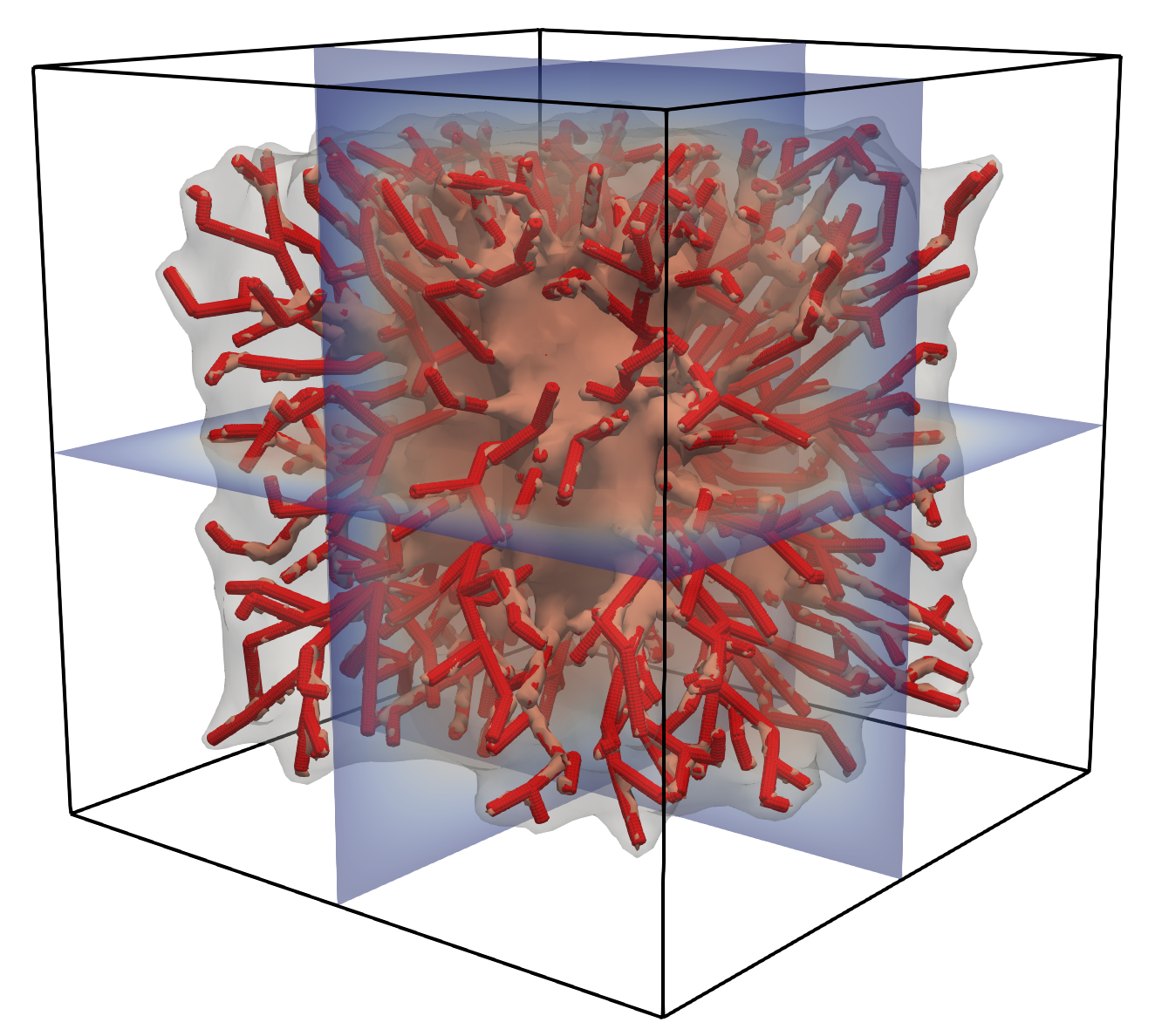}
    \hfill
    \includegraphics[width=.45\textwidth]{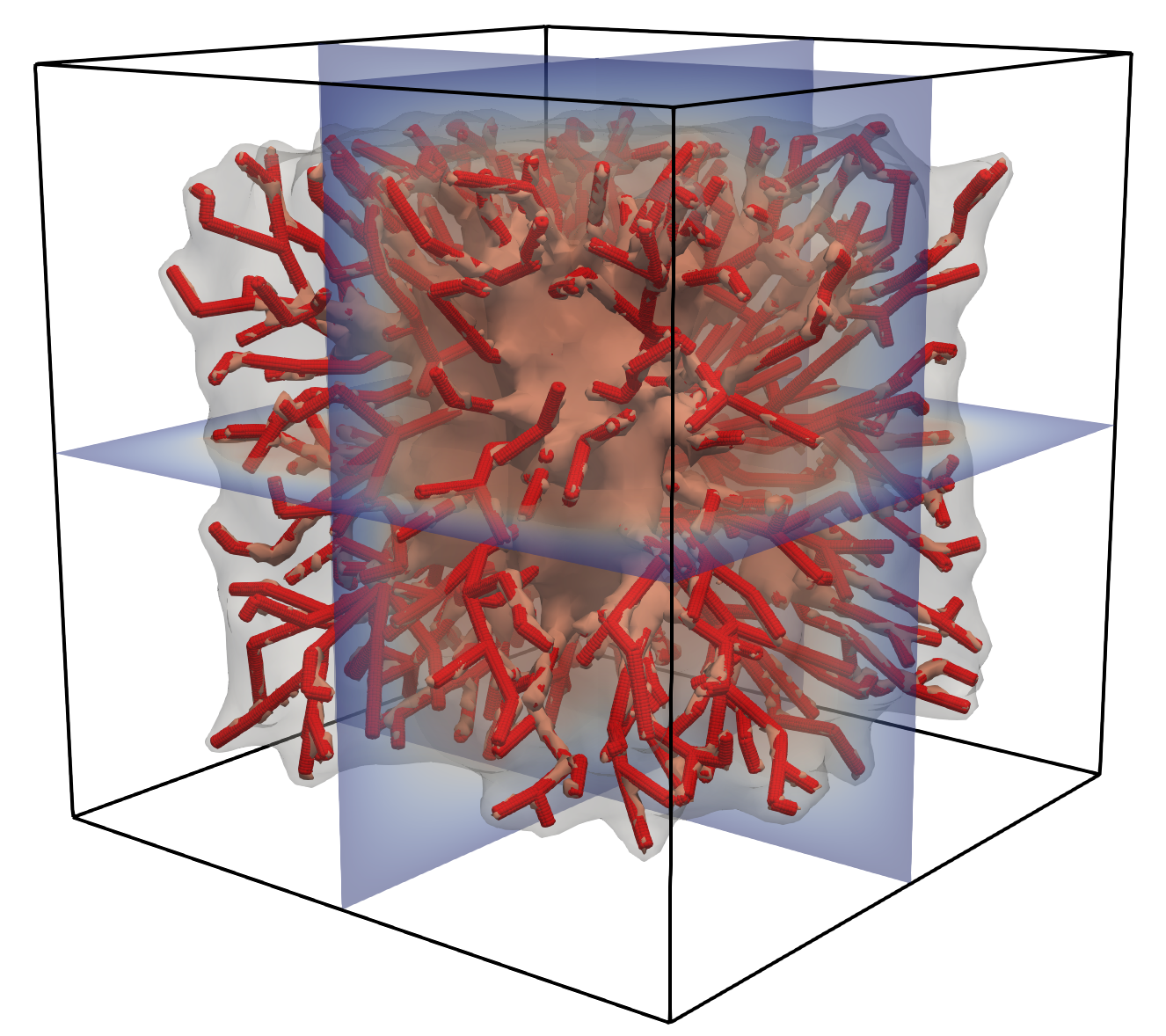}

    \caption{Numerical results for a complex vascular tree with radius $r=0.01$, with constant Dirichlet value equal to one on the boundary of the vessels, approximated with one Fourier mode (left) and three Fourier modes (right).}
    \label{fig:complex-tree}
\end{figure}

A further confirmation of this fact is presented in
Figure~\ref{fig:complex-tree}, showing the numerical solution of a complex
vascular tree with radius $r=0.01$, with constant Dirichlet value equal to one
on the boundary of the vessels, approximated with one Fourier mode (left) and
three Fourier modes (right). In this case, the radii of the vessels are
comparable with the grid size, the solution is well approximated even with a
single Fourier mode, and the dimensionality reduction error is in the same order of the finite
element approximation error.

\section{Conclusions}
We addressed a Lagrange multiplier method to couple mixed-dimensional problems, where the main difficulty is about the enforcement and approximation of boundary/interface constraints across dimensions.
We tackled this issue by means of a general approach, called the reduced Lagrange multiplier formulation, where a suitable restriction operator is applied to the classical Lagrange multiplier space. The mathematical properties of this formulation, precisely its well posedness, stability and corresponding error with respect to the original problem were thoroughly analyzed.

The fundamental ingredients for the reduced Lagrange multiplier formulation are the restriction and extension operators, discussed in Section \ref{sec:isomorphism} in the context of a general framework, and in Section \ref{sec:fourier} for the particular case of cylindrical inclusions embedded in three-dimensional domains (named 1D-3D coupling). This case is of particular interest for many applications, such as, fiber reinforced materials, microcirculation and perforated porous media.
For the specific case of 3D-1D mixed-dimensional problems we proposed a numerical discretization based on finite elements and we analyzed the stability and convergence of it. 

This work illustrates that the discrete scheme, including the dimensionality reduction and numerical approximation, is overall governed by three main parameters, $h,n,\epsilon$, the mesh characteristic size, the dimension of the Lagrange multiplier space and the size of the inclusion, respectively. The proposed approach offers full control on the different error sources, allowing us to optimally balance these components in the different scenarios where the inclusion is fully resolved or not resolved by the computational mesh.

\section{Acknowledgements}
The authors are members of Gruppo Nazionale per il Calcolo Scientifico (GNCS) of Istituto Nazionale di Alta Matematica (INdAM).

\bibliography{heltai-zunino-rev.bib}
\bibliographystyle{plain}

\end{document}